\documentclass[a4paper,10pt]{amsart}
\usepackage{amssymb}
\usepackage{bbm}
\usepackage{mathrsfs}
\usepackage[all]{xy}

\normalsize

\newlength{\aufzleft}
\newenvironment{aufz}{\begin{list}{}{\setlength{\listparindent}{0pt}\setlength{\itemsep}{\topsep}\setlength{\labelwidth}{3.2ex}\setlength{\aufzleft}{\labelsep}\addtolength{\aufzleft}{\labelwidth}\setlength{\leftmargin}{\aufzleft}}}{\end{list}}
\newenvironment{equi}{\begin{list}{}{\setlength{\listparindent}{0pt}\setlength{\itemsep}{\topsep}\setlength{\labelwidth}{4.1ex}\setlength{\aufzleft}{\labelsep}\addtolength{\aufzleft}{\labelwidth}\setlength{\leftmargin}{\aufzleft}}}{\end{list}}

\newtheoremstyle{bracket}{1ex}{2ex}{\rm}{}{\bfseries}{}{0.8em}{\thmnumber{(#2)}}
\newtheoremstyle{thm}{1ex}{2ex}{\itshape}{}{\bfseries}{}{0.9em}{\thmnumber{(#2)}\thmname{ #1}\thmnote{ (#3)}}
\newtheoremstyle{example}{1ex}{2ex}{\rm}{}{\bfseries}{}{0.8em}{\thmnumber{(#2)}\thmname{ #1}}
\newtheoremstyle{bracketsm}{1ex}{2ex}{\rm\small}{}{\bfseries}{}{0.5em}{\thmnumber{(#2)} }

\theoremstyle{bracket}
\newtheorem{no}{}[subsection]

\theoremstyle{bracketsm}
\newtheorem{nosm}[no]{}

\theoremstyle{thm}
\newtheorem{prop}[no]{Proposition}
\newtheorem{lemma}[no]{Lemma}
\newtheorem{cor}[no]{Corollary}
\newtheorem{thm}[no]{Theorem}

\theoremstyle{example}

\DeclareMathOperator{\rk}{rk}
\DeclareMathOperator{\cone}{cone}
\DeclareMathOperator{\face}{face}
\DeclareMathOperator{\pic}{Pic}
\DeclareMathOperator{\ke}{Ker}
\DeclareMathOperator{\cok}{Coker}
\DeclareMathOperator{\im}{Im}
\DeclareMathOperator{\degmon}{degm}
\newcommand{\R}{\mathbbm{R}}
\newcommand{\Z}{\mathbbm{Z}}
\newcommand{\N}{\mathbbm{N}}
\newcommand{\sig}{\Sigma}
\newcommand{\catmod}{{\sf Mod}}
\newcommand{\Q}{\mathbbm{Q}}
\newcommand{\hm}[3]{{\rm Hom}_{#1}(#2,#3)}
\newcommand{\alp}{\widehat{\alpha}}
\newcommand{\del}{\widehat{\delta}}
\newcommand{\fleq}{\preccurlyeq}
\newcommand{\flneq}{\prec}
\DeclareMathOperator{\card}{Card}
\newcommand{\Id}{{\rm Id}}
\newcommand{\eps}{\varepsilon}
\DeclareMathOperator{\conv}{conv}

\DeclareMathOperator{\diff}{Diff}
\newcommand{\mon}{{\sf Mon}}
\newcommand{\sch}{{\sf Sch}}
\DeclareMathOperator{\spec}{Spec}
\newcommand{\alg}{{\sf Alg}}
\newcommand{\ann}{{\sf Ann}}
\newcommand{\hmf}[2]{{\sf Hom}(#1,#2)}
\newcommand{\grhm}[4]{{}^{#1}{\rm Hom}_{#2}(#3,#4)}
\newcommand{\grann}{{\sf GrAnn}}
\newcommand{\gralg}{{\sf GrAlg}}
\newcommand{\res}{\!\upharpoonright}
\newcommand{\T}{\mathbbm{T}}
\newcommand{\zs}{\widehat{Z}_{\sigma}}
\newcommand{\zt}{\widehat{Z}_{\tau}}
\newcommand{\grmod}{{\sf GrMod}}
\newcommand{\qcmod}{{\sf QcMod}}
\newcommand{\sat}{{\rm sat}}
\renewcommand{\o}{\mathscr{O}}
\renewcommand{\S}{\mathscr{S}}
\newcommand{\grgam}[2]{{}^{#1}\Gamma_{#2}}
\newcommand{\A}{\mathscr{A}}

\newcommand{\F}{\mathscr{F}}

\newcommand{\pgrmod}{{\sf PGrMod}}
\newcommand{\etaq}{\overline{\eta}}
\newcommand{\dfgl}{\mathrel{\mathop:}=}
\newcommand{\G}{\mathscr{G}}
\newcommand{\prmod}{{\sf PMod}}
\newcommand{\J}{\mathbbm{J}}
\DeclareMathOperator{\Sat}{Sat}
\newcommand{\ilim}{\varinjlim}
\DeclareMathOperator{\proj}{Proj}
\newcommand{\gride}[2]{{}^{#1}D_{#2}}
\newcommand{\gridt}[3]{{}^{#1}D^{#2}_{#3}}
\newcommand{\grloc}[3]{{}^{#1}H^{#2}_{#3}}
\newcommand{\C}{{\sf C}}
\newcommand{\D}{{\sf D}}

\newcommand{\ia}{\mathfrak{a}}
\renewcommand{\L}{\mathbbm{L}}
\newcommand{\grext}[5]{{}^{#1}{\rm Ext}^{#2}_{#3}(#4,#5)}
\newcommand{\U}{\mathbbm{U}}
\newcommand{\Cc}{\mathbbm{C}}

\newdir{ (}{!/5pt/@{ }*@^{(}}
\newdir{ >}{{}*!/-10pt/\dir{>}}

\newcommand{\sq}{\hskip2pt\raisebox{.225ex}{\rule{.8ex}{.8ex}\hskip2pt}}

\newcommand{\snf}{\renewcommand{\thefootnote}{*}\footnotetext{The author was supported by the Swiss National Science Foundation.}}

\begin{document}

\title{Quasicoherent sheaves on toric schemes\protect\snf}
\author{Fred Rohrer}
\address{Universit\"at T\"ubingen, Fachbereich Mathematik, Auf der Morgenstelle 10, 72076 T\"u\-bingen, Germany}
\email{rohrer@mail.mathematik.uni-tuebingen.de}

\begin{abstract}
Let $X$ be the toric scheme over a ring $R$ associated with a fan $\sig$. It is shown that there are a group $B$, a $B$-graded $R$-algebra $S$ and a graded ideal $I\subseteq S$ such that there is an essentially surjective, exact functor $\widetilde{\bullet}$ from the category of $B$-graded $S$-modules to the category of quasicoherent $\o_X$-modules that vanishes on $I$-torsion modules and that induces for every $B$-graded $S$-module $F$ a surjection $\Xi_F$ from the set of $I$-saturated graded sub-$S$-modules of $F$ onto the set of quasicoherent sub-$\o_X$-modules of $\widetilde{F}$. If $\sig$ is simplicial, the above data can be chosen such that $\widetilde{\bullet}$ vanishes precisely on $I$-torsion modules and that $\Xi_F$ is bijective for every $F$. In case $R$ is noetherian, a toric version of the Serre-Grothendieck correspondence is proven, relating sheaf cohomology on $X$ with $B$-graded local cohomology with support in $I$.
\end{abstract}

\maketitle


\section*{Introduction}

Let $R$ be a commutative ring, let $S$ be a positively $\Z$-graded $R$-algebra that is generated by finitely many elements of degree $1$, and let $S_+=\bigoplus_{n>0}S_n$ be its irrelevant ideal. We consider the projective $R$-scheme $X=\proj(S)\rightarrow\spec(R)$. It is a fundamental and classical result that there is an essentially surjective, exact functor $\widetilde{\bullet}$ from the category of $\Z$-graded $S$-modules to the category of quasicoherent $\o_X$-modules that vanishes precisely on $S_+$-torsion modules, and that induces for every $\Z$-graded $S$-module $F$ a bijection between the set of $S_+$-saturated graded sub-$S$-modules of $F$ and the set of quasicoherent sub-$\o_X$-modules of $\widetilde{F}$ (\cite[II.2]{ega}). This allows a reasonable translation between projective geometry and $\Z$-graded commutative algebra. Even more, the well-known Serre-Grothendieck correspondence extends this in some sense to yield a translation between sheaf cohomology on $X$ and $\Z$-graded local cohomology with support in $S_+$, provided $R$ is noetherian (\cite[20.3]{bs}). The aim of this article is to present analogous results for toric schemes (which are not necessarily projective), a natural generalisation of toric varieties.

\smallskip

Recall that a \textit{toric variety} (over the field $\Cc$ of complex numbers) is a normal, irreducible, separated $\Cc$-scheme of finite type containing an open torus whose multiplication extends to an algebraic action on the whole variety. If $V$ is an $\R$-vector space of finite dimension, $N$ is a $\Z$-structure on $V$, $M$ is its dual $\Z$-structure on $V^*$, and $\sig$ is an $N$-rational fan in $V$, then these data define a toric variety as follows: If $\sigma$ is a cone in $\sig$, then the set $\sigma^{\vee}\cap M$ of $M$-rational points of its dual is a submonoid of finite type of $M$. Its algebra $\Cc[\sigma^{\vee}\cap M]$ corresponds to a variety $X_{\sigma}(\Cc)$, and the facial relations between the cones in $\sig$ allow to glue these varieties and obtain a toric variety $X_{\sig}(\Cc)$. In fact, one can show that \textit{every} toric variety can be constructed in this way. Besides tori, the class of toric varieties contains affine spaces, projective spaces, weighted projective spaces, or Hirzebruch surfaces, but also singular varieties (e.g., the surface in $\Cc^3$ defined by the equation $x_1x_2=x_3^2$, or the hypersurface in $\Cc^4$ defined by the equation $x_1x_2=x_3x_4$), and non-projective proper varieties (cf.~\cite[Appendice]{demazure}); we refer the reader to \cite{ewald} or \cite{fulton} for a lot more examples of toric varieties.

The above construction of a toric variety from a fan can be performed over an arbitrary ring $R$ instead of $\Cc$, yielding what we call a \textit{toric scheme.} This construction is compatible with base change, and in this sense every toric scheme is obtained by base change from a ``universal'' toric scheme over $\Z$; this is essentially the viewpoint taken by Demazure in \cite{demazure}, where toric varieties make their first appearance in the literature. Even when one is only interested in toric varieties over $\Cc$, performing base change and hence considering toric schemes instead of toric varieties seems necessary in order to attack fundamental geometric questions (and is, from a scheme-theoretical point of view, anyway a very natural thing to do). For example, in order to investigate whether the Hilbert scheme of a toric variety $X_{\sig}(\Cc)$ exists one has to study quasicoherent sheaves on the base change $X_{\sig}(R)$ for \textit{every} $\Cc$-algebra $R$, and similarly for related representability questions. Somewhat less fancy, using base change and hence toric schemes one can reduce the study of toric varieties to the study of toric schemes defined by fans that are not contained in a hyperplane of their ambient space.

Most of the toric literature treats varieties over $\Cc$ (or algebraically closed fields), and only rarely over arbitrary fields. In arithmetic geometry, toric schemes over more general bases (mostly valuation rings) appear more often (cf.~\cite{kkms}, \cite{maillot}, \cite{gil}, \cite{ling}, \cite{gubler}). The author's article \cite{geometry} contains a systematic study of basic properties of toric schemes over arbitrary bases.

\smallskip

Now, let $X=X_{\sig}(R)$ be the toric scheme over $R$ associated with a rational fan $\sig$; for simplicity, we suppose that $\sig$ is not contained in a hyperplane of its ambient space. If $R=\Cc$ there is a partial analogue to the classical result recalled in the first paragraph, due to Cox and Musta\c{t}\v{a} (\cite{cox}, \cite{mustata}). Namely, there are a commutative group $B$, a $B$-graded $\Cc$-algebra $S$, and a graded ideal $I\subseteq S$, all defined in terms of $\sig$, such that there is an essentially surjective, exact functor $\widetilde{\bullet}$ from the category of $B$-graded $S$-modules to the category of quasicoherent $\o_X$-modules (\cite[1.1]{mustata}). If $\sig$ is simplicial we have more: the above data can be chosen such that $\widetilde{\bullet}$ vanishes precisely on those $B$-graded $S$-modules of finite type that are $I$-torsion modules (\cite[3.11]{cox}), and that it induces a bijection between the set of $I$-saturated graded ideals of $S$ and the set of quasicoherent ideals of $\o_X=\widetilde{S}$ (\cite[3.12]{cox}). Finally, there is also a correspondence between sheaf cohomology on $X$ and $B$-graded local cohomology with support in $I$ (\cite[1.3]{mustata}).

Comparing the projective and the toric situation they seem to be so similar as to provoke the question whether the toric results hold over arbitrary base rings and can moreover -- at least in the simplicial case -- be extended to yield bijections between $I$-saturated graded sub-$S$-modules of a graded $S$-module $F$ and quasicoherent sub-$\o_X$-modules of $\widetilde{F}$. Unfortunately, Cox's and Musta\c{t}\v{a}'s proofs use the language of Weil divisors, which is not appropriate over arbitrary base rings. But since the corresponding proofs in the projective case are elementary it is tempting to look for similar proofs in the toric case. In this article we will see that both these hopes -- general results and elementary proofs -- can indeed be fulfilled. Namely, we show that there are a commutative group $B$, a $B$-graded $R$-algebra $S$, and a graded ideal $I\subseteq S$, all defined in terms of $\sig$, such that there is an essentially surjective, exact functor $\widetilde{\bullet}$ from the category of $B$-graded $S$-modules to the category of quasicoherent $\o_X$-modules that vanishes on $I$-torsion modules and that induces for every $B$-graded $S$-module $F$ a surjection $\Xi_F$ from the set of $I$-saturated graded sub-$S$-modules of $F$ onto the set of quasicoherent sub-$\o_X$-modules of $\widetilde{F}$. If $\sig$ is simplicial we have more: the above data can be chosen such that $\widetilde{\bullet}$ vanishes precisely on $I$-torsion modules, that $\Xi_F$ is bijective for every $F$, and that $\widetilde{\bullet}$ preserves the properties of being of finite type, of finite presentation, pseudocoherent, or coherent. Finally, for arbitrary $\sig$, but supposing $R$ to be noetherian, there is a toric version of the Serre-Grothendieck correspondence relating sheaf cohomology on $X$ with $B$-graded local cohomology with support in $I$.

\smallskip

While these results do not seem to be available in the literature, they allegedly can be derived from the general theory of quotients under actions of group schemes. However, our goal is to provide an \textit{elementary and rigorous exposition} of this theory that is accessible without much prerequisites. Therefore, we assume only basic knowledge in algebra (\cite{a}, \cite{ac}), algebraic geometry (\cite[I]{ega}), and -- only in Sections \ref{sub2.3} and \ref{sub4.5} -- homological algebra (\cite{bs}, \cite{tohoku}). Some basic facts from polyhedral geometry are recalled in Section \ref{sub1.1}; for proofs we refer the reader to \cite{bg}. For lack of suitable reference we provide in Section \ref{sec2} the necessary preliminaries on graded commutative algebra.

\smallskip

The structure of the article is as follows: In Section \ref{sec1} we treat the combinatorics of fans and especially the notion of Picard group of a fan. Section \ref{sec2} provides basics on graded commutative algebra, including graded local cohomology. In Section \ref{sec3} we introduce Cox schemes and toric schemes and study their relation. Finally, Section \ref{sec4} contains the main results: the correspondence between graded and quasicoherent modules, and the toric Serre-Grothendieck correspondence. Each section starts with a more detailed overview of its content.

\medskip

\noindent\textit{Notations and conventions:}
In general we use the terminology of Bourbaki's \textit{\'El\'ements de math\'ematique} and Grothendieck's \textit{\'El\'ements de g\'eom\'etrie alg\'ebrique.} Concerning $\delta$-functors we use the terminology of \cite{tohoku}. Monoids and groups are understood to be additively written and commutative, rings are understood to be commutative, and algebras are understood to be commutative, unital and associative. For a ring $R$, an $R$-module $M$ and a subset $E\subseteq M$ we denote by $\langle E\rangle_R$ the sub-$R$-module of $M$ generated by $E$; if $M$ is free then we denote by $\rk_R(M)$ its rank. For a monoid $M$ and a subset $E\subseteq M$ we denote by $\langle E\rangle_{\N}$ the submonoid of $M$ generated by $E$. For a category $\C$ and an object $S$ of $\C$ we denote by $\C_{/S}$ and $\C^{/S}$ the categories of objects of $\C$ over and under $S$, respectively. For a further category $\D$ we denote by $\hmf{\C}{\D}$ the category of functors from $\C$ to $\D$. We denote by $\mon$, $\ann$ and $\sch$ the categories of monoids, rings and schemes, respectively. For a ring $R$ we denote by $\alg(R)$ and $\catmod(R)$ the categories of $R$-algebras and $R$-modules, respectively, and we write $\sch_{/R}$ instead of $\sch_{/\spec(R)}$.


\section{The combinatorics of fans}\label{sec1}

\noindent\textit{Throughout this section let $R\subseteq\R$ be a principal subring, let $K$ denote its field of fractions considered as a subfield of $\R$, and let $V$ be an $\R$-vector space with $n\dfgl\rk_{\R}(V)\in\N$, canonically identified with its bidual $V^{**}$ and considered as an $R$-module by means of scalar restriction.}\smallskip

In \ref{sub1.1} we recall terminology and basic facts about rational structures, polycones, and fans; for proofs we refer the reader to \cite{bg}. We consider rationality with respect to an arbitrary principal subring $R\subseteq\R$, so that we can treat the two cases in which we are mainly interested, $R=\Z$ and $R=\Q$, at once. The key notion of this section, introduced in \ref{sub1.2}, is the Picard group $\pic(\sig)$ of a fan $\sig$. The terminology is of course inspired by toric geometry, for one can show that $\pic(\sig)\cong\pic(X_{\sig}(\Cc))$ (\cite[VII.2.15]{ewald}). In \ref{sub1.3} we relate $\pic(\sig)$ to other groups encoding combinatorial properties of $\sig$, and we embed it in particular into a certain group $A_{\sig_1}$ depending only on the set of $1$-dimensional polycones in $\sig$ (and isomorphic to the divisor class group of $X_{\sig}(\Cc)$ (\cite[3.4]{fulton})). In \ref{sub1.4} we show that $\sig$ is simplicial or regular if and only if $A_{\sig_1}/\pic(\sig)$ is finite or zero, respectively. Our proof is purely combinatorial and moreover reveals combinatorial conditions for simpliciality or regularity of single polycones in the fan.

\subsection{Rational structures, polycones, and fans}\label{sub1.1}

\begin{no}\label{1.10}
An \textit{$R$-structure on $V$} is a sub-$R$-module $W\subseteq V$ such that the canonical morphism of $\R$-vector spaces $\R\otimes_RW\rightarrow V$ with $a\otimes x\mapsto ax$ is an isomorphism, or -- equivalently -- that $\langle W\rangle_{\R}=V$ and\footnote{Note that $W$ is free since $R$ is principal (\cite[VII.3 Th\'eor\`eme 1 Corollaire 2]{a}).} $\rk_R(W)=n$; if $R$ is a field this coincides with the notion of $R$-structure in \cite[II.8.1]{a}. For a sub-$R$-module $W\subseteq V$ we denote the image of the canonical monomorphism of $K$-vector spaces $K\otimes_RW\rightarrow V$ with $a\otimes x\mapsto ax$ by $W_K$. So, $W$ is an $R$-structure on $V$ if and only if $W_K$ is a $K$-structure on $V$, and then $W$ is a lattice of $W_K$ in the sense of \cite[VII.4.1]{ac}.

Let $W$ be an $R$-structure on $V$. A sub-$\R$-vector space $V'\subseteq V$ is called \textit{$W$-rational} if it has a basis contained in $W$; then, $W\cap V'$ is an $R$-structure on $V'$ (\cite[II.8.2]{a}). A closed halfspace of $V$ defined by a $W$-rational linear hyperplane in $V$ is called a \textit{closed $W$-halfspace (of $V$)}. The dual $W^*$ of $W$ can and will be canonically identified with an $R$-structure on the dual $V^*$ of $V$; then, the identification of $V$ with $V^{**}$ identifies $W$ with its bidual (\cite[II.8.4]{a}, \cite[VII.4.1 Proposition 3]{ac}).
\end{no}

\begin{no}\label{1.25}
Let $W$ be an $R$-structure on $V$. For $A\subseteq V$ we denote by $A^{\perp}$ the orthogonal $\{u\in V^*\mid u(A)\subseteq 0\}$ and by $A^{\vee}$ the dual $\{u\in V^*\mid u(A)\subseteq\R_{\geq 0}\}$ of $A$, and we set $A^{\vee}_{W^*}\dfgl A^{\vee}\cap W^*$ and $A^{\perp}_{W^*}\dfgl A^{\perp}\cap W^*$; in case $A=\{x\}$ we write $x^{\perp}$ and $x^{\vee}$. We denote by $\conv(A)$ the convex hull of $A$ and by $\cone(A)$ the conic hull of $A$, i.e., the set of $\R$-linear combinations of $A$ with coefficients in $\R_{\geq 0}$.

A subset $\sigma\subseteq V$ is called a \textit{$W$-polycone (in $V$)} if it is the intersection of finitely many closed $W$-halfspaces in $V$, or -- equivalently -- if there is a finite subset $A\subseteq W$ with $\sigma=\cone(A)$. A subset $\sigma\subseteq V$ is a $W$-polycone if and only if it is a $W_K$-polycone; then, its dual $\sigma^{\vee}$ is a $W^*$-polycone and $\sigma^{\vee\vee}=\sigma$.

Let $\sigma$ be a $W$-polycone. A subset $A\subseteq W$ with $\sigma=\cone(A)$ is called a \textit{$W$-generating set of $\sigma$.} The $W$-polycone $\sigma$ is called \textit{simplicial} if it has a $V$-generating set that is free over $\R$, or -- equivalently -- a $W$-generating set that is free over $R$; it is called \textit{$W$-regular} if it has a $W$-generating set contained in an $R$-basis of $W$. A $W$-regular $W$-polycone is simplicial and a simplicial $W$-polycone is $W_K$-regular, so the properties of $\sigma$ to be $W_K$-regular, simplicial or $V$-regular are equivalent. The $W$-polycone $\sigma$ is called \textit{sharp} if it does not contain a line, and \textit{full (in $V$)} if $\langle\sigma\rangle_{\R}=V$; moreover, $\dim(\sigma)\dfgl\rk_{\R}(\langle\sigma\rangle_{\R})$ is called \textit{the dimension of $\sigma$.} A \textit{face of $\sigma$} is a subset $\tau\subseteq\sigma$ such that there exists $u\in\sigma^{\vee}$ (or -- equivalently -- $u\in\sigma^{\vee}_{W^*}$) with $\tau=\sigma\cap u^{\perp}$. The set $\face(\sigma)$ of faces of $\sigma$ consists of finitely many $W$-polycones. The relation ``$\tau$ is a face of $\sigma$'' is an ordering on the set of $W$-polycones, denoted by $\tau\fleq\sigma$.
\end{no}

\begin{no}\label{1.35}
Let $W$ be an $R$-structure on $V$ and let $\rho$ be a $1$-dimensional sharp $W$-polycone. The relation $x\in\conv(\{0,y\})$ is a total ordering on $\rho$. A $\subseteq$-minimal $W$-generating set of $\rho$ has a unique element. If this element is minimal in $W\cap\rho\setminus 0$ with respect to the ordering induced from the above ordering on $\rho$ then it is uniquely determined by $\rho$, called \textit{the $W$-minimal $W$-generator of $\rho$,} and denoted by $\rho_W$. In general, $W$-minimal $W$-generators do not exist, but if $R=\Z$ they do.\footnote{In fact, $\rho$ has a $W$-minimal $W$-generator if and only if $R=\Z$.}
\end{no}

\begin{no}\label{1.40}
Let $W$ be an $R$-structure on $V$. A \textit{$W$-fan (in $V$)} is a finite set $\sig$ of sharp $W$-polycones such that $\sigma\cap\tau\in\face(\sigma)\subseteq\sig$ for all $\sigma,\tau\in\sig$. If $\sig$ is a $W$-fan then the sub-$\R$-vector space $\langle\sig\rangle\dfgl\langle\bigcup\sig\rangle_{\R}\subseteq V$ is $W$-rational and $\cone(\sig)\dfgl\cone(\bigcup\sig)$ is a $W$-polycone. We call $\dim(\sig)\dfgl\rk_{\R}(\langle\sig\rangle)$ \textit{the dimension of $\sig$.} By means of the relation $\tau\fleq\sigma$ we consider $\sig$ as an ordered set, and then it is a lower semilattice\footnote{i.e., every nonempty finite subset has an infimum}. We denote for $k\in\Z$ by $\sig_k$ the set of $k$-dimensional $W$-polycones in $\sig$, and by $\sig_{\max}$ the set of $\fleq$-maximal elements of $\sig$. A $W$-fan $\sig$ is called \textit{relatively full-dimensional} if $\sig_{\dim(\sig)}\neq\emptyset$, \textit{relatively skeletal complete} if $\cone(\sig)=\langle\sig\rangle$, \textit{complete (in $V$)} if $\bigcup\sig=V$, \textit{full (in $V$)} if $\langle\sig\rangle=V$, and \textit{full-dimensional (in $V$)} if $\sig_n\neq\emptyset$; it is called \textit{$W$-regular} or \textit{simplicial} if every $\sigma\in\sig$ is $W$-regular or simplicial, respectively.

If $\sigma$ is a sharp $W$-polycone then $\face(\sigma)$ is a $W$-fan, called \textit{the facial fan of $\sigma$;} for $k\in\Z$ we write by abuse of language $\sigma_k$ instead of $\face(\sigma)_k$. A $W$-fan is called \textit{affine} if it is empty or the facial fan of a sharp $W$-polycone, i.e., if $\card(\sig_{\max})\leq 1$.
\end{no}

\begin{no}\label{1.45}
Let $W$ be an $R$-structure on $V$, let $\sig$ be a $W$-fan, let $V'\dfgl\langle\sig\rangle$, and let $W'\dfgl W\cap V'$ be the $R$-structure on $V'$ induced by $W$. We denote by $\sig'$ the set $\sig$ considered as a full $W'$-fan, and we call it \textit{the full fan associated with $\sig$.} If we consider $\sigma\in\sig$ as an element of $\sig'$, i.e., as a $W'$-polycone in $V'$, then we denote it by $\sigma'$.
\end{no}


\subsection{Picard modules}\label{sub1.2}

\noindent\textit{Let $W$ be an $R$-structure on $V$ and let $\sig$ be a $W$-fan. Morphisms, sequences and diagrams are understood to be morphisms, sequences and diagrams of $R$-modules.}

\begin{no}\label{1.70}
The diagonal morphism $\Delta_{W,\sig}\colon W^*\rightarrow(W^*)^{\sig}$ induces by restriction and coastriction a morphism $\Delta'_{W,\sig}\colon \langle\sig\rangle^{\perp}_{W^*}\rightarrow\prod_{\sigma\in\sig}\sigma^{\perp}_{W^*}$. Taking cokernels yields a commutative diagram with exact rows $$\xymatrix@R15pt@C40pt{0\ar[r]&\langle\sig\rangle^{\perp}_{W^*}\ar@{ (->}[r]\ar[d]_{\Delta'_{W,\sig}}&W^*\ar[r]^(.46){f'_{W,\sig}}\ar[d]^{\Delta_{W,\sig}}&K_{W,\sig}\ar[r]\ar@{->}[d]^{\Delta''_{W,\sig}}&0\\0\ar[r]&\prod_{\sigma\in\sig}\sigma^{\perp}_{W^*}\ar@{ (->}[r]&(W^*)^{\sig}\ar[r]^{h_{W,\sig}}&P_{W,\sig}\ar[r]&0.}$$ Its left quadrangle is cartesian. Hence, $\Delta''_{W,\sig}$ is a monomorphism by means of which we consider $K_{W,\sig}$ as a sub-$R$-module of $P_{W,\sig}$. In particular, we can choose the above cokernels such that the set underlying $P_{W,\sig}$ is the set of families $(m_{\sigma}+\sigma^{\vee})_{\sigma\in\sig}$ of subsets of $V^*$ with $(m_{\sigma})_{\sigma\in\sig}\in(W^*)^{\sig}$, that $h_{W,\sig}$ is given by $(m_{\sigma})_{\sigma\in\sig}\mapsto(m_{\sigma}+\sigma^{\vee})_{\sigma\in\sig}$, and that $K_{W,\sig}=\im(h_{W,\sig}\circ\Delta_{W,\sig})\subseteq P_{W,\sig}$. Clearly, $f'_{W,\sig}$ is an isomorphism if and only if $\sig$ is full, and then we denote its inverse by $f_{W,\sig}\colon K_{W,\sig}\rightarrow W^*$.

If $(m_{\sigma})_{\sigma\in\sig},(m'_{\sigma})_{\sigma\in\sig}\in(W^*)^{\sig}$ with $h_{W,\sig}((m_{\sigma})_{\sigma\in\sig})=h_{W,\sig}((m'_{\sigma})_{\sigma\in\sig})$, then for $\tau\fleq\sigma\in\sig$ the relations $m_{\sigma}-m_{\tau}\in\tau^{\perp}$ and $m'_{\sigma}-m'_{\tau}\in\tau^{\perp}$ are equivalent. This allows to define a \textit{virtual polytope over $\sig$ (with respect to $W$)} as an element $(m_{\sigma}+\sigma^{\vee})_{\sigma\in\sig}\in P_{W,\sig}$ such that for every $\sigma\in\sig$ and every $\tau\fleq\sigma$ (or -- equivalently -- every $\tau\in\sigma_1$) we have $m_{\sigma}-m_{\tau}\in\tau^{\perp}$. The set $\overline{P}_{W,\sig}$ of virtual polytopes over $\sig$ is a sub-$R$-module of $P_{W,\sig}$ containing $K_{W,\sig}$; it is of finite type since $R$ is principal. We denote by $g_{W,\sig}\colon K_{W,\sig}\hookrightarrow\overline{P}_{W,\sig}$ the canonical injection and by $e_{W,\sig}\colon \overline{P}_{W,\sig}\rightarrow\pic_W(\sig)$ its cokernel, obtaining an exact sequence $$0\longrightarrow K_{W,\sig}\xrightarrow{g_{W,\sig}}\overline{P}_{W,\sig}\xrightarrow{e_{W,\sig}}\pic_W(\sig)\longrightarrow 0.$$ The $R$-module $\pic_W(\sig)$ is called \textit{the Picard module of $\sig$ (with respect to $W$);} it is of finite type. 
\end{no}

\begin{no}\label{1.80}
Let $V'\dfgl\langle\sig\rangle$ and $W'\dfgl W\cap V'$, and let $\sig'$ denote the full fan associated with $\sig$ (\ref{1.45}). The dual of the canonical injection $W'\hookrightarrow W$ is an epimorphism $\alpha\colon W^*\rightarrow(W')^*$ with kernel $\langle\sig\rangle^{\perp}_{W^*}$ that induces by corestriction and astriction an isomorphism $\alpha'\colon K_{W,\sig}\rightarrow K_{W',\sig'}$ such that the diagram $$\xymatrix@R15pt@C40pt{W^*\ar[r]^(.46){f'_{W,\sig}}\ar[d]_{\alpha}&K_{W,\sig}\ar[d]^{\cong}_{\alpha'}\\(W')^*\ar[r]^(.49){f'_{W',\sig'}}&K_{W',\sig'}}$$ commutes. It also induces an epimorphism $\beta\colon (W^*)^{\sig}\rightarrow((W')^*)^{\sig}$ with kernel $(\langle\sig\rangle^{\perp}_{W^*})^{\sig}$, inducing by restriction and coastriction an epimorphism $\prod_{\sigma\in\sig}\sigma^{\perp}_{W^*}\rightarrow\prod_{\sigma'\in\sig'}\sigma'^{\perp}_{(W')^*}$ with the same kernel. Hence, $\beta$ induces by corestriction and astriction an isomorphism $\beta'\colon P_{W,\sig}\rightarrow P_{W',\sig'}$ such that the diagram $$\xymatrix@R15pt@C40pt{(W^*)^{\sig}\ar[r]^{h_{W,\sig}}\ar[d]_{\beta}&P_{W,\sig}\ar[d]^{\cong}_{\beta'}\\((W')^*)^{\sig}\ar[r]^(.52){h_{W',\sig'}}&P_{W',\sig'}}$$ commutes. This induces by restriction and coastriction an isomorphism $\gamma\colon\overline{P}_{W,\sig}\rightarrow\overline{P}_{W',\sig'}$ such that the diagram $$\xymatrix@R15pt@C40pt{K_{W,\sig}\ar@{ (->}[r]\ar[d]^{\cong}_{\alpha'}&\overline{P}_{W,\sig}\ar@{ (->}[r]\ar[d]^{\cong}_{\gamma}&P_{W,\sig}\ar[d]^{\cong}_{\beta'}\\K_{W',\sig'}\ar@{ (->}[r]&\overline{P}_{W',\sig'}\ar@{ (->}[r]&P_{W',\sig'}}$$ commutes, and that induces by corestriction and astriction an isomorphism $\delta\colon\pic_W(\sig)\rightarrow\pic_{W'}(\sig')$ such that the diagram $$\xymatrix@R15pt@C40pt{\overline{P}_{W,\sig}\ar[r]^(.44){e_{W,\sig}}\ar[d]^{\cong}_{\gamma}&\pic_W(\sig)\ar[d]^{\cong}_{\delta}\\\overline{P}_{W',\sig'}\ar[r]^(.45){e_{W',\sig'}}&\pic_{W'}(\sig')}$$ commutes.
\end{no}

\begin{prop}\label{1.90}
If $\sig$ is relatively full-dimensional then $\pic_W(\sig)$ is free.
\end{prop}

\begin{proof}
We can suppose that $\sig$ is full (\ref{1.80}), so that there exists $\omega\in\sig_n$. Let $(m_{\sigma})_{\sigma\in\sig}\in(W^*)^{\sig}$ and $r\in R\setminus 0$ such that there exists $m\in W^*$ with $(rm_{\sigma}+\sigma^{\vee})_{\sigma\in\sig}=(m+\sigma^{\vee})_{\sigma\in\sig}$. We have $rm_{\omega}-m\in\omega^{\perp}=0$, hence $(rm_{\sigma}+\sigma^{\vee})_{\sigma\in\sig}=(rm_{\omega}+\sigma^{\vee})_{\sigma\in\sig}$, implying $r(m_{\sigma}-m_{\omega})\in\sigma^{\perp}$ and therefore $m_{\sigma}-m_{\omega}\in\sigma^{\perp}$ for every $\sigma\in\sig$. It follows $(m_{\sigma}+\sigma^{\vee})_{\sigma\in\sig}=(m_{\omega}+\sigma^{\vee})_{\sigma\in\sig}\in K_{W,\sig}$. Thus, $\pic_W(\sig)$ is torsionfree, and as $R$ is principal we get the claim (\ref{1.70}, \cite[VII.4.4 Th\'eor\`eme 2 Corollaire 2]{a}).
\end{proof}

\begin{nosm}\label{1.100}
The $R$-module $\pic_W(\sig)$ is not necessarily free, even if $\sig$ is $W$-regular. A counterexample is given by the $\Z^2$-regular $\Z^2$-fan $\sig$ in $\R^2$ with maximal polycones $\cone(1,0)$ and $\cone(1,2)$, for $\pic_{\Z^2}(\sig)\cong\Z/2\Z$.

If $\pic_W(\sig)$ is free then $\sig$ is not necessarily relatively full-dimensional, even if $\sig$ is $W$-regular and relatively skeletal complete. A counterexample is given by the $\Z^2$-regular, relatively skeletal complete $\Z^2$-fan $\sig$ in $\R^2$ with maximal polycones $\cone(1,0)$, $\cone(0,1)$ and $\cone(-1,-1)$, for $\pic_{\Z^2}(\sig)\cong\Z$.\footnote{cf.~\ref{1.395}}
\end{nosm}

\begin{prop}\label{1.120}
For $\omega\in\sig$ and $\alpha\in\pic_W(\sig)$ there exists a unique $(m_{\sigma}+\sigma^{\vee})_{\sigma\in\sig}\in\overline{P}_{W,\sig}$ with $e_{W,\sig}((m_{\sigma}+\sigma^{\vee})_{\sigma\in\sig})=\alpha$ such that $m_{\sigma}=0$ for every $\sigma\fleq\omega$.
\end{prop}

\begin{proof}
Let $p=(m'_{\sigma}+\sigma^{\vee})_{\sigma\in\sig}\in e_{W,\sig}^{-1}(\alpha)$. Set $m_{\sigma}\dfgl m'_{\omega}$ for $\sigma\fleq\omega$ and $m_{\sigma}\dfgl m'_{\sigma}$ for $\sigma\in\sig\setminus\face(\omega)$. It is readily checked that $(m_{\sigma}+\sigma^{\vee})_{\sigma\in\sig}$ has the desired properties.
\end{proof}


\subsection{The diagram of a fan}\label{sub1.3}

\noindent\textit{Let $N$ be a $\Z$-structure on $V$, let $M\dfgl N^*$, let $W$ denote the image of the canonical morphism of $R$-modules $R\otimes_{\Z}N\rightarrow V$ with $a\otimes x\mapsto ax$ (which is an $R$-structure on $V$), and let $\sig$ be an $N$-fan. If not specified otherwise, morphisms, sequences and diagrams are understood to be morphisms, sequences and diagrams of $R$-modules.}

\begin{no}\label{1.140}
The family $(\rho_N)_{\rho\in\sig_1}$ of $N$-minimal $N$-generators of the $1$-dimensional polycones in $\sig$ (\ref{1.35}), considered as a family in $\hm{R}{W^*}{R}$, corresponds to a morphism\footnote{Although the modules and morphisms constructed in the following depend not only on $W$ but on $N$ and $R$, we do not indicate this in the notation for the sake of readability.} $$c_{W,\sig_1}\colon W^*\rightarrow R^{\sig_1},\;m\mapsto(\rho_N(m))_{\rho\in\sig_1}.$$ Denoting its cokernel by $a_{W,\sig_1}\colon R^{\sig_1}\rightarrow A_{W,\sig_1}$ we obtain an exact sequence $$W^*\xrightarrow{c_{W,\sig_1}}R^{\sig_1}\xrightarrow{a_{W,\sig_1}}A_{W,\sig_1}\longrightarrow 0.$$ As $\ke(c_{W,\sig_1})=\langle\sig\rangle^{\perp}_{W^*}$ we see that $c_{W,\sig_1}$ is a monomorphism if and only if $\sig$ is full.

We denote the canonical basis of the free $R$-module $R^{\sig_1}$ by $(\delta_{\rho})_{\rho\in\sig_1}$. For $\rho\in\sig_1$ we set $\alpha_{\rho}\dfgl a_{W,\sig_1}(\delta_{\rho})\in A_{W,\sig_1}$. For $\sigma\in\sig$ we set $$\del_{\sigma}\dfgl\sum_{\rho\in\sig_1\setminus\sigma_1}\delta_{\rho}\in R^{\sig_1}$$ and $$\alp_{\sigma}\dfgl a_{W,\sig_1}(\del_{\sigma})=\sum_{\rho\in\sig_1\setminus\sigma_1}\alpha_{\rho}\in A_{W,\sig_1},$$ and moreover $A^{\sigma}_{W,\sig}\dfgl\langle\alpha_{\rho}\mid\rho\in\sig_1\setminus\sigma_1\rangle_R\subseteq A_{W,\sig_1}$. Then, $A_{W,\sig_1}=\langle\alp_{\sigma}\mid\sigma\in\sig\rangle_R$, and as for $\tau\fleq\sigma\in\sig$ we have $A^{\sigma}_{W,\sig}\subseteq A^{\tau}_{W,\sig}$ it follows $$\bigcap_{\sigma\in\sig}A^{\sigma}_{W,\sig}=\bigcap_{\sigma\in\sig_{\max}}A^{\sigma}_{W,\sig}\subseteq A_{W,\sig_1}.$$

Composition of the canonical projection $(W^*)^{\sig}\rightarrow(W^*)^{\sig_1}$ with $\prod_{\rho\in\sig_1}\rho_N\colon(W^*)^{\sig_1}\rightarrow R^{\sig_1}$ yields a morphism $d'_{W,\sig}\colon(W^*)^{\sig}\rightarrow R^{\sig_1}$ with $d'_{W,\sig}\circ\Delta_{W,\sig}=c_{W,\sig_1}$, inducing by astriction a morphism $d''_{W,\sig}\colon P_{W,\sig}\rightarrow R^{\sig_1}$ with $d''_{W,\sig}\circ h_{W,\sig}=d'_{W,\sig}$. Restriction yields a morphism $d_{W,\sig}\colon\overline{P}_{W,\sig}\rightarrow R^{\sig_1}$ with $d_{W,\sig}\circ g_{W,\sig}\circ h_{W,\sig}\circ\Delta_{W,\sig}=c_{W,\sig_1}$, and this induces by corestriction and astriction a morphism $b_{W,\sig}\colon\pic_W(\sig)\rightarrow A_{W,\sig_1}$ such that we have a commutative diagram with exact rows $$\xymatrix@R15pt@C40pt{0\ar[r]&K_{W,\sig}\ar[r]^(.49){g_{W,\sig}}&\overline{P}_{W,\sig}\ar[r]^(.45){e_{W,\sig}}\ar[d]^{d_{W,\sig}}&\pic_W(\sig)\ar[r]\ar[d]^{b_{W,\sig}}&0\\&W^*\ar[r]^(.49){c_{W,\sig_1}}\ar[u]^{f'_{W,\sig}}&R^{\sig_1}\ar[r]^(.47){a_{W,\sig_1}}&A_{W,\sig_1}\ar[r]&0.}$$ This diagram is denoted by $\mathbbm{D}_{W,\sig}$ and called \textit{the diagram of $\sig$ (with respect to $N$ and $R$).} If $\sig$ is full then we get from it a commutative diagram with exact rows $$\xymatrix@R15pt@C40pt{0\ar[r]&K_{W,\sig}\ar[d]_{f_{W,\sig}}\ar[r]^(.49){g_{W,\sig}}&\overline{P}_{W,\sig}\ar[r]^(.45){e_{W,\sig}}\ar[d]^{d_{W,\sig}}&\pic_W(\sig)\ar[r]\ar[d]^{b_{W,\sig}}&0\\0\ar[r]&W^*\ar[r]^(.49){c_{W,\sig_1}}&R^{\sig_1}\ar[r]^(.47){a_{W,\sig_1}}&A_{W,\sig_1}\ar[r]&0.}$$ It is clear that $\mathbbm{D}_{W,\sig}\otimes_RK=\mathbbm{D}_{W_K,\sig}$.

In case $R=\Z$ (and hence $W=N$) we omit $W$ in all the above notations and write for example $\mathbbm{D}_{\sig}$, $\pic(\sig)$ and $A_{\sig_1}$ instead of $\mathbbm{D}_{N,\sig}$, $\pic_N(\sig)$ and $A_{N,\sig_1}$.
\end{no}

\begin{no}\label{1.160}
Let $V'\dfgl\langle\sig\rangle$, $N'\dfgl N\cap V'$ and $W'\dfgl W\cap V'$, and let $\sig'$ denote the full fan associated with $\sig$ (\ref{1.45}). By \ref{1.80}, the diagram $$\xymatrix@R10pt@C30pt{&K_{W,\sig}\ar[ld]_{\alpha'}^{\cong}\ar[rr]^{g_{W,\sig}}&&\overline{P}_{W,\sig}\ar[ld]^{\cong}_{\gamma}\ar[dd]^{d_{W,\sig}}\\K_{W',\sig'}\ar[rr]^(.7){g_{W',\sig'}}&&\overline{P}_{W',\sig'}\ar[dd]^(.3){d_{W',\sig'}}&\\&W^*\ar[ld]_{\alpha}\ar[uu]^(.3){f'_{W,\sig}}\ar[rr]^(.3){c_{W,\sig_1}}&&R^{\sig_1}\ar@{=}[ld]\\(W')^*\ar[uu]^{f'_{W',\sig'}}\ar[rr]_{c_{W',\sig'_1}}&&R^{\sig_1}&}$$ commutes. Hence, there is an isomorphism $\eps\colon A_{W,\sig_1}\rightarrow A_{W,\sig'_1}$ such that the diagram $$\xymatrix@R10pt@C30pt{&\overline{P}_{W,\sig}\ar[rr]^{e_{W,\sig}}\ar[ld]_{\cong}^{\gamma}\ar[dd]^(.3){d_{W,\sig}}&&\pic_W(\sig)\ar[dd]^{b_{W,\sig}}\ar[ld]_{\cong}^{\delta}\\\overline{P}_{W',\sig'}\ar[rr]_(.65){e_{W',\sig'}}\ar[dd]_{d_{W',\sig'}}&&\pic_{W'}(\sig')\ar[dd]^(.3){b_{W',\sig'}}&\\&R^{\sig_1}\ar[rr]^(.3){a_{W,\sig_1}}\ar@{=}[ld]&&A_{W,\sig_1}\ar[ld]_{\cong}^{\eps}\\R^{\sig_1}\ar[rr]^{a_{W',\sig'_1}}&&A_{W',\sig'_1}&}$$ commutes. By restriction and coastriction it induces for every $\sigma\in\sig$ an isomorphism $\eps'_{\sigma}\colon A^{\sigma}_{W,\sig}\rightarrow A^{\sigma'}_{W',\sig'}$.
\end{no}

\begin{no}\label{1.170}
If $(m_{\sigma})_{\sigma\in\sig}\in(W^*)^{\sig}$ with $(m_{\sigma}+\sigma^{\vee})_{\sigma\in\sig}\in\ke(d_{W,\sig})$ then $m_{\rho}\in\rho^{\perp}$ for $\rho\in\sig_1$ and $m_{\sigma}-m_{\rho}\in\rho^{\perp}$ for $\sigma\in\sig$ and $\rho\in\sigma_1$, hence $m_{\sigma}\in\bigcap_{\rho\in\sigma_1}\rho^{\perp}=\sigma^{\perp}$ for $\sigma\in\sig$, and thus $(m_{\sigma}+\sigma^{\vee})_{\sigma\in\sig}=0$. So, $d_{W,\sig}$ is a monomorphism. Thus, keeping in mind \ref{1.160}, the Snake Lemma implies that $b_{W,\sig}$ is a monomorphism; by means of this we consider $\pic_W(R)$ as a sub-$R$-module of $A_{W,\sig_1}$.
\end{no}

\begin{prop}\label{1.210}
$\pic_W(\sig)=\bigcap_{\sigma\in\sig}A^{\sigma}_{W,\sig}$.
\end{prop}

\begin{proof}
If $\alpha\in\pic_W(\sig)$, then for $\omega\in\sig$ there exists a family $(m_{\sigma})_{\sigma\in\sig}$ in $W^*$ with $\alpha=e_{W,\sig}((m_{\sigma}+\sigma^{\vee})_{\sigma\in\sig})$ and $m_{\sigma}=0$ for $\sigma\fleq\omega$ (\ref{1.120}), implying $\alpha=a_{W,\sig_1}((\rho_N(m_{\rho}))_{\rho\in\sig_1})=\sum_{\rho\in\sig_1\setminus\omega_1}\rho_N(m_{\rho})\alpha_{\rho}\in A^{\omega}_{W,\sig}$.

Conversely, let $\alpha\in\bigcap_{\sigma\in\sig}A^{\sigma}_{W,\sig}$. There exists $(r_{\rho})_{\rho\in\sig_1}\in R^{\sig_1}$ with $\alpha=\sum_{\rho\in\sig_1}r_{\rho}\alpha_{\rho}$. Let $\sigma\in\sig$. As $\alpha\in A^{\sigma}_{W,\sig}$ there exists $(s_{\rho}^{(\sigma)})_{\rho\in\sig_1}\in R^{\sig_1}$ with $s_{\rho}^{(\sigma)}=0$ for $\rho\fleq\sigma$ and $\alpha=\sum_{\rho\in\sig_1}s_{\rho}^{(\sigma)}\alpha_{\rho}$. It follows $(r_{\rho}-s_{\rho}^{(\sigma)})_{\rho\in\sig_1}\in\ke(a_{W,\sig_1})$, and hence there exists $m_{\sigma}\in W^*$ with $c_{W,\sig_1}(m_{\sigma})=(r_{\rho}-s_{\rho}^{(\sigma)})_{\rho\in\sig_1}$. So, for $\rho\in\sigma_1$ we have $\rho_N(m_{\sigma})=r_{\rho}-s_{\rho}^{(\sigma)}=r_{\rho}$, hence $\rho_N(m_{\sigma}-m_{\rho})=r_{\rho}-r_{\rho}=0$. It follows $p\dfgl(m_{\sigma}+\sigma^{\vee})_{\sigma\in\sig}\in\overline{P}_{W,\sig}$ and $d_{W,\sig}(p)=(\rho_N(m_{\rho}))_{\rho\in\sig_1}=(r_{\rho})_{\rho\in\sig_1}$, implying $\alpha=a_{W,\sig_1}(d_{W,\sig_1}(p))=e_{W,\sig}(p)\in\pic_W(\sig)$.
\end{proof}

\begin{cor}\label{1.220}
If $\sig$ is affine then $\pic_W(\sig)=0$.
\end{cor}

\begin{proof}
Immediately from \ref{1.210}.
\end{proof}

\begin{nosm}\label{1.230}
If $\pic_W(\sig)=0$ then $\sig$ is not necessarily affine, even if $\sig$ is $W$-regular. A counterexample is given by the $\Z^2$-regular $\Z^2$-fan in $\R^2$ with maximal polycones $\cone(1,0)$ and $\cone(0,1)$.\footnote{cf.~\ref{1.310}~c)}
\end{nosm}

\begin{no}\label{1.200}
A sub-$R$-module $B\subseteq A_{W,\sig_1}$ is called \textit{big} if $(A_{W,\sig_1}/B)\otimes_RK=0$ (i.e., if $B$ has finite index in $A_{W,\sig_1}$), and \textit{small} if $B\subseteq\pic_{W,\sig}$. Note that bigness depends only on $\sig_1$, while smallness depends on $\sig$.
\end{no}

\begin{no}\label{1.350}
Let $S$ be a cancellable monoid\footnote{i.e., if $x,y,z\in S$ with $x+z=y+z$ then $x=y$}, considered as a submonoid of its group of differences $\diff(S)$ (\cite[I.2.4]{a}). We furnish $\diff(S)$ with the structure of preordered group whose monoid of positive elements equals $S$, and $S$ with the induced structure of preordered monoid. We call $S$ \textit{sharp} if $\diff(S)$ is an ordered group, or -- equivalently -- if $S\cap(-S)=0$ (\cite[VI.1.3 Propopsition 3]{a}).

Submonoids of sharp cancellable monoids are sharp. For $m\in\N$ the cancellable monoid $\N^m$ is sharp, and its structure of ordered monoid is the product of the usual structure of ordered monoid on $\N$. Sharp cancellable monoids of finite type are artinian\footnote{i.e., nonempty subsets have minimal elements}.
\end{no}

\begin{no}\label{1.380}
The submonoid $\degmon(\sig_1)\dfgl\langle\alpha_{\rho}\mid\rho\in\sig_1\rangle_{\N}\subseteq A_{\sig_1}$ is called \textit{the degree monoid of $\sig_1$.} It is cancellable and hence furnished with a structure of preordered monoid (\ref{1.350}). The morphism of groups $a_{\sig_1}\colon\Z^{\sig_1}\rightarrow A_{\sig_1}$ induces by restriction and coastriction a surjective morphism of preordered monoids $a_{\sig_1}^+\colon\N^{\sig_1}\rightarrow\degmon(\sig_1)$ (where $\N^{\sig_1}$ is furnished with the product ordering) with $\ke(a_{\sig_1}^+)=\im(c_{\sig_1})\cap\N^{\sig_1}$.
\end{no}

\begin{prop}\label{1.390}
The following statements are equivalent:
\begin{equi}
\item[(i)] $\sig$ is relatively skeletal complete;
\item[(ii)] $\im(c_{W,\sig})\cap R_{\geq 0}^{\sig_1}=0$;
\item[(iii)] $\degmon(\sig_1)$ is sharp and $\alpha_{\rho}\neq 0$ for every $\rho\in\sig_1$.
\end{equi}
\end{prop}

\begin{proof}
(i) holds if and only if $\bigcap_{\rho\in\sig_1}\rho^{\vee}\subseteq\langle\sig\rangle^{\perp}$, hence if and only if $W^*\cap(\bigcap_{\rho\in\sig_1}\rho^{\vee})\subseteq\ke(c_{W,\sig_1})$, thus if and only if (ii) holds. If (ii) holds then $a_{\sig_1}^+$ is an isomorphism of preordered monoids (\ref{1.380}), and \textit{a fortiori} of ordered monoids, implying (iii) (\ref{1.350}). Finally, suppose that (iii) holds. Then, $A_{W,\sig_1}$ is torsionfree. Let $(r_{\rho})_{\rho\in\sig_1}\in\im(c_{W,\sig_1})\cap R_{\geq 0}^{\sig_1}$. For $\xi\in\sig_1$ we have $$r_{\xi}\alpha_{\xi}=-\sum_{\rho\in\sig_1\setminus\{\xi\}}r_{\rho}\alpha_{\rho}\in\degmon(\sig_1)\cap(-\degmon(\sig_1))=0,$$ implying $r_{\xi}=0$ and therefore (ii).
\end{proof}

\begin{cor}\label{1.395}
If $\sig$ is relatively skeletal complete then $A_{W,\sig_1}$ and $\pic_W(\sig)$ are free.
\end{cor}

\begin{proof}
The monoid $\degmon(\sig_1)$ is sharp (\ref{1.390}), so its group of differences $A_{\sig_1}$ is an ordered group, hence torsionfree. Therefore, $A_{\sig_1}$ and its subgroup $\pic(\sig)$ are free, and then the claim follows from \ref{1.140}.
\end{proof}

\begin{nosm}\label{1.400}
If $\degmon(\sig_1)$ is sharp then $\sig$ is not necessarily relatively skeletal complete. A counterexample is given by the $\Z^2$-fan in $\R^2$ with maximal polycones $\cone(1,0)$ and $\cone(0,1)$, for $\degmon(\sig_1)=A_{\sig_1}=0$.

The monoid $\degmon(\sig_1)$ is not necessarily sharp, even if $A_{\sig_1}$ is free (and hence the converse of \ref{1.395} does not hold). A counterexample is given by the $\Z^2$-fan in $\R^2$ with maximal polycones $\cone(1,0)$, $\cone(1,1)$ and $\cone(0,1)$, for $\degmon(\sig_1)=A_{\sig_1}\cong\Z$.
\end{nosm}


\subsection{Regular and simplicial fans}\label{sub1.4}

\noindent\textit{We keep the hypotheses from \ref{sub1.3}.}

\begin{lemma}\label{1.240}
Let $A$ be a ring, let $L$ be a free $A$-module of finite rank, let $E$ be a generating set of $L$, and let $c\colon L\rightarrow L^{**}$ denote the canonical isomorphism of $A$-modules. Then, $E$ is a basis of $L$ if and only if\/ $\bigcap_{e\in E}(u_e+\ke(c(e)))\neq\emptyset$ for every family $(u_e)_{e\in E}$ in $L^*$.
\end{lemma}

\begin{proof}
Suppose $E$ is a basis of $L$ and let $(u_e)_{e\in E}$ be a family in $L^*$. For $e\in E$ let $e^*$ denote the element of the dual basis of $E$ corresponding to $e$. For $e\in E$ there is a family $(u_e^{(f)})_{f\in E}$ in $A$ with $u_e=\sum_{f\in E}u_e^{(f)}f^*$, and $\sum_{f\in E}u_f^{(f)}f^*\in\bigcap_{e\in E}(u_e+\ke(c(e)))$.

Conversely, suppose the above condition to hold and assume there are $e_0\in E\setminus 0$ and a family $(r_e)_{e\in E}$ in $A$ with $r_{e_0}\neq 0$ and $r_{e_0}e_0=\sum_{e\in F}r_ee$, where $F\dfgl E\setminus\{e_0\}$. For $u\in L^*$ there is a $v\in(u+\ke(c(e_0)))\cap(\bigcap_{e\in F}\ke(c(e)))$, hence a $w\in\ke(c(e_0))$ with $u=v-w$. Therefore, $$L^*\subseteq(\bigcap_{e\in F}\ke(c(e)))+\ke(c(e_0))\subseteq\ke(c(\sum_{e\in F}r_ee))+\ke(c(r_{e_0}e_0))\subseteq\ke(c(r_{e_0}e_0)),$$ implying $r_{e_0}e_0=0$ and thus the contradiction $e_0=0$.
\end{proof}

\begin{prop}\label{1.250}
If $\sig$ is $W$-regular then $\pic_W(\sig)=A_{W,\sig_1}$.
\end{prop}

\begin{proof}
Let $(m_{\rho})_{\rho\in\sig_1}$ be a family in $W^*$ and let $\sigma\in\sig$. There are $(a_{\sigma,\rho})_{\rho\in\sigma_1}\in\prod_{\rho\in\sigma_1}W\cap\rho$ and an $R$-basis $E^{(\sigma)}$ of $W$ with $\{a_{\sigma,\rho}\mid\rho\in\sigma_1\}\subseteq E^{(\sigma)}$. For $\rho\in\sigma_1$ we set $m_{a_{\sigma,\rho}}^{(\sigma)}\dfgl m_{\rho}$ and for $e\in E^{(\sigma)}\setminus\{a_{\sigma,\rho}\mid\rho\in\sigma_1\}$ we set $m_e^{(\sigma)}\dfgl 0\in W^*$. Then, $W^*\cap(\bigcap_{e\in E^{(\sigma)}}(m_e^{(\sigma)}+e^{\perp}))\neq\emptyset$ (\ref{1.240}) and hence $W^*\cap(\bigcap_{\rho\in\sigma_1}(m_{\rho}+\rho^{\perp}))\neq\emptyset$. So, there exists $m\in W^*$ with $m-m_{\rho}\in\rho^{\perp}$ for every $\rho\in\sigma_1$. Thus, there is a family $(m_{\tau}^{(\sigma)})_{\tau\fleq\sigma}$ in $W^*$ with $(m_{\rho}^{(\sigma)})_{\rho\in\sigma_1}=(m_{\rho})_{\rho\in\sig_1}$ and $(m_{\tau}^{(\sigma)}+\tau^{\vee})_{\tau\fleq\sigma}\in\overline{P}_{W,\face(\sigma)}$. If $\tau\fleq\sigma$ then for $\rho\in\tau_1$ it follows $m_{\tau}^{(\sigma)}-m_{\tau}^{(\tau)}=(m_{\tau}^{(\sigma)}-m_{\rho})-(m_{\tau}^{(\tau)}-m_{\rho})\in\rho^{\perp}$, hence $m_{\tau}^{(\sigma)}-m_{\tau}^{(\tau)}\in\bigcap_{\rho\in\tau_1}\rho^{\perp}=\tau^{\perp}$ and therefore $m_{\tau}^{(\sigma)}+\tau^{\vee}=m_{\tau}^{(\tau)}+\tau^{\vee}$. This shows that there exists a family $(m_{\sigma})_{\sigma\in\sig\setminus\sig_1}$ in $W^*$ with $(m_{\sigma}+\sigma^{\vee})_{\sigma\in\sig}\in\overline{P}_{W,\sig}$. Therefore, $d_{W,\sig}$ is an epimorphism, and the Snake Lemma applied to $\mathbbm{D}_{W,\sig}$ yields the claim.
\end{proof}

\begin{prop}\label{1.260}
Let $\sigma\in\sig$. The following statements are equivalent:
\begin{equi}
\item[(i)] $A^{\sigma}_{W,\sig}=A_{W,\sig_1}$;
\item[(ii)] For $\tau\flneq\sigma$ there exists $u\in\tau^{\perp}_{W^*}$ such that for $\rho\in\sigma_1\setminus\tau_1$ we have $\rho_N(u)=1$;
\item[(iii)] For $\tau\in\sig$ there exists $u\in(\sigma\cap\tau)^{\perp}_{W^*}$ such that for $\rho\in\sigma_1\setminus\tau_1$ we have $\rho_N(u)=1$;
\item[(iv)] $\sigma$ is $W$-regular.
\end{equi}
\end{prop}

\begin{proof}
``(i)$\Rightarrow$(ii)'': If $\tau\flneq\sigma$ then $\alp_{\tau}\in A_{W,\sig_1}=A_{W,\sig}^{\sigma}=\langle\alp_{\rho}\mid\rho\in\sig_1\setminus\sigma_1\rangle_R$, so there exists a family $(r_{\rho})_{\rho\in\sig_1\setminus\sigma_1}$ in $R$ with $$a(\sum_{\rho\in\sig_1\setminus\tau_1}\delta_{\rho}-\sum_{\rho\in\sig_1\setminus\sigma_1}r_{\rho}\delta_{\rho})=0.$$ Hence, there exists $u\in W^*$ with $$\sum_{\rho\in\sig_1\setminus\tau_1}\delta_{\rho}-\sum_{\rho\in\sig_1\setminus\sigma_1}r_{\rho}\delta_{\rho}=\sum_{\rho\in\sig_1}\rho_N(u)\delta_{\rho}.$$ It follows $\rho_N(u)=0$ for $\rho\in\tau_1$ and $\rho_N(u)=1$ for $\rho\in\sigma_1\setminus\tau_1$, thus (ii).

``(ii)$\Rightarrow$(iii)'': Let $\tau\in\sig$. If $\sigma\fleq\tau$ then $\sigma_1\setminus\tau_1=\emptyset$, hence every $u\in(\sigma\cap\tau)^{\perp}_{W^*}$ has the desired property. Otherwise, $\omega\dfgl\sigma\cap\tau\flneq\sigma$, and as $\sigma_1\setminus\omega_1=\sigma_1\setminus\tau_1$ we get (iii).

``(iii)$\Rightarrow$(i)'': Let $\tau\in\sig$ and let $u\in(\sigma\cap\tau)^{\perp}_{W^*}$ with $\rho_N(u)=1$ for $\rho\in\sigma_1\setminus\tau_1$. Setting $$f\dfgl\sum_{\rho\in\sig_1\setminus(\sigma_1\cup\tau_1)}(1-\rho_N(u))\delta_{\rho}-\sum_{\rho\in\tau_1\setminus\sigma_1}\rho_N(u)\delta_{\rho}\in R^{\sig_1}$$ it follows $$\sum_{\rho\in\sig_1\setminus\tau_1}\delta_{\rho}-f=\sum_{\rho\in\sig_1}\rho_N(u)\delta_{\rho}=c_{W,\sig_1}(u),$$ hence $a(\sum_{\rho\in\sig_1\setminus\tau_1}\delta_{\rho}-f)=0$, and therefore $$\alp_{\tau}=a_{W,\sig_1}(f)\in\langle\alp_{\rho}\mid\rho\in\sig_1\setminus\sigma_1\rangle_R=A_{W,\sig}^{\sigma}.$$ As $A_{W,\sig_1}=\langle\alp_{\tau}\mid\tau\in\sig\rangle_R$ this implies (i).

``(ii)$\Rightarrow$(iv)'': We can without loss of generality suppose that $\sig=\face(\sigma)$ and that $\sig$ is moreover full. If $\sigma$ is not simplicial then $\card(\sig_1)>n$, hence $A^{\sigma}_{W,\sig}=0\neq A_{W,\sig_1}$ (\ref{1.220}), but as (iii) implies (i) this contradicts (iii). So, $\sigma$ is simplicial.

Now, we prove the claim by induction on $n$. If $n\leq 1$ it is clear. Suppose that $n>1$ and that the claim holds for strictly smaller values of $n$. Property (ii) is obviously inherited by faces, hence every proper face of $\sigma$ is $W$-regular. In particular, there exists a $W$-regular $\tau\in\sigma_{n-1}$. By simpliciality of $\sigma$ there exists a unique $\xi\in\sigma_1\setminus\tau_1$, and by (ii) there exists $u\in\tau^{\perp}_{W^*}$ with $u(\xi_N)=1$. The morphism $p\colon W\rightarrow W,$\linebreak$y\mapsto u(y)\xi_N$ is readily checked to be a projector of $W$ with image $\langle\xi_N\rangle_R$ and $\langle\rho_N\mid\rho\in\tau_1\rangle_R\subseteq\ke(p)$. We have $$\rk_K(\langle\rho_N\mid\rho\in\tau_1\rangle_K)=n-\dim(\xi)=n-\rk_K(\im(p\otimes_RK))=\rk_K(\ke(p\otimes_RK)),$$ hence $\langle\rho_N\mid\rho\in\tau_1\rangle_K=\ke(p\otimes_RK)$, therefore $\ke(p)\subseteq\langle\rho_N\mid\rho\in\tau_1\rangle_K\cap W=\langle\rho_N\mid\rho\in\tau_1\rangle_R\subseteq\ke(p)$ by $W$-regularity of $\tau$, and thus $\ke(p)=\langle\rho_N\mid\rho\in\tau_1\rangle_R$. So, $W=\langle\rho_N\mid\rho\in\tau_1\rangle_R\oplus\langle\xi_N\rangle_R$, and $W$-regularity of $\tau$ implies that $\{\rho_N\mid\rho\in\tau_1\}\cup\{\xi_N\}$ is an $R$-basis of $W$. This implies that $\sigma$ is $W$-regular and thus (iv).

``(iv)$\Rightarrow$(i)'': By what we have already shown we can without loss of generality suppose that $\sig=\face(\sigma)$. The claim follows then immediately from \ref{1.210} and \ref{1.250}.
\end{proof}

\begin{cor}\label{1.270}
Let $\sigma\in\sig$. Then, $\sigma$ is simplicial if and only if $A^{\sigma}_{W,\sig}$ is big.
\end{cor}

\begin{proof}
$\sigma$ is simplicial if and only if it is $W_K$-regular, hence if and only if $(A_{W,\sig_1}/A^{\sigma}_{W,\sig})\otimes_RK\cong A_{W_K,\sig_1}/A^{\sigma}_{W_K,\sig}=0$ (\ref{1.140}, \ref{1.260}), and thus if and only if $A^{\sigma}_{W,\sig}$ is big.
\end{proof}

\begin{thm}\label{1.280}
a) $\sig$ is $W$-regular if and only if $\pic_W(\sig)=A_{W,\sig_1}$.

b) $\sig$ is simplicial if and only if $\pic_W(\sig)$ is big.
\end{thm}

\begin{proof}
a) follows immediately from \ref{1.210}, \ref{1.250} and \ref{1.260}. b) is proven analogously to \ref{1.270} on use of a).
\end{proof}

\begin{cor}\label{1.300}
There exists a big and small subgroup of $A_{W,\sig_1}$ if and only if $\sig$ is simplicial.
\end{cor}

\begin{proof}
Immediately by \ref{1.280}~b).
\end{proof}

\begin{cor}\label{1.310}
a) If $A_{W,\sig_1}$ is finite then $\sig$ is simplicial.

b) $\sig$ is simplicial and affine if and only if $\sig$ is relatively full-dimensional and $A_{W,\sig}$ is finite.

c) If $\sig$ is relatively full-dimensional and simplicial, then $\pic_W(\sig)=0$ holds if and only if $\sig$ is affine.
\end{cor}

\begin{proof}
a) If $A_{W,\sig_1}$ is finite then $\pic_W(\sig)$ is big, hence $\sig$ is simplicial (\ref{1.280}~b)). b) If $\omega\in\sig_{\dim(\sig)}$ and $A_{W,\sig}$ is finite then $\card(\sig_1)=\dim(\sig)=\card(\omega_1)$, hence $\sig_1=\omega_1$, implying $\sig=\face(\omega)$ and that $\omega$ is simplicial. The converse is clear. c) If $\pic_W(\sig)=0$ then $A_{W,\sig}$ is finite (\ref{1.280}~b)), hence $\sig$ is affine by b). The converse holds by \ref{1.220}.
\end{proof}


\section{Preliminaries on graded rings and modules}\label{sec2}

\noindent\textit{Let $G$ be a group and let $R$ be a $G$-graded ring.}\smallskip

While in projective geometry one meets $\Z$-graduations, we will have need of graduations by more general groups. Hence, and due to the lack of a suitable reference, we provide now some background on graded rings and modules. In \ref{sub2.1} we fix terminology and notation and look at some basic properties of categories of graded modules. We also describe graded versions of algebras of monoids (including polynomial algebras) and of rings and modules of fractions. In \ref{sub2.2} we study the behaviour of finiteness conditions under restriction of degrees. Finally, \ref{sub2.3} treats graded local cohomology. Besides some basic properties we give a general nonsense proof of the universal property of ideal transformation functors (cf.~\cite[2.2]{bs}).

\subsection{Graded rings and modules}\label{sub2.1}

\begin{no}\label{2.10}
We define the category $\grann^G$ of $G$-graded rings as in \cite[II.11.2]{a}. There is a faithful functor $U\colon\grann^G\rightarrow\ann$ that maps a $G$-graded ring onto its underlying ring; in case $G=0$ it is an isomorphism by means of which we identify $\grann^0$ and $\ann$. We define the category $\gralg^G(R)$ of $G$-graded $R$-algebras as the category $(\grann^G)^{/R}$ of $G$-graded rings under $R$.

We define the category $\grmod^G(R)$ of $G$-graded $R$-modules as the category with objects the $G$-graded $R$-modules and with morphisms the homomorphisms of $G$-graded $R$-modules of degree $0$ as in \cite[II.11.2]{a}. This category is abelian and fulfils Grothendieck's axioms AB5 and AB4$^*$, hence in particular has inductive and projective limits. There is a faithful, exact functor $U\colon\grmod^G(R)\rightarrow\catmod(U(R))$ that maps a $G$-graded $R$-module onto its underlying $U(R)$-module; in case $G=0$ it is an isomorphism by means of which we identify $\grmod^0(R)$ and $\catmod(R)$. The functor $U$ has a right adjoint and thus commutes with inductive limits and with finite projective limits.

Keeping in mind that $R_0$ is a subring of $U(R)$, taking homogeneous components yields functors $\bullet_g\colon\grmod^G(R)\rightarrow\catmod(R_0)$ for $g\in G$ that commute with inductive limits and with finite projective limits. For a $G$-graded $R$-module $M$ we set $M^{\hom}\dfgl\bigcup_{g\in G}M_g$ and denote by $\deg\colon M^{\hom}\setminus 0\rightarrow G$ its degree map. For $g\in G$ we denote by $\bullet(g)\colon\grmod^G(R)\rightarrow\grmod^G(R)$ the functor of shifting by $g$, which is an isomorphism.
\end{no}

\begin{no}\label{2.35}
For $G$-graded $R$-modules $M$ and $N$ we define a group $$\grhm{G}{R}{M}{N}\dfgl\bigoplus_{g\in G}\hm{\grmod^G(R)}{M}{N(g)}.$$ The structure of $U(R)$-module on $\hm{U(R)}{U(M)}{U(N)}$ induces a structure of $G$-graded $R$-module on $\grhm{G}{R}{M}{N}$ with $G$-graduation $(\hm{\grmod^G(R)}{M}{N(g)})_{g\in G}$. Note that $U(\grhm{G}{R}{M}{N})$ is a sub-$U(R)$-module of $\hm{U(R)}{U(M)}{U(N)}$, but not necessarily equal to the latter.

Varying $M$ and $N$ we get a contra-covariant bifunctor $$\grhm{G}{R}{\bullet}{\sq}\colon\grmod^G(R)^2\rightarrow\grmod^G(R)$$ that is left exact in both arguments.
\end{no}

\begin{no}\label{2.50}
The tensor product of $G$-graded $R$-modules as defined in \cite[II.11.5]{a} yields a bifunctor $$\bullet\otimes_R\sq\colon\grmod^G(R)^2\rightarrow\grmod^G(R)$$ that is right exact in both arguments. For a $G$-graded $R$-algebra $S$ we get functors $$\bullet\otimes_RS\colon\grmod^G(R)\rightarrow\grmod^G(S)$$ and $$\bullet\otimes_RS\colon\gralg^G(R)\rightarrow\gralg^G(S)$$ (\cite[II.11.5; III.4.8]{a}). We have $U(\bullet\otimes_R\sq)=U(\bullet)\otimes_{U(R)}U(\sq)$, and for $G$-graded $R$-modules $M$ and $N$ and $g\in G$ we have $$(M\otimes_RN)_g=\bigoplus\{M_h\otimes_{R_0}N_{h'}\mid h,h'\in G,h+h'=g\}.$$ For $g\in G$ there is a canonical isomorphism of functors $\bullet(g)\cong R(g)\otimes_R\bullet$. If $M$ is a $G$-graded $R$-module then $\bullet\otimes_RM$ is left adjoint to $\grhm{G}{R}{M}{\bullet}$ and there is a canonical isomorphism $\bullet\otimes_RM\cong M\otimes_R\bullet$. Thus, $\bullet\otimes_R\sq$ commutes with inductive limits in both arguments.

A $G$-graded $R$-module $M$ is called \textit{flat} if $M\otimes_R\bullet\colon\grmod^G(R)\rightarrow\grmod^G(R)$ is exact; this holds if and only if $U(M)$ is flat (\cite[A.I.2.18]{no1}). A $G$-graded $R$-algebra is called \textit{flat} if its underlying $G$-graded $R$-module is so.
\end{no}

\begin{no}\label{2.90}
Let $A$ be a ring and let $E$ be an $A$-module. If $\L\subseteq E$ is a subset then a sub-$A$-module of $E$ is called \textit{of (finite) $\L$-type} if it has a (finite) generating set contained in $\L$, and $E$ is called \textit{$\L$-noetherian} if the set of all its sub-$A$-modules of $\L$-type, ordered by inclusion, is noetherian, or -- equivalently -- if every sub-$A$-module of $E$ of $\L$-type is of finite $\L$-type. If $\L\subseteq A$ then the ring $A$ is called $\L$-noetherian if its underlying $A$-module is $\L$-noetherian.
\end{no}

\begin{no}\label{2.67}
For a family $E=(E_g)_{g\in G}$ of sets there exists a $G$-graded $R$-module $L(E)$ together with a map $\iota\colon\coprod_{g\in G}E_g\rightarrow L(E)$ such that for every $G$-graded $R$-module $M$ and every family of maps\linebreak $(u_g\colon E_g\rightarrow M_g)_{g\in G}$ there exists a unique morphism of $G$-graded $R$-modules $u\colon L(E)\rightarrow M$ such that $u\circ\iota$ and $u_g$ coincide on $E_g$ for every $g\in G$. Indeed, $L(E)\dfgl\bigoplus_{g\in G}R(g)^{\oplus E_g}$ together with the map induced by the canonical injections $E_g\rightarrow L(E)$ for $g\in G$ has this property. The pair $(L(E),\iota)$ is uniquely determined up to canonical isomorphism and is called \textit{the free $G$-graded $R$-module with basis $E$.}

A $G$-graded $R$-module $M$ is called \textit{free (of finite rank)} if there is a family of sets $E=(E_g)_{g\in G}$ with $M\cong L(E)$ (such that $\coprod_{g\in G}E_g$ is finite). If $M$ is free (of finite rank) then so is $U(M)$, but the converse does not necessarily hold.

Moreover, $M$ is called \textit{of finite type} (or \textit{of finite presentation}) if there is an exact sequence $F\rightarrow M\rightarrow 0$ (or $F'\rightarrow F\rightarrow M\rightarrow 0$) in $\grmod^G(R)$ with $F$ (and $F'$) free of finite rank. Clearly, $M$ is of finite type if and only if $U(M)$ is so, and this holds if and only if $U(M)$ is of finite $M^{\hom}$-type (\ref{2.90}). Hence, \cite[X.1.4 Propositon 6]{a} implies that $M$ is of finite presentation if and only if $U(M)$ is so. It is readily checked that $M$ is the inductive limit of its graded sub-$R$-modules of finite type.

Furthermore, $M$ is called \textit{pseudocoherent} if graded sub-$R$-modules of $M$ of finite type are of finite presentation, and \textit{coherent} if it is pseudocoherent and of finite type. By the above, (pseudo-)coherence of $U(M)$ implies (pseudo-)coherence of $M$. The $G$-graded ring $R$ is called \textit{coherent} if it is so considered as a $G$-graded $R$-module.

Finally, $M$ is called \textit{noetherian} if the set of all its graded sub-$R$-modules, ordered by inclusion, is noetherian, or -- equivalently -- if every graded sub-$R$-module of $M$ is of finite type. The $G$-graded ring $R$ is called \textit{noetherian} if it is so considered as a $G$-graded $R$-module. Clearly, $M$ is noetherian if and only if $U(M)$ is $M^{\hom}$-noetherian (\ref{2.90}). If $U(M)$ is noetherian then so is $M$, but the converse does not necessarily hold. However, it does hold if $G$ is of finite type (\cite{gotoyamagishi}).
\end{no}

\begin{no}\label{2.70}
There is a functor $R[\bullet]\colon\mon_{/G}\rightarrow\gralg^G(R)$ that maps a monoid $d\colon M\rightarrow G$ over $G$ to a $G$-graded $R$-algebra $R\rightarrow R[d]$, called \textit{the $G$-graded algebra of $d$ over $R$}, where $U(R[d])=U(R)[M]$ is the algebra of $M$ over $R$ and $$R[d]_g=\bigoplus_{h\in G}\bigoplus_{m\in d^{-1}(g-h)}(R_h\otimes_{R_0}R_0e_m)$$ for $g\in G$, denoting by $(e_m)_{m\in M}$ the canonical basis of $U(R)[M]$. Varying $R$ we get a bifunctor $$\bullet[\sq]\colon\grann^G\times\mon_{/G}\rightarrow\grann^G$$ under the first canonical projection of $\grann^G\times\mon_{/G}$. For a $G$-graded $R$-algebra $S$ there is a canonical isomorphism $\bullet[\sq]\otimes_RS\cong(\bullet\otimes_RS)[\sq]$.

In particular, if $I$ is a set then a map $d\colon I\rightarrow G$ corresponds to a unique monoid $d'\colon\N^{\oplus I}\rightarrow G$ over $G$, and we can consider the algebra $R[d']$ of $d'$ over $R$. Its underlying $U(R)$-algebra is the polynomial algebra in indeterminates $(X_i)_{i\in I}$ over $U(R)$, furnished with the $G$-graduation with $\deg(X_i)=d(i)$ for $i\in I$. The $G$-graded $R$-algebra $R[d']$ is denoted by $R[(X_i)_{i\in I},d]$; if $\card(I)=1$ and $d(I)=\{g\}$ it is denoted by $R[X,g]$. 

For a $G$-graded $R$-algebra $S$ there exist a set $I$, a map $d\colon I\rightarrow G$, and a surjective morphism of $G$-graded $R$-algebras $h\colon R[(X_i)_{i\in I},d]\rightarrow S$. The $G$-graded $R$-algebra $S$ is called \textit{of finite type} if there exist $I$, $d$ and $h$ as above such that $I$ is finite.

There is a graded version of Hilbert's Basissatz: if $R$ is noetherian then $G$-graded $R$-algebras of finite type are noetherian. (In order to prove this it suffices by the above to show that the $G$-graded polynomial algebra $S\dfgl R[X,g]$ with $g\in G$ is noetherian. Furnishing $R[X]$ also with its canonical $\Z$-graduation and denoting the corresponding total degree map by $\deg_{\Z}$, we find for a graded ideal $\ia\subseteq S$ that is not of finite type recursively a sequence $(f_i)_{i\in\N}$ in $\ia^{\hom}$ with $f_i\in\ia^{\hom}\setminus\langle f_0,\ldots,f_{i-1}\rangle_S$ and $\deg_{\Z}(f_i)$ minimal in $\deg_{\Z}(\ia^{\hom}\setminus\langle f_0,\ldots,f_{i-1}\rangle_S)$ for $i\in\N$, and then we conclude as in the ungraded case.)
\end{no}

\begin{no}\label{2.80}
Let $M$ be a $G$-graded $R$-module, let $N\subseteq M$ be a graded sub-$R$-module, and let $\ia\subseteq R$ be a graded ideal. Then, $(N:_M\ia)$ is a $G$-graded sub-$R$-module of $M$. Hence, $\Sat_M(N,\ia)\dfgl\bigcup_{n\in\N}(N:_M\ia^n)$ is a graded sub-$R$-module of $M$, called \textit{the $\ia$-saturation of $N$ in $M$,} and $N$ is called \textit{$\ia$-saturated in $M$} if $N=\Sat_M(N,\ia)$. If $\ia$ is of finite type then $\Sat_M(N,\ia)$ is the smallest graded sub-$R$-module of $M$ containing $N$ that is $\ia$-saturated in $M$.
\end{no}

\begin{no}\label{2.130}
The functor $\bullet_0\colon\grmod^G(R)\rightarrow\catmod(R_0)$ of taking components of degree $0$ has a left adjoint $R\otimes_{R_0}\bullet\colon\catmod(R_0)\rightarrow\grmod^G(R)$. The corresponding counit is the morphism of functors\linebreak $\nu\colon R\otimes_{R_0}(\bullet_0)\rightarrow\Id_{\grmod^G(R)}$ with $\nu(M)(r\otimes x)=rx$ for a $G$-graded $R$-module $M$, $r\in R$ and $x\in M_0$ (\cite[2.5.5]{no2}).

The $G$-graded ring $R$ is called \textit{strongly graded} if $R_{g+h}=\langle R_gR_h\rangle_{\Z}$ for all $g,h\in G$. This holds if and only if $\nu\colon R\otimes_{R_0}(\bullet_0)\rightarrow\Id_{\grmod^G(R)}$ is an isomorphism (\cite[A.I.3.4]{no1}).
\end{no}

\begin{no}\label{2.170}
The submonoid $\degmon(R)\dfgl\langle g\in G\mid R_g\neq 0\rangle_{\N}$ of $G$ is called \textit{the degree monoid of $R$.} If $A$ is a ring such that $R$ is a $G$-graded $A$-algebra and $E\subseteq R^{\hom}\setminus 0$ is a subset with $R=A[E]$, then $\degmon(R)=\langle\deg(x)\mid x\in E\rangle_{\N}$; if additionally $R\neq 0$ and $E$ contains no zerodivisors of $R$, then $\degmon(R)=\{g\in G\mid R_g\neq 0\}$. This applies in particular if $R=A[(X_i)_{i\in I},d]$ for some ring $A\neq 0$ and some map $d\colon I\rightarrow G$ (\ref{2.70}).

Furthermore, $R$ is said to be \textit{positively $G$-graded} if the monoid $\degmon(R)$ is sharp; then, $\degmon(R)$ is canonically furnished with a structure of ordered monoid (\ref{1.350}). The $R_0$-module $\bigoplus_{g\in\degmon(R)\setminus 0}R_g$ is a (graded) ideal of $R$ if and only if $R$ is positively $G$-graded, and then it is denoted by $R_+$.
\end{no}

\begin{no}\label{2.60}
Let $S\subseteq R^{\hom}$ be a subset. We denote by $\overline{S}$ the multiplicative closure of $S$ and set $$\widetilde{S}\dfgl\{r\in R^{\hom}\mid\exists r'\in R:rr'\in\overline{S}\}.$$ A subset $T\subseteq R^{\hom}$ is called \textit{saturated over $S$} if $S\subseteq\overline{T}\subseteq\widetilde{S}$, or -- equivalently -- if $S\subseteq\overline{T}$ and $\widetilde{S}=\widetilde{T}$.

We define a $G$-graded ring $S^{-1}R$ with underlying ring $S^{-1}(U(R))$ by setting $$\textstyle(S^{-1}R)_g=\{\frac{x}{s}\mid s\in\overline{S}\setminus 0,x\in R^{\hom}\setminus 0,\deg(x)=\deg(s)+g\}\cup\{0\}$$ for $g\in G$. For $T\subseteq R^{\hom}$ with $S\subseteq\overline{T}$ there is a canonical morphism of $G$-graded rings $\eta^S_T(R)\colon$\linebreak$S^{-1}R\rightarrow T^{-1}R$ with $\frac{x}{s}\mapsto\frac{x}{s}$ for $x\in R^{\hom}$ and $s\in S$; if $T$ is saturated over $S$ then this is an isomorphism of $G$-graded rings. Furthermore, $T^{-1}R$ is flat over $S^{-1}R$ by means of $\eta^S_T(R)$ (\ref{2.50}, \cite[II.2.3 Proposition 7; II.2.4 Th\'eor\`eme 1]{ac}). If $V\subseteq R^{\hom}$ such that $T$ is saturated over $V$ then there is a morphism of $G$-graded rings $(\eta^V_T(R))^{-1}\circ\eta^S_T(R)\colon S^{-1}R\rightarrow V^{-1}R$, denoted by abuse of language by $\eta^S_V(R)$.

By the above there is an exact functor $$\bullet\otimes_RS^{-1}R\colon\grmod^G(R)\rightarrow\grmod^G(S^{-1}R),$$ denoted by $S^{-1}\bullet$, and we have $U(S^{-1}\bullet)=S^{-1}(U(\bullet))$ (\ref{2.50}). Hence, for a $G$-graded $R$-module $M$ the $S^{-1}(U(R))$-module underlying $S^{-1}M$ is the module of fractions of $U(M)$ with denominators in $S$, and for $g\in G$ we have $$\textstyle(S^{-1}M)_g=\{\frac{x}{s}\mid s\in\overline{S}\setminus 0,x\in M^{\hom}\setminus 0,\deg(x)=\deg(s)+g\}\cup\{0\}.$$ For $T,V\subseteq R^{\hom}$ with $S,V\subseteq\overline{T}$ such that $T$ is saturated over $V$ we have a morphism of functors $$\eta^S_V\dfgl\bullet\otimes_R\eta^S_V(R)\colon S^{-1}\bullet\rightarrow V^{-1}\bullet.$$ We set $\eta_S\dfgl\eta_S^{\emptyset}\colon \Id_{\grmod^G(R)}\rightarrow S^{-1}\bullet$. Furthermore, we denote the compositions of $\eta_S^V$ and $\eta_S$ with the functor $\bullet_0$ of taking components of degree $0$ by $\eta_{(V)}^{(S)}$ and $\eta_{(S)}$. If $S=\{f\}$ we write $\eta_f$ and $\eta_{(f)}$ instead of $\eta_S$ and $\eta_{(S)}$; if moreover $V=\{g\}$ where $f$ divides $g$ we write $\eta^f_g$ and $\eta_{(g)}^{(f)}$instead of $\eta^S_V$ and $\eta_{(V)}^{(S)}$.

As $S^{-1}\bullet$ commutes with $U$ there is a canonical isomorphism of bifunctors $$(S^{-1}\bullet)\otimes_{S^{-1}R}(S^{-1}\sq)\cong S^{-1}(\bullet\otimes_R\sq)$$ (\cite[II.5.1 Proposition 3]{a}). The canonical injection $(S^{-1}\bullet)_0\hookrightarrow S^{-1}\bullet$ induces a morphism $$(S^{-1}\bullet)_0\otimes_{(S^{-1}R)_0}(S^{-1}\sq)_0\rightarrow (S^{-1}\bullet)\otimes_{S^{-1}R}(S^{-1}\sq).$$ Its composition with the above isomorphism induces a morphism $$\delta_S\colon(S^{-1}\bullet)_0\otimes_{(S^{-1}R)_0}(S^{-1}\sq)_0\cong S^{-1}(\bullet\otimes_R\sq)_0$$ of bifunctors from $\grmod^G(R)^2$ to $\catmod((S^{-1}R)_0)$. This is not necessarily an isomorphism (cf.~\ref{3.290}). However, as $S^{-1}\bullet$ and $\bullet_0$ commute with inductive limits, $\delta_S(R^{\oplus I},\bullet)$ is an isomorphism for every set $I$. Moreover, for $T,U\subseteq R^{\hom}$ with $S\subseteq\overline{T}$ such that $T$ is saturated over $U$ we have a commutative diagram $$\xymatrix@R15pt{(S^{-1}\bullet)_0\otimes_{(S^{-1}R)_0}(S^{-1}\sq)_0\ar[r]^(.6){\delta_S}\ar[d]_{(\eta_U^S)_0\otimes(\eta_U^S)_0}&S^{-1}(\bullet\otimes_R\sq)_0\ar[d]^{(\eta_U^S)_0}\\(U^{-1}\bullet)_0\otimes_{(U^{-1}R)_0}(U^{-1}\sq)_0\ar[r]^(.6){\delta_U}&U^{-1}(\bullet\otimes_R\sq)_0.}$$
\end{no}


\subsection{Degree restriction}\label{sub2.2}

\noindent\textit{Let $F\subseteq G$ be a subgroup. We denote by $(G:F)$ the index of $F$ in $G$.}

\begin{no}\label{2.20}
For an $F$-graded ring $S$ we define a $G$-graded ring $S^{(G)}$ with $U(S^{(G)})=U(S)$ by setting $S^{(G)}_g=S_g$ for $g\in F$ and $S^{(G)}_g=0$ for $g\in G\setminus F$. A morphism of $F$-graded rings $u\colon S\rightarrow T$ is a morphism of $G$-graded rings $S^{(G)}\rightarrow T^{(G)}$ and as such is denoted by $u^{(G)}$. So, we get a functor $\bullet^{(G)}\colon\grann^F\rightarrow\grann^G$, called \textit{$G$-extension.}

We define an $F$-graded ring $R_{(F)}$ where $U(R_{(F)})$ equals the subring $\bigoplus_{g\in F}R_g$ of $U(R)$ by setting $(R_{(F)})_g=R_g$ for $g\in F$. A morphism of $G$-graded rings $u\colon R\rightarrow S$ induces by restriction and coastriction a morphism of $F$-graded rings $u_{(F)}\colon R_{(F)}\rightarrow S_{(F)}$. So, we get a functor $\bullet_{(F)}\colon\grann^G\rightarrow\grann^F$, called \textit{$F$-restriction,} which is right adjoint to $\bullet^{(G)}$. We have $\degmon(R_{(F)})=\degmon(R)\cap F$. If $R$ is positively $G$-graded then $R_{(F)}$ is positively $F$-graded, and $(R_+)_{(F)}=(R_{(F)})_+$.

For a $G$-graded $R$-module $M$ we define an $F$-graded $R_{(F)}$-module $M_{(F)}$ with underlying $U(R_{(F)})$-module the sub-$U(R_{(F)})$-module $\bigoplus_{g\in F}M_g$ of $U(M)$ by setting $(M_{(F)})_g=M_g$ for $g\in F$. A morphism $u\colon M\rightarrow N$ of $G$-graded $R$-modules induces by restriction and coastriction a morphism of $F$-graded $R_{(F)}$-modules $u_{(F)}\colon M_{(F)}\rightarrow N_{(F)}$. So, we get a functor $\bullet_{(F)}\colon\grmod^G(R)\rightarrow\grmod^F(R_{(F)})$, also called \textit{$F$-restriction.}
\end{no}

\begin{no}\label{2.40}
Let $S\subseteq R^{\hom}$ with $\deg(s)\in F$ for every $s\in S\setminus 0$. Then, $S\subseteq(R_{(F)})^{\hom}$, and the $F$-graded $R_{(F)}$-algebras $(\eta_S(R))_{(F)}\colon R_{(F)}\rightarrow(S^{-1}R)_{(F)}$ and $\eta_S(R_{(F)})\colon R_{(F)}\rightarrow S^{-1}(R_{(F)})$ are canonically isomorphic and will henceforth be identified. Thus, we also identify the functors $$(S^{-1}\bullet)_{(F)}\colon\grmod^G(R)\rightarrow\grmod^F((S^{-1}R)_{(F)})$$ and $$S^{-1}(\bullet_{(F)})\colon\grmod^G(R)\rightarrow\grmod^F(S^{-1}(R_{(F)})).$$
\end{no}

\begin{lemma}\label{2.100}
If $M$ is a $G$-graded $R$-module and $N\subseteq M_{(F)}$ is a graded sub-$R_{(F)}$-module then $N=\langle N\rangle_R\cap M_{(F)}$.
\end{lemma}

\begin{proof}
If $f\in F$, $g\in G$, $x\in N$ and $r\in R_g$ with $rx\in M_{(F)}$ then $g+f=\deg(rx)\in F$, hence $g\in F$, and therefore $rx\in R_{(F)}N\subseteq N$. This shows $\langle N\rangle_R\cap M_{(F)}\subseteq N$. The other inclusion is obvious.
\end{proof}

\begin{prop}\label{2.110}
Let $M$ be a $G$-graded $R$-module and let $\L\subseteq M^{\hom}$.

a) If $M$ is $\L$-noetherian then $M_{(F)}$ is $\L\cap M_{(F)}$-noetherian.

b) If $N\subseteq M$ is a graded sub-$R$-module that is $\L$-generated as an $R_0$-module then $N_{(F)}$ is $\L\cap M_{(F)}$-generated as an $R_{(F)}$-module.
\end{prop}

\begin{proof}
Straightforward on use of \ref{2.100}.
\end{proof}

\begin{cor}\label{2.120}
Let $M$ be a noetherian $G$-graded $R$-module.

a) The $F$-graded $R_{(F)}$-module $M_{(F)}$ if noetherian.

b) If $g\in G$ then the $R_0$-module $M_g$ is noetherian.
\end{cor}

\begin{proof}
Applying \ref{2.110}~a) with $\mathbbm{L}=M^{\hom}$ yields a). Applying a) with $F=0$ to $M(g)$ yields b).
\end{proof}

\begin{prop}\label{2.160}
Suppose that $(G:F)<\infty$ and that $R$ is an $R_0$-algebra of finite type. Then, $R$ is an $R_{(F)}$-module of finite type, and if $M$ is a $G$-graded $R$-module of finite type then $M_{(F)}$ is an $F$-graded $R_{(F)}$-module of finite type.
\end{prop}

\begin{proof}
Analogously to \cite[III.1.3 Proposition 2(ii)]{ac}.
\end{proof}

\begin{prop}\label{2.180}
Let $R$ be positively $G$-graded and let $(x_i)_{i\in I}$ be a family in $R^{\hom}\setminus R_0$. We consider the following statements:
\begin{aufz}
\item[(1)] For $g\in G$ we have $R_g=\langle\prod_{i\in I}x_i^{n_i}\mid(n_i)_{i\in I}\in\N^{\oplus I}\wedge\deg(\prod_{i\in I}x_i^{n_i})=g\rangle_{R_0}$;
\item[(2)] $R=R_0[x_i\mid i\in I]$;
\item[(3)] $R_+=\langle x_i\mid i\in I\rangle_R$.
\end{aufz}
We have (1)$\Leftrightarrow$(2)$\Rightarrow$(3), and if $\degmon(R)$ is artinian then (1)--(3) are equivalent.
\end{prop}

\begin{proof}
The first claim is clear. Suppose that (3) holds and set $R'\dfgl R_0[x_i\mid i\in I]$. To show (1) it suffices to show that $R_g\subseteq R'$ for every $g\in\degmon(R)$. For $g=0$ this is clear. Let $g\in\degmon(R)\setminus 0$ and suppose $R_h\subseteq R'$ for every $h\in\degmon(R)$ with $h<g$. For $y\in R_g\subseteq R_+$ there exist a finite subset $J\subseteq I$ and a family $(r_i)_{i\in J}$ in $R^{\hom}\setminus 0$ with $y=\sum_{i\in J}r_ix_i$ and $\deg(r_i)=g-\deg(x_i)<g$ for $i\in J$, implying $r_i\in R'$ for $i\in J$ and thus $y\in R'$. So, as $\degmon(R)$ is artinian, (1) follows by noetherian induction (\cite[III.6.5 Proposition 7]{e}).\footnote{cf.~\cite[III.1.2 Proposition 1]{ac}}
\end{proof}

\begin{cor}\label{2.190}
Let $R$ be positively $G$-graded. If $R$ is an $R_0$-algebra of finite type then $R_+$ is an $R$-module of finite type, and if $\degmon(R)$ is artinian then the converse holds.
\end{cor}

\begin{proof}
Immediately from \ref{2.180}.
\end{proof}

\begin{cor}\label{2.200}
Let $R$ be positively $G$-graded and suppose that $(G:F)<\infty$. If $R$ is an $R_0$-algebra of finite type and $\degmon(R)\cap F$ is artinian, then $R_{(F)}$ is an $R_0$-algebra of finite type.
\end{cor}

\begin{proof}
Analogously to \cite[III.1.3 Proposition 2(i)]{ac} on use of \ref{2.160}, \ref{2.170} and \ref{2.190}.
\end{proof}


\subsection{Graded local cohomology}\label{sub2.3}

\begin{no}\label{2.150}
A $G$-graded $R$-module $M$ is projective if and only if $\grhm{G}{R}{M}{\bullet}$ is exact, and this is the case if and only if the $U(R)$-module $U(M)$ is projective. A $G$-graded $R$-module $M$ is injective if and only if $\grhm{G}{R}{\bullet}{M}$ is exact. If $U(M)$ is injective then so is $M$, but the converse does not necessarily hold (\cite[2.2--4]{no2}).

The category $\grmod^G(R)$ has enough projectives and injectives, since $\bigoplus_{g\in G}R(g)$ is a projective generator (\ref{2.10}, \cite[1.9.1; 1.10.1]{tohoku}).
\end{no}

\begin{no}\label{2.210}
Let $\ia\subseteq R$ be a graded ideal. The subfunctor $\grgam{G}{\ia}$ of $\Id_{\grmod^G(R)}$ with $\grgam{G}{\ia}(M)=\bigcup_{n\in\N}(0:_M\ia^n)$ for a $G$-graded $R$-module $M$ is left exact and is called \textit{the $G$-graded $\ia$-torsion functor} (cf.~\cite[1.1.6]{bs}). Clearly, $U\circ\grgam{G}{\ia}=\grgam{0}{U(\ia)}$ equals the (usual) $\ia$-torsion functor $\Gamma_{\ia}$. There is a canonical isomorphism of functors $\grgam{G}{\ia}(\bullet)\cong\ilim_{n\in\N}\grhm{G}{R}{R/\ia^n}{\bullet}$ by means of which we identify them (cf.~\cite[1.2.11]{bs}). A $G$-graded $R$-module $M$ is called an \textit{$\ia$-torsion module} if $M=\grgam{G}{\ia}(M)$.
\end{no}

\begin{no}\label{2.220}
Let $\ia\subseteq R$ be a graded ideal. We say that $R$ \textit{has ITI with respect to $\ia$} if  $\grgam{G}{\ia}(I)$ is injective for every injective $G$-graded $R$-module $I$. This holds if and only if every $G$-graded $\ia$-torsion $R$-module has an injective resolution whose components are $\ia$-torsion modules. If $R$ is noetherian then it has ITI with respect to every ideal, but the converse is not true. Moreover, there exist (necessarily non-noetherian) rings without ITI with respect to some ideal of finite type (\cite{qr}).
\end{no}

\begin{no}\label{2.230}
For a $G$-graded $R$-module $M$ we consider the right derived cohomological functor\linebreak $(\grext{G}{i}{R}{M}{\bullet})_{i\in\Z}$ of $\grhm{G}{R}{M}{\bullet}$ (\ref{2.10}, \ref{2.35}, \ref{2.150}). Note that the graded covariant Ext functors do not necessarily commute with $U$ (\cite{coarsening}).

Now, let $F$ be a projective system in $\grmod^G(R)$ over a right filtering ordered set $J$, and let $i\in\Z$. Composing $F$ with the contravariant functor from $\grmod^G(R)$ to the category of endofunctors of $\grmod^G(R)$ that maps $M$ onto $\grext{G}{i}{R}{M}{\bullet}$ yields an inductive system of endofunctors of $\grmod^G(R)$ over $J$. As the category of endofunctor of $\grmod^G(R)$ is abelian and fulfils AB5 (\ref{2.10}, \cite[1.6.1; 1.7 d)]{tohoku}) we can consider the inductive limit $\ilim_J\grext{G}{i}{R}{F}{\bullet}$ in this category. Since $J$ is right filtering it follows from AB5 that $\ilim_J\grext{G}{0}{R}{F}{\bullet}=\ilim_J\grhm{G}{R}{F}{\bullet}$ is left exact, and hence $(\ilim_J\grext{G}{i}{R}{F}{\bullet})_{i\in\Z}$ is its right derived cohomological functor.
\end{no}

\begin{no}\label{2.240}
Let $\ia\subseteq R$ be a graded ideal. Applying the construction from \ref{2.230} to the projective systems $(R/\ia^n)_{n\in\N}$ and $(\ia^n)_{n\in\N}$ in $\grmod^G(R)$ over the right filtering ordered set $\N$ and setting $\grloc{G}{i}{\ia}(\bullet)\dfgl\ilim_{n\in\N}\grext{G}{i}{R}{R/\ia^n}{\bullet}$ and $\gridt{G}{i}{\ia}(\bullet)\dfgl\ilim_{n\in\N}\grext{G}{i}{R}{\ia^n}{\bullet}$ for $i\in\Z$ we get universal $\delta$-functors $(\grloc{G}{i}{\ia})_{i\in\Z}$ and $(\gridt{G}{i}{\ia})_{i\in\Z}$ from $\grmod^G(R)$ to itself that are the right derived cohomological functors of the $G$-graded $\ia$-torsion functor $\grgam{G}{\ia}$ (\ref{2.210}) and the so-called \textit{$G$-graded $\ia$-transformation functor} $\gride{G}{\ia}\dfgl\gridt{G}{0}{\ia}$. For $i\in\Z$ we call $\grloc{G}{i}{\ia}$ and $\gridt{G}{i}{\ia}$ \textit{the $i$-th $G$-graded local cohomology functor with respect to $\ia$} and \textit{the $i$-th $G$-graded ideal transformation functor with respect to $\ia$,} respectively.
\end{no}

\begin{prop}\label{2.250}
Let $\ia\subseteq R$ be a graded ideal.

a) There are an exact sequence of functors $$0\rightarrow\grgam{G}{\ia}\hookrightarrow\Id_{\grmod^G(R)}\xrightarrow{\eta_{\ia}}\gride{G}{\ia}\xrightarrow{\zeta_{\ia}}\grloc{G}{1}{\ia}\rightarrow 0$$ and a unique morphism of $\delta$-functors $(\zeta^i_{\ia})_{i\in\Z}\colon(\gridt{G}{i}{\ia})_{i\in\Z}\rightarrow(\grloc{G}{i+1}{\ia})_{i\in\Z}$ with $\zeta^0_{\ia}=\zeta_{\ia}$, and if $i\in\N^*$ then $\zeta^i_{\ia}$ is an isomorphism.

b) Suppose that $R$ has ITI with respect to $\ia$. Then, $\grloc{G}{i}{\ia}\circ\grgam{G}{\ia}=0$ for $i\in\N^*$ and $\gride{G}{\ia}\circ\grgam{G}{\ia}=\grgam{G}{\ia}\circ\gride{G}{\ia}=0$, and $\gride{G}{\ia}\circ\eta_{\ia}=\eta_{\ia}\circ\gride{G}{\ia}$ is an isomorphism.
\end{prop}

\begin{proof}
Analogously to \cite[2.2.4; 2.1.7; 2.2.8]{bs}.
\end{proof}

\begin{lemma}\label{d10}
Let $\C$ and $\D$ be abelian categories, let $F\colon\C\rightarrow\D$ be a functor, and let $(T^0,T^1)$ be an exact $\delta$-functor from $\C$ to $\D$ with $T^0\circ F=T^1\circ F=0$. If $a$ is a morphism in $\C$ with $F(\ke(a))=\ke(a)$ and $F(\cok(a))=\cok(a)$ then $T^0(a)$ is an isomorphism.
\end{lemma}

\begin{proof}
Let $A$ and $B$ denote source and target of $a$, respectively, and let $a=a''\circ a'$ be the canonical factorisation of $a$ over its image. The exact sequences $$0=T^0(F(\ke(a)))\longrightarrow T^0(A)\xrightarrow{T^0(a')}T^0(\im(a))\longrightarrow T^1(F(\ke(a)))=0$$ and $$0\longrightarrow T^0(\im(a))\xrightarrow{T^0(a'')}T^0(B)\longrightarrow T^0(F(\cok(a)))=0$$ show that $T^0(a')$ and $T^0(a'')$ are isomorphisms, and thus $T^0(a)=T^0(a'')\circ T^0(a')$ is an isomorphism, too.
\end{proof}

\begin{lemma}\label{d20}
Let $\C$ be a category, let $T\colon\C\rightarrow\C$ be a functor, let $\eta\colon\Id_{\C}\rightarrow T$ be a morphism of functors such that $\eta\circ T=T\circ\eta$ is an isomorphism, and let $a\in\hm{\C}{A}{B}$ and $b\in\hm{\C}{A}{C}$ such that $T(a)$ is an isomorphism. There exists a unique morphism $b\colon B\rightarrow T(C)$ in $\C$ with $b\circ a=\eta(C)\circ c$, and $b=T(c)\circ T(a)^{-1}\circ\eta(B)$; if $c$ and $\eta(B)$ are isomorphisms then so is $b$.
\end{lemma}

\begin{proof}
Setting $b\dfgl T(c)\circ T(a)^{-1}\circ\eta(B)$ we get $$b\circ a=T(c)\circ T(a)^{-1}\circ\eta(B)\circ a=T(c)\circ T(a)^{-1}\circ T(a)\circ\eta(A)=T(c)\circ\eta(A)=\eta(C)\circ c,$$ hence $b$ exists. For $b'\in\hm{\C}{B}{T(c)}$ with $b'\circ a=\eta(C)\circ c$ we get $T(b')\circ T(a)=T(\eta(C))\circ T(c)$, hence $T(b')=T(\eta(C))\circ T(c)\circ T(a)^{-1}$, and $b'=\eta(T(C))^{-1}\circ T(b')\circ\eta(B)$, therefore $$b'=T(\eta(C))^{-1}\circ T(b')\circ\eta(B)=T(c)\circ T(a)^{-1}\circ\eta(B)=b.$$ The second statement is clear.
\end{proof}

\begin{prop}\label{d40}
Let $\ia\subseteq R$ be a graded ideal such that $R$ has ITI with respect to $\ia$. Let $a\colon A\rightarrow B$ and $c\colon A\rightarrow C$ be morphisms in $\grmod^G(R)$ with $\grgam{G}{\ia}(\ke(a))=\ke(a)$ and $\grgam{G}{\ia}(\cok(a))=\cok(a)$. Then, $\gride{G}{\ia}(a)$ is an isomorphism, there is a unique morphism $b\colon B\rightarrow\gride{G}{\ia}(C)$ in $\grmod^G(R)$ with $b\circ a=\eta_{\ia}(C)\circ c$, and $b=\gride{G}{\ia}(c)\circ\gride{G}{\ia}(a)^{-1}\circ\eta_{\ia}(B)$; if $c$ and $\eta_{\ia}(B)$ are isomorphisms then so is $b$.
\end{prop}

\begin{proof}
Setting $T=\gride{G}{\ia}$, $F=\grgam{G}{\ia}$ and $\eta=\eta_{\ia}$ this follows from \ref{d10} and \ref{d20} on use of \ref{2.250}~b).
\end{proof}

\begin{prop}\label{d30}
Let $\C$ and $\D$ be categories, let $T\colon\C\rightarrow\C$, $S\colon\D\rightarrow\D$ and $F\colon\D\rightarrow\C$ be functors, and let $\eta\colon\Id_{\C}\rightarrow T$ and $\eps\colon\Id_{\D}\rightarrow S$ be morphisms of functors such that $\eta\circ T=T\circ\eta$ and $T\circ F\circ\eps$ are isomorphisms. Then, there exists a unique morphism of functors $\eps'\colon F\circ S\rightarrow T\circ F$ with $\eps'\circ(F\circ\eps)=\eta\circ F$. Moreover, $\eps'$ is a mono-, epi- or isomorphism if and only if $(\eta\circ F)\circ S$ has the same property.
\end{prop}

\begin{proof}
For an object $A$ of $\D$, applying \ref{d20} with $a=F(\eps(A))$ and $c=\Id_{F(A)}$ yields a unique morphism $\eps'(A)\colon F(S(A))\rightarrow T(F(A))$ in $\C$ with $\eps'(A)\circ F(\eps(A))=\eta(F(A))$, and $$\eps'(A)=T(F(\eps(A)))^{-1}\circ\eta(F(S(A))).$$ For $a\in\hm{\D}{A}{B}$ we get a diagram $$\xymatrix@R15pt@C60pt{F(A)\ar[rr]^{\eta(F(A))}\ar[rd]_{F(\eps(A))}\ar[dd]_{F(a)}&&T(F(A))\ar[dd]_(.65){T(F(a))}\ar[rd]^{T(F(\eps(A)))}&\\&F(S(A))\ar[rr]^(.35){\eta(F(S(A)))}\ar[dd]_(.65){F(S(a))}\ar[ru]^{\eps'(A)}&&T(F(S(A)))\ar[dd]^{T(F(S(a)))}\\F(B)\ar[rr]^(.35){\eta(F(B))}\ar[rd]_{F(\eps(B))}&&T(F(B))\ar[rd]^{T(F(\eps(B)))}&\\&F(S(B))\ar[rr]^{\eta(F(S(B)))}\ar[ru]_{\eps'(B)}&&T(F(S(B)))}$$ whose outer faces commute obviously. As $T(F(\eps(B)))$ is an isomorphism it follows that the whole diagram commutes, so the morphisms $\eps'(A)$ are natural in $A$, and thus the claim is proven.
\end{proof}

\begin{cor}\label{d50}
Let $\ia\subseteq R$ be a graded ideal such that $R$ has ITI with respect to $\ia$. Let \linebreak$S\colon\grmod^G(R)\rightarrow\grmod^G(R)$ be a functor, and let $\eps\colon\Id_{\grmod^G(R)}\rightarrow S$ be a morphism of functors such that $\grgam{G}{\ia}\circ\ke(\eps)=\ke(\eps)$ and $\grgam{G}{\ia}\circ\cok(\eps)=\cok(\eps)$. Then, $\gride{G}{\ia}\circ\eps$ is an isomorphism, there is a unique morphism of functors $\eps'\colon S\rightarrow\gride{G}{\ia}$ with $\eps'\circ\eps=\eta_{\ia}$, and $\eps'=(\gride{G}{\ia}\circ\eps)^{-1}\circ(\eta_{\ia}\circ S)$. Moreover, $\eps'$ is a mono-, epi- or isomorphism if and only if $\eta_{\ia}\circ S$ has the same property, and this holds if and only if $\grgam{G}{\ia}\circ S=0$, $\grloc{G}{1}{\ia}\circ S=0$, or $\grgam{G}{\ia}\circ S=\grloc{G}{1}{\ia}\circ S=0$, respectively.
\end{cor}

\begin{proof}
This follows from \ref{d30} (with $T=\gride{G}{\ia}$ and $F=\Id_{\grmod^G(R)}$), \ref{2.250}~b) and \ref{d40}.
\end{proof}


\section{Cox rings, Cox schemes, and toric schemes}\label{sec3}

\noindent\textit{Throughout this section let $V$ be an $\R$-vector space with $n\dfgl\rk_{\R}(V)\in\N$, let $N$ be a $\Z$-structure on $V$, let $M\dfgl N^*$, let $\sig$ be an $N$-fan, and let $B\subseteq A_{\sig_1}$ be a subgroup. If no confusion can arise we write $A$, $a$ and $c$ instead of $A_{\sig_1}$, $a_{\sig_1}$ and $c_{\sig_1}$ (\ref{1.140}).}\smallskip

In \ref{sub3.1} we define and study the Cox ring associated with $\sig$. Beware that our notion of Cox ring differs from and must not be confused with the one of the same name introduced in \cite{hukeel} (although both notions are inspired by \cite{cox}). In \ref{sub3.2} and \ref{sub3.3}, Cox schemes and toric schemes are defined as special cases of a far more general construction of schemes from certain projective systems of monoids (treated in detail in \cite{geometry}). Then, the Cox scheme associated with $\sig$ is shown to be isomorphic to the toric scheme associated with $\sig$ if and only if $\sig$ is not contained in a hyperplane of the ambient space. Together with the base change property of toric schemes this allows to study Cox schemes instead of toric schemes, which is an advantage since Cox schemes have an affine open covering given by degree $0$ parts of rings of fractions of some graded ring -- similar to projective schemes.

In \cite{cox} and \cite{mustata}, the Cox rings are graded by $A$ or, in case $\sig$ is simplicial, by $\pic(\sig)$. To treat these two cases at once and to reveal what is really needed, we develop our theory for an arbitrary subgroup $B\subseteq A$ and impose further conditions on $B$ -- as being big or small (\ref{1.200}) -- only if we need them.

\subsection{Cox rings}\label{sub3.1}

\begin{no}\label{3.10}
By restriction, $a\colon\Z^{\sig_1}\rightarrow A$ induces a morphism of monoids $a\res_{\N^{\sig_1}}\colon\N^{\sig_1}\rightarrow A$. We have a diagram of categories $$\xymatrix@C40pt@R0pt{&&\ar@{=>}[dd]&&\\\ann\ar[r]^(.45){\bullet^{(A)}}&\grann^A\ar@/^3ex/[rr]^{\Id_{\grann^A}}\ar@/_3ex/[rr]_{\bullet[a\,\res_{\N^{\sig_1}}]}&&\grann^A\ar[r]^{\bullet_{(B)}}&\grann^B\\&&&&}$$ (\ref{2.70}, \ref{2.20}). Setting $S_{\sig_1,B}(\bullet)\dfgl\bullet^{(A)}[a\res_{\N^{\sig_1}}]_{(B)}$ we get a diagram of categories $$\xymatrix@C40pt@R0pt{&\ar@<1ex>@{=>}[dd]&\\\ann\ar@/^3ex/[rr]^{\bullet^{(B)}}\ar@/_3ex/[rr]_{S_{\sig_1,B}(\bullet)}&&\grann^B.\\&&}$$ If no confusion can arise we write $S_B$ instead of $S_{\sig_1,B}$, and if $\alpha\in B$ then $S_B(\bullet)_{\alpha}=S_A(\bullet)_{\alpha}$ is henceforth denoted by $S(\bullet)_{\alpha}$. For a ring $R$ the above yields a functor $S_B\colon\alg(R)\rightarrow\gralg^B(R^{(B)})$ (\ref{2.10}), and for an $R$-algebra $R'$ there is a canonical isomorphism $$S_B(\bullet)\otimes_RR'\cong S_B(\bullet\otimes_RR')$$ in $\hmf{\alg(R)}{\gralg^B((R')^{(B)})}$ (\ref{2.70}).
\end{no}

\begin{no}\label{3.20}
Let $R$ be a ring. The $B$-graded ring $S_B(R)$ is called \textit{the $B$-restricted Cox ring over $R$ associated with $\sig_1$.} The underlying $R$-algebra of $S_A(R)$ is the polynomial algebra over $R$ in indeterminates $(Z_{\rho})_{\rho\in\sig_1}$, and its $A$-graduation is given by $\deg(Z_{\rho})=\alpha_{\rho}$ for $\rho\in\sig_1$ (\ref{1.140}); we denote its set of monomials\footnote{i.e., products of indeterminates} by $\T_R$. Hence, the underlying $R$-algebra of $S_B(R)$ is the sub-$R$-algebra of the polynomial algebra over $R$ in indeterminates $(Z_{\rho})_{\rho\in\sig_1}$ generated by the monomials $\prod_{\rho\in\sig_1}Z_{\rho}^{m_{\rho}}$ with degree $\sum_{\rho\in\sig_1}m_{\rho}\alpha_{\rho}$ in $B$, i.e., by $\T_{B,R}\dfgl\T_R\cap S_B(R)$. For a morphism of rings $h\colon R\rightarrow R'$ the morphism of $B$-graded rings $S_B(h)\colon S_B(R)\rightarrow S_B(R')$ is given by $t\mapsto t$ for $t\in\T_{B,R}$ (\ref{2.70}).

For $\sigma\in\sig$ we set $\zs\dfgl\prod_{\rho\in\sig_1}Z_{\rho}\in S_A(R)$, so that $\deg(\zs)=\alp_{\sigma}$ (\ref{1.140}). The graded ideal $I_{\sig,B}(R)\dfgl(\langle\zs\mid\sigma\in\sig\rangle_{S_A(R)})_{(B)}\subseteq S_B(R)$ is called \textit{the $B$-restricted irrelevant ideal over $R$ associated with $\sig$;} if no confusion can arise we denote it by $I_B(R)$.

If $R\neq 0$ then $\degmon(S_B(R))=\degmon(S_A(R))\cap B=\degmon(\sig_1)\cap B$ (\ref{1.380}, \ref{2.170}, \ref{2.20}), so $\degmon(S_B(R))$ is independent of $R$.
\end{no}

\begin{prop}\label{3.30}
Let $R$ be a ring.

a) $S_B(R)$ is $\T_{B,R}$-noetherian and $I_B(R)$ is finitely $\T_{B,R}$-generated.\footnote{cf.~\ref{2.90}}

b) The following statements are equivalent:\/\footnote{Since $A$ and $B$ are of finite type it does not matter whether we understand noetherianness as a property of graded  or of ungraded rings (\ref{2.67}).}
\begin{equi}
\item[(i)] $S_A(R)$ is noetherian;
\item[(ii)] $S_B(R)$ is noetherian;
\item[(iii)] $S(R)_0$ is noetherian.
\end{equi}
\end{prop}

\begin{proof}
Monomial ideals in polynomial algebras in finitely many indeterminates are generated by finitely many monomials by Dickson's Lemma (\cite[1.3.6]{kr}), hence a) follows from \ref{2.110}~b). The implications ``(i)$\Rightarrow$(ii)$\Rightarrow$(iii)'' in b) follow from \ref{2.120}. The remaining implication follows from Hilbert's Basissatz since $S_A(R)$ is of finite type over $R$ and hence of finite type over $S(R)_0$.
\end{proof}

\begin{prop}\label{3.40}
The following statements are equivalent:
\begin{equi}
\item[(i)] $\sig$ is relatively skeletal complete;
\item[(ii)] $S(\bullet)_0=\Id_{\ann}(\bullet)$;
\item[(iii)] There is a ring $R\neq 0$ with $S(R)_0=R$.
\end{equi}
\end{prop}

\begin{proof}
As $S(\bullet)_0=S_A(\bullet)_0=\bullet[\ke(a\res_{\N^{\sig_1}})]=\bullet[\im(c)\cap\N^{\sig_1}]$, (i) and (ii) are equivalent by \ref{1.390}. If $R$ is a ring with $S(R)_0=R\neq 0$, then the $R$-module underlying $S(R)_0=R[\ke(a\res_{\N^{\sig_1}})]$ is free of rank $1$, implying $\im(c)\cap\N^{\sig_1}=\ke(a\res_{\N^{\sig_1}})=0$ and thus (i) (\ref{1.390}).
\end{proof}

\begin{cor}\label{3.45}
Let $R$ be a ring. If $R$ is noetherian then so is $S_B(R)$, and if $\sig$ is relatively skeletal complete then the converse holds.
\end{cor}

\begin{proof}
Immediately from Hilbert's Basissatz, \ref{3.30}~b) and \ref{3.40}.
\end{proof}

\begin{prop}\label{3.60}
Let $R$ be a ring. If $B$ is big and $\sig$ is relatively skeletal complete then $S_B(R)$ is a positively $B$-graded $R$-algebra of finite type.
\end{prop}

\begin{proof}
If $R=0$ this is clear, so we suppose $R\neq 0$. Since $\degmon(\sig_1)$ is a sharp monoid (\ref{1.390}) so is its submonoid $\degmon(\sig_1)\cap B=\degmon(S_B(R))$ (\ref{3.20}), and thus $S_B(R)$ is positively $B$-graded. In particular (considering $B=A$), $S_A(R)$ is a positively $A$-graded $S_A(R)_0$-algebra of finite type. Thus, artinianness of $\degmon(S_B(R))$ implies the claim (\ref{1.350}, \ref{2.200}, \ref{3.40}).
\end{proof}

\begin{nosm}\label{3.70}
Let $R$ be a ring. If $S_B(R)$ is a positively $B$-graded $R$-algebra of finite type then $\sig$ is not necessarily skeletal complete, even if $B=A$. A counterexample is obtained immediately from the first example in \ref{1.400}.
\end{nosm}


\subsection{Cox schemes}\label{sub3.2}

\begin{no}\label{3.80}
Let $\sigma\in\sig$, and suppose there is an $m\in\N$ with $m\alp_{\sigma}\in B$. The functor $$S_B(\bullet)_{\zs^m}\colon\ann\rightarrow\grann^B$$ under $\bullet^{(B)}$ and the canonical morphism $\eta_{\zs^m}\colon S_B(\bullet)\rightarrow S_B(\bullet)_{\zs^m}$ in $\hmf{\ann}{\grann^B}^{/\bullet^{(B)}}$ (\ref{2.60}) are independent of the choice of $m$; we denote them by $S_{B,\sigma}(\bullet)$ and $\eta_{\sigma}$, respectively. We have $S_{B,\sigma}(\bullet)=S_{A,\sigma}(\bullet)_{(B)}$ (\ref{2.40}), and $S_B(\bullet)(\alpha)_{(\sigma)}=S_{B,\sigma}(\bullet)_{\alpha}=S_{A,\sigma}(\bullet)_{\alpha}=S_A(\bullet)(\alpha)_{(\sigma)}$ for $\alpha\in B$ is henceforth denoted just by $S(\bullet)(\alpha)_{(\sigma)}$. For $\alpha=0$ we have in particular a functor $$S(\bullet)_{(\sigma)}\dfgl S_{B,\sigma}(\bullet)_0\colon\ann\rightarrow\ann$$ under $\Id_{\ann}$, independent of $B$, which we henceforth denote by $S_{(\sigma)}(\bullet)$, and a canonical morphism $\eta_{(\sigma)}\colon S(\bullet)_0\rightarrow S_{(\sigma)}(\bullet)$ in $\hmf{\ann}{\ann}^{/\Id_{\ann}}$. For a ring $R$ we have functors $$\bullet_{\sigma}\dfgl\bullet\otimes_{S_B(R)}S_{B,\sigma}(R)\colon\grmod^B(S_B(R))\rightarrow\grmod^B(S_{B,\sigma}(R))$$ (\ref{2.60}) and $$\bullet_{(\sigma)}\dfgl(\bullet_{\sigma})_0\colon\grmod^B(S_B(R))\rightarrow\catmod(S_{(\sigma)}(R)).$$

Now, let $\tau\fleq\sigma$, and suppose that $m\alp_{\tau}\in B$. Then, $\zs^m$ divides $\zt^m$ in $S_B(R)$ for every ring $R$. Hence, mapping a ring $R$ onto the canonical morphism of $B$-graded $R$-algebras $$\eta^{\zs^m}_{\zt^m}(S_B(R))\colon S_B(R)_{\zs^m}\rightarrow S_B(R)_{\zt^m},$$ which is independent of the choice of $m$, yields a morphism $$\eta^{\sigma}_{\tau}\colon S_{B,\sigma}(\bullet)\rightarrow S_{B,\tau}(\bullet)$$ in $\hmf{\ann}{\grann^B}^{/\bullet^{(B)}}$. Composition with $\bullet_0$ yields a morphism $$\eta^{(\sigma)}_{(\tau)}\colon S_{(\sigma)}(\bullet)\rightarrow S_{(\tau)}(\bullet)$$ in $\hmf{\ann}{\ann}^{/\Id_{\ann}}$. Moreover, for a ring $R$ we have functors $$\eta^{\sigma}_{\tau}(R)^*\colon\grmod^B(S_{B,\sigma}(R))\rightarrow\grmod^B(S_{B,\tau}(R))$$ and $$\eta^{(\sigma)}_{(\tau)}(R)^*\colon\catmod(S_{(\sigma)}(R))\rightarrow\catmod(S_{(\tau)}(R))$$ such that the diagram of categories $$\xymatrix@R15pt{\grmod^B(S_B(R))\ar[r]^(.47){\bullet_{\sigma}}\ar[rd]_{\bullet_{\tau}}&\grmod^B(S_{B,\sigma}(R))\ar[r]^(.53){\bullet_0}\ar[d]^{\eta^{\sigma}_{\tau}(R)^*}&\catmod(S_{(\sigma)}(R))\ar[d]^{\eta^{\sigma}_{\tau}(R)^*}\\&\grmod^B(S_{B,\tau}(R))\ar[r]^(.53){\bullet_0}&\catmod(S_{(\tau)}(R))}$$ commutes.

Finally, if $B$ is big then there exists $m\in\N$ such that for every $\sigma\in\sig$ we have $m\alp_{\sigma}\in B$, and thus the above gives rise to a projective system $((S_{(\sigma)}(\bullet))_{\sigma\in\sig},(\eta^{(\sigma)}_{(\tau)})_{\tau\fleq\sigma})$ in $\hmf{\ann}{\ann}^{/\Id_{\ann}}$.
\end{no}

\begin{no}\label{3.100}
Let $I$ be a lower semilattice and let $\mathbbm{M}=((M_i)_{i\in I},(p_{ij})_{i\leq j})$ be a projective system in $\mon$ over $I$ such that for $i,j\in I$ with $i\leq j$ there is a $t_{ij}\in M_j$ with\footnote{If $E$ is a monoid and $t\in E$ then we denote by $E-t$ the monoid of differences with negatives the multiples of $t$, and by $\eps_t\colon E\rightarrow E-t$ the canonical epimorphism.} $M_i=M_j-t_{ij}$ and $p_{ij}=\eps_{t_{ij}}$. For $i\in I$ we set $X_{\mathbbm{M},i}(\bullet)\dfgl\spec(\bullet[M_i])\colon\ann^{\circ}\rightarrow\sch$ and denote by $t_{\mathbbm{M},i}(\bullet)\colon X_{\mathbbm{M},i}\rightarrow\spec$ the canonical morphism of contravariant functors. For $i,j\in I$ with $i\leq j$ we set $\iota_{\mathbbm{M},i,j}(\bullet)\dfgl\spec(\bullet[p_{ij}])\colon X_{\mathbbm{M},i}(\bullet)\rightarrow X_{\mathbbm{M},j}(\bullet)$; by our hypothesis this is an open immersion.

If $R$ is a ring then the inductive system $$((X_{\mathbbm{M},i}(R)\xrightarrow{t_{\mathbbm{M},i}(R)}\spec(R))_{i\in I},(\iota_{\mathbbm{M},i,j})_{i\leq j})$$ in $\sch_{/R}$ has an inductive limit $t_{\mathbbm{M}}(R)\colon X_{\mathbbm{M}}(R)\rightarrow\spec(R)$ in $\sch_{/R}$ such that for every $i\in I$ the canonical morphism $\iota_{\mathbbm{M},i}(R)\colon X_{\mathbbm{M},i}(R)\rightarrow X_{\mathbbm{M}}(R)$ is an open immersion by means of which we consider $X_{\mathbbm{M},i}(R)$ as an open subscheme of $X_{\mathbbm{M}}(R)$. The inductive limit $X_{\mathbbm{M}}(R)$ can then be understood as obtained by gluing the family $(X_{\mathbbm{M},i}(R))_{i\in I}$ along $(X_{\mathbbm{M},\inf(i,j)}(R))_{(i,j)\in I^2}$.

The above gives rise to a functor $X_{\mathbbm{M}}\colon\alg(R)^{\circ}\rightarrow\sch_{/R}$ over $\spec$ together with open immersions $\iota_{\mathbbm{M},i}\colon X_{\mathbbm{M},i}\rightarrow X_{\mathbbm{M}}$. For an $R$-algebra $R'$ we have $X_{\mathbbm{M}}(\bullet\otimes_RR')\cong X_{\mathbbm{M}}(\bullet)\otimes_RR'$.\footnote{cf.~\cite{geometry} for an extensive study of geometric properties of these functors}
\end{no}

\begin{no}\label{3.110}
For $\tau\fleq\sigma\in\sig$ and $u\in\sigma^{\vee}$ with $\tau=\sigma\cap u^{\perp}$ we have $c(u)\in\N^{\sig_1}-\del_{\sigma}$ and $(\N^{\sig_1}-\del_{\sigma})-c(u)=\N^{\sig_1}-\del_{\tau}$. As $c(u)\in\ke(a)$ this implies $$(\N^{\sig_1}-\del_{\sigma})\cap\ke(a)-c(u)=(\N^{\sig_1}-\del_{\tau})\cap\ke(a).$$ Therefore, the family $((\N^{\sig_1}-\del_{\sigma})\cap\ke(a))_{\sigma\in\sig}$ defines a projective system of submonoids of $\Z^{\sig_1}$ fulfilling the hypothesis of \ref{3.100}; we denote it by $C_{\sig}$. By \ref{3.100} we get a functor $Y_{\sig}\dfgl X_{C_{\sig}}\colon\ann^{\circ}\rightarrow\sch$ over $\spec$, mapping a ring $R$ onto an $R$-scheme $u_{\sig}(R)\colon Y_{\sig}(R)\rightarrow\spec(R)$ -- called \textit{the Cox scheme over $R$ associated with $\sig$} -- and a finite affine open covering $(Y_{\sigma})_{\sigma\in\sig}$ of $Y_{\sig}$, where $Y_{\sigma}(\bullet)=\spec(\bullet[(\N^{\sig_1}-\del_{\sigma})\cap\ke(a)])$.

By the above we have a projective system $((\bullet[(\N^{\sig_1}-\del_{\sigma})\cap\ke(a)])_{\sigma\in\sig},(\eps^{\sigma}_{\tau})_{\tau\fleq\sigma})$ in $\hmf{\ann}{\ann}^{/\Id_{\ann}}$ over $\sig$, where $\eps^{\sigma}_{\tau}$ is obtained by restriction and coastriction from the canonical injection\linebreak $\N^{\sig_1}-\del_{\sigma}\hookrightarrow\N^{\sig_1}-\del_{\tau}$.
\end{no}

\begin{prop}\label{3.120}
There is a canonical isomorphism $$((\bullet[(\N^{\sig_1}-\del_{\sigma})\cap\ke(a)])_{\sigma\in\sig},(\eps^{\sigma}_{\tau})_{\tau\fleq\sigma})\cong((S_{(\sigma)}(\bullet))_{\sigma\in\sig},(\eta^{(\sigma)}_{(\tau)})_{\tau\fleq\sigma})$$ of projective systems in $\hmf{\ann}{\ann}^{/\Id_{\ann}}$.
\end{prop}

\begin{proof}
For $\sigma\in\sig$ there is a canonical isomorphism $\eps_{\sigma}\colon\bullet[\N^{\sig_1}-\del_{\sigma}]\rightarrow S_{A,\sigma}(\bullet)$ in $\hmf{\ann}{\grann^A}^{/\bullet^{(A)}}$ such that for $\tau\fleq\sigma\in\sig$ the diagram $$\xymatrix@R15pt@C40pt{\bullet[\N^{\sig_1}-\del_{\sigma}]\ar[d]_{\eps^{\sigma}_{\tau}}\ar[r]^(.54){\eps_{\sigma}}_(.54){\cong}&S_{A,\sigma}(\bullet)\ar[d]^{\eta^{\sigma}_{\tau}}\\\bullet[\N^{\sig_1}-\del_{\tau}]\ar[r]^(.54){\eps_{\tau}}_(.54){\cong}&S_{A,\tau}(\bullet)}$$ commutes. As $\bullet[\N^{\sig_1}-\del_{\sigma}]_0=\bullet[(\N^{\sig_1}-\del_{\sigma})\cap\ke(a)]$ for $\sigma\in\sig$ we get the claim by taking components of degree $0$.
\end{proof}

\begin{no}\label{3.130}
It follows from \ref{3.120} that $Y_{\sig}$ is canonically isomorphic to the inductive limit of $$((\spec(S_{(\sigma)}(\bullet)))_{\sigma\in\sig},(\spec\circ\eta^{(\sigma)}_{(\tau)})_{\tau\fleq\sigma})$$ in $\hmf{\ann^{\circ}}{\sch}_{/\spec}$. Hence, if $R$ is a ring then the Cox scheme $Y_{\sig}(R)$ can be understood as obtained by gluing the family $(\spec(S_{(\sigma)}(R)))_{\sigma\in\sig}$ along $(\spec(S_{(\sigma\cap\tau)}(R)))_{(\sigma,\tau)\in\sig^2}$.
\end{no}


\subsection{Toric schemes and Cox schemes}\label{sub3.3}

\begin{no}\label{3.140}
The family $(\sigma^{\vee}_M)_{\sigma\in\sig}$ gives rise to a projective system of submonoids of $M$ fulfilling the hypothesis of \ref{3.100}; we denote it by $\sig_M^{\vee}$. By \ref{3.100} we get a functor $X_{\sig}\dfgl X_{\sig_M^{\vee}}\colon\ann^{\circ}\rightarrow\sch$ over $\spec$, mapping a ring $R$ onto an $R$-scheme $t_{\sig}(R)\colon X_{\sig}(R)\rightarrow\spec(R)$ -- called \textit{the toric scheme over $R$ associated with $\sig$} -- and a finite affine open covering $(X_{\sigma})_{\sigma\in\sig}$ of $X_{\sig}$, where $X_{\sigma}(\bullet)=\spec(\bullet[\sigma_M^{\vee}])$.
\end{no}

\begin{no}\label{3.150}
By restriction and coastriction, $c\colon M\rightarrow\Z^{\sig_1}$ induces for $\sigma\in\sig$ a surjective morphism of monoids $c_{\sigma}\colon\sigma^{\vee}_M\rightarrow(\N^{\sig_1}-\del_{\sigma})\cap\ke(a)$ such that for $\tau\fleq\sigma\in\sig$ the diagram $$\xymatrix@R15pt@C60pt{\sigma^{\vee}_M\ar[r]^(.34){c_{\sigma}}\ar@{ (->}[d]&(\N^{\sig_1}-\del_{\sigma})\cap\ke(a)\ar@{ (->}[d]\\\tau^{\vee}_M\ar[r]^(.34){c_{\tau}}&(\N^{\sig_1}-\del_{\tau})\cap\ke(a)}$$ commutes. So, for $\sigma\in\sig$ we get a closed immersion $\gamma_{\sigma}\dfgl\spec(\bullet[c_{\sigma}])\colon Y_{\sigma}\rightarrow X_{\sigma}$ in $\hmf{\ann^{\circ}}{\sch}_{/\spec}$ such that for $\tau\fleq\sigma\in\sig$ the diagram $$\xymatrix@R15pt@C40pt{Y_{\sigma}\ar[r]^{\gamma_{\sigma}}&X_{\sigma}\\Y_{\tau}\ar[r]^{\gamma_{\tau}}\ar[u]&X_{\tau},\ar[u]}$$ where the unmarked morphisms are the canonical open immersions, commutes. Thus, there is a unique morphism $\gamma_{\sig}\colon Y_{\sig}\rightarrow X_{\sig}$ in $\hmf{\ann^{\circ}}{\sch}_{/\spec}$ that induces $\gamma_{\sigma}$ by restriction and coastriction for every $\sigma\in\sig$.
\end{no}

\begin{prop}\label{3.160}
$\gamma_{\sig}\colon Y_{\sig}\rightarrow X_{\sig}$ is an isomorphism if and only if $\sig$ is full.
\end{prop}

\begin{proof}
As $\ke(c)=M\cap\langle\sig\rangle^{\perp}\subseteq\sigma^{\vee}_M$ for $\sigma\in\sig$ we see that $\sig$ is full if and only if $c_{\sigma}$ is a monomorphism for every $\sigma\in\sig$ (\ref{3.150}). Therefore, if $\sig$ is full then $c_{\sigma}$ is an isomorphism for $\sigma\in\sig$, hence so is $\gamma_{\sigma}$ for $\sigma\in\sig$, and thus so is $\gamma_{\sig}$.

Conversely, suppose that $\gamma_{\sig}$ is an isomorphism and let $\sigma\in\sig$. The (topological) image of $Y_{\sigma}(\Z)$ under $\gamma_{\sigma}(\Z)$ is nonempty, closed in $X_{\sigma}(\Z)$ since $\gamma_{\sigma}(\Z)$ is a closed immersion (\ref{3.150}), and open in $X_{\sigma}(\Z)$ since $\gamma_{\sig}$ is an isomorphism and hence open. As $X_{\sigma}(\Z)$ is connected (\cite[3.4]{geometry}) this implies that $\gamma_{\sigma}(\Z)$ is surjective and therefore an epimorphism in the category of affine schemes (as can be seen on use of \cite[I.5.2.5]{ega}). Therefore, $\Z[c_{\sigma}]$ is a monomorphism, hence so is $c_{\sigma}$, and thus the claim is proven.
\end{proof}

\begin{prop}\label{3.135}
Let $R$ be a ring.

a) If $R$ is noetherian then so is $Y_{\sig}(R)$.

b) If $\sig$ is complete then $Y_{\sig}(R)$ is proper over $R$.
\end{prop}

\begin{proof}
If $R$ is noetherian then so is $S_{(\sigma)}$ for $\sigma\in\sig$ (\ref{3.45}, \ref{2.120}~b)), hence \ref{3.130} and finiteness of $\sig$ imply that $Y_{\sig}(R)$ is noetherian. b) follows from \ref{3.160} and the well-known result that toric schemes associated with complete fans are proper over their base (\cite[4.4]{demazure}\footnote{In \textit{loc.\,cit.} the hypothesis that the fan is regular is not needed in the proof.}).
\end{proof}


\section{Quasicoherent sheaves on toric schemes}\label{sec4}

\noindent\textit{We keep the hypotheses from Section \ref{sec3}. In addition, suppose that $B$ is big, and let $R$ be a ring. If no confusion can arise we write $e\dfgl e_{\sig}$ (\ref{1.140}), $Y\dfgl Y_{\sig}(R)$ and $Y_{\sigma}\dfgl Y_{\sigma}(R)$ for $\sigma\in\sig$ (\ref{3.110}), $S_B\dfgl S_B(R)$ and $I_B\dfgl I_{\sig,B}(R)$ (\ref{3.20}), and $S_{B,\sigma}\dfgl S_{B,\sigma}(R)$ and $S_{(\sigma)}\dfgl S_{(\sigma)}(R)$ for $\sigma\in\sig$ (\ref{3.80}). Moreover, if $\sigma\in\sig$ and $\alpha\in B$ then the component $(S_{B,\sigma})_{\alpha}$ of $S_{B,\sigma}$, equal to $S_B(\alpha)_{(\sigma)}$ and independent of $B$, is sometimes denoted by $S_{\sigma,\alpha}$.\footnote{Similarly, in the notations to be introduced below we usually do not indicate dependency on $\sig$ or $R$.} Furthermore, for a scheme $X$ we denote by $\U_X$ the set of open subsets of $X$, and by $\qcmod(\o_X)$ and $\prmod(\o_X)$ the categories of quasicoherent $\o_X$-modules and of presheaves of $\o_X$-modules, respectively; for a group $G$ and an $\o_X$-algebra $\A$ we denote by $\pgrmod^G(\A)$ the category of presheaves of $G$-graded $\A$-modules.}\smallskip

In \ref{sub4.1} we construct an exact functor $\S_{\sig}$ from the category of graded modules over the Cox ring $S_B$ associated with $\sig$ to the category of quasicoherent modules on the Cox scheme $Y$ associated with $\sig$ and prove first results about it. In \ref{sub4.2} we show that if $\sig$ is simplicial then certain rings of fractions of $S_B$ are strongly graded. This fact will be of crucial importance later on. A first application shows that $\S_{\sig}$ preserves some finiteness properties. In \ref{sub4.3} we try to construct a right quasi-inverse to $\S_{\sig}$. Since $\S_{\sig}$ does not necessarily commute with shifting, we have two candidates $\Gamma_*$ and $\Gamma_{**}$ of such total functors of sections, and both will later turn out to be useful. The key result in this subsection is that while these two functors may differ, their compositions with $\S_{\sig}$ are canonically isomorphic. In \ref{sub4.4} we put everything together and show that the first total functor of sections $\Gamma_*$ is a right quasi-inverse of $\S_{\sig}$, so that the latter functor is in particular essentially surjective. Moreover, supposing that $\sig$ is simplicial, we show that $\S_{\sig}$ preserves (pseudo-)coherence, that it vanishes precisely on $I_B$-torsion modules, and that it induces a bijection between $I_B$-saturated graded sub-$S_B$-modules of a $B$-graded $S_B$-module $F$ and quasicoherent sub-$\o_Y$-modules of $\S_{\sig}(F)$. In \ref{sub4.5} we prove the toric Serre-Grothendieck correspondence, relating sheaf cohomology on the Cox scheme $Y$ with $B$-graded local cohomology with support in $I_B$; the former enters as the right derived cohomological functor of the second total functor of sections $\Gamma_{**}$. Finally, as an application we derive a toric version of Serre's finiteness theorem.

\subsection{Quasicoherent sheaves associated with graded modules}\label{sub4.1}

\begin{no}\label{3.210}
For $\sigma\in\sig$ there is a functor $\S_{\sigma}\colon\catmod(S_{(\sigma)})\rightarrow\qcmod(\o_{Y_{\sigma}})$, mapping an $S_{(\sigma)}$-module onto the quasicoherent $\o_{Y_{\sigma}}$-module associated with it, and taking global sections yields a quasi-inverse $$\Gamma(Y_{\sigma},\bullet)\colon\qcmod(\o_{Y_{\sigma}})\rightarrow\catmod(S_{(\sigma)})$$ of $\S_{\sigma}$ (\cite[I.1.4.2]{ega}). As for $\tau\fleq\sigma\in\sig$ the diagram of categories $$\xymatrix@R15pt@C40pt{\catmod(S_{(\sigma)})\ar[r]^(.47){\S_{\sigma}}\ar[d]_{(\eta^{(\sigma)}_{(\tau)})^*}&\qcmod(\o_{Y_{\sigma}})\ar[d]\\\catmod(S_{(\tau)})\ar[r]^(.47){\S_{\tau}}&\qcmod(\o_{Y_{\tau}}),}$$ where the unmarked morphism is the inverse image under the canonical open immersion $Y_{\tau}\hookrightarrow Y_{\sigma}$, commutes (\cite[I.1.7.7]{ega}), there exists a unique functor $\S_B\colon\grmod^B(S_B)\rightarrow\qcmod(\o_Y)$ such that for $\sigma\in\sig$ the diagram $$\xymatrix@R15pt@C40pt{\grmod^B(S_B)\ar[r]^{\S_B}\ar[d]_{\bullet_{(\sigma)}}&\qcmod(\o_Y)\ar[d]^{\res_{Y_{\sigma}}}\\\catmod(S_{(\sigma)})\ar[r]^{\S_{\sigma}}&\qcmod(\o_{Y_{\sigma}})}$$ of categories commutes (\cite[0.3.3.1--2]{ega}). By \cite[I.1.3.9]{ega} this functor is exact and commutes with inductive limits. By \ref{2.40}, the diagram of categories $$\xymatrix@R15pt@C40pt{\grmod^A(S_A)\ar[r]^{\S_A}\ar[d]_{\bullet_{(B)}}&\qcmod(\o_Y)\\\grmod^B(S_B)\ar[ru]_{\S_B}&}$$ commutes.

If $F$ is a $B$-graded $S_B$-module then $\S_B(F)$ is called \textit{the (quasicoherent) $\o_Y$-module associated with $F$.} If $\sigma\in\sig$ then $\S_{\sigma}(S_{(\sigma)})=\o_{Y_{\sigma}}$ is a sheaf of rings on $Y_{\sigma}$, and gluing yields on $\S_B(S_B)$ a structure of sheaf of rings on $Y_{\sig}$, equal to $\o_Y$; we always furnish $\S_B(S_B)$ with this structure. 
\end{no}

\begin{no}\label{3.250}
Let $F$ be a $B$-graded $S_B$-module. We denote by $\mathbbm{J}_F$, $\mathbbm{J}^{\sat}_F$ and $\widetilde{\mathbbm{J}}_F$ the sets of graded sub-$S_B$-modules of $F$, of $I_B$-saturated graded sub-$S_B$-modules of $F$ (\ref{2.80}), and of quasicoherent sub-$\o_Y$-modules of $\S_B(F)$, respectively. Since $\S_B$ is exact (\ref{3.210}) it induces a map $\Xi_F\colon\mathbbm{J}_F\rightarrow\widetilde{\mathbbm{J}}_F$, restricting to a map $\Xi_F^{\sat}\colon\mathbbm{J}_F^{\sat}\rightarrow\widetilde{\mathbbm{J}}_F$. If $G\in\mathbbm{J}_F$ and $j\colon G\hookrightarrow F$ denotes the canonical injection then we identify $\S_B(G)$ with its image $\Xi_F(G)$ under the monomorphism $\S_B(j)\colon\S_B(G)\rightarrow\S_B(F)$ and thus consider it as an element of $\widetilde{\mathbbm{J}}_F$. Then, for $G,H\in\mathbbm{J}_F$ we have $\S_B(G+H)=\S_B(G)+\S_B(H)$ and $\S_B(G\cap H)=\S_B(G)\cap\S_B(H)$ (\ref{3.210}).
\end{no}

\begin{prop}\label{3.230}
For an $R$-algebra $h\colon R\rightarrow R'$ there exists a canonical isomorphism $$Y(h)^*(\S_{B,R}(\bullet))\cong\S_{B,R'}(\bullet\otimes_RR')$$ in $\hmf{\grmod^B(S_B(R))}{\qcmod(\o_{Y(R')})}$.
\end{prop}

\begin{proof}
If $\sigma\in\sig$ then the diagram of categories $$\xymatrix@C20pt@R10pt{&\grmod^B(S_B(R))\ar[rr]^{\S_{B,R}}\ar[dd]^(.35){\bullet_{(\sigma)}}\ar[ld]_{\bullet\otimes_RR'}&&\qcmod(\o_{Y(R)})\ar[dd]^{\res_{Y_{\sigma}(R)}}\ar[ld]^{\bullet\otimes_RR'}\\\grmod^B(S_B(R'))\ar[rr]^(.4){\S_{B,R'}}\ar[dd]_{\bullet_{(\sigma)}}&&\qcmod(\o_{Y_{\sigma}(R')})\ar[dd]^(.35){\res_{Y_{\sigma}(R')}}&\\&\catmod(S_{(\sigma)}(R))\ar[rr]^(.4){\S_{\sigma,R}}\ar[ld]_{\bullet\otimes_RR'}&&\qcmod(\o_{Y_{\sigma}(R)})\ar[ld]^{\bullet\otimes_RR'}\\\catmod(S_{(\sigma)}(R'))\ar[rr]^{\S_{\sigma,R'}}&&\qcmod(\o_{Y_{\sigma}(R')})&}$$ commutes, as is readily checked on use of \ref{3.210} and \cite[0.4.3.6; I.1.7.7]{ega}. This implies the claim.
\end{proof}

\begin{prop}\label{3.270}
If $F$ is a $B$-graded $S_B$-module that is an $I_B$-torsion module then $\S_B(F)=0$.\footnote{See \ref{3.980} for a converse.}
\end{prop}

\begin{proof}
Let $\sigma\in\sig$ and $x\in F_{(\sigma)}$. There exist $m\in\N$ with $m\alp_{\sigma}\in B$ and $x'\in F_{m\alp_{\sigma}}$ with $x=\frac{x'}{\zs^m}$, and there exists $l\in\N$ with $\zs^lx'=0$, implying $x=0$. This shows $F_{(\sigma)}=0$, and thus the claim.
\end{proof}

\begin{prop}\label{3.295}
If $F$ is a noetherian $B$-graded $S_B$-module then $\S_B(F)$ is of finite type; if $R$ is in addition noetherian then $\S_B(F)$ is coherent.\footnote{See \ref{3.340} and \ref{3.990} for more natural results.}
\end{prop}

\begin{proof}
For $\sigma\in\sig$ the $B$-graded $S_{B,\sigma}$-module $F_{\sigma}$ is noetherian, hence $F_{(\sigma)}$ is noetherian (\ref{2.120}~b)) and in particular of finite type. Thus, $\S_B(F)$ is of finite type (\cite[I.1.4.3]{ega}). If $R$ is noetherian then so is $Y$ (\ref{3.135}~a)), hence $\S_B(F)$ is coherent (\cite[I.2.7.1]{ega}).
\end{proof}

\begin{no}\label{3.280}
By \ref{2.60} we have for $\sigma\in\sig$ a canonical morphism of bifunctors $$\delta_{\sigma}\colon\bullet_{(\sigma)}\otimes_{S_{(\sigma)}}\sq_{(\sigma)}\rightarrow(\bullet\otimes_{S_B}\sq)_{(\sigma)}$$ such that for $\tau\fleq\sigma$ the diagram $$\xymatrix@R15pt@C40pt{\bullet_{(\sigma)}\otimes_{S_{(\sigma)}}\sq_{(\sigma)}\ar[r]^{\delta_{\sigma}}\ar[d]_{\eta^{(\sigma)}_{(\tau)}\otimes\eta^{(\sigma)}_{(\tau)}}&(\bullet\otimes_{S_B}\sq)_{(\sigma)}\ar[d]^{\eta^{(\sigma)}_{(\tau)}}\\\bullet_{(\tau)}\otimes_{S_{(\tau)}}\sq_{(\tau)}\ar[r]^{\delta_{\tau}}&(\bullet\otimes_{S_B}\sq)_{(\tau)}}$$ commutes. Therefore, and keeping in mind \cite[I.1.3.12]{ega}, these morphisms give rise to a morphism of bifunctors $$\delta_B\colon\S_B(\bullet)\otimes_{\o_Y}\S_B(\sq)\rightarrow\S_B(\bullet\otimes_{S_B}\sq)$$ with $\delta_B(\bullet,\sq)(Y_{\sigma})=\delta_{\sigma}(\bullet,\sq)$ for $\sigma\in\sig$.
\end{no}

\begin{nosm}\label{3.290}
The morphism $$\delta_B\colon\S_B(\bullet)\otimes_{\o_Y}\S_B(\sq)\rightarrow\S_B(\bullet\otimes_{S_B}\sq)$$ from \ref{3.280} is neither necessarily a mono- nor an epimorphism, and in particular not necessarily an isomorphism. For a counterexample\footnote{taken from \cite[01ML]{stacks}} we consider the $\Z^2$-polycones $\rho_1=\cone(1,0)$, $\rho_2=\cone(0,1)$ and $\rho_3=\cone(-2,-3)$ in $\R^2$, the complete, simplicial $\Z^2$-fan $\sig$ with maximal cones $\sigma_1=\rho_1+\rho_2$, $\sigma_2=\rho_2+\rho_3$ and $\sigma_3=\rho_3+\rho_1$, and $R\neq 0$. It follows $A=\Z$, and setting $Z_i\dfgl Z_{\rho_i}$ for $i\in[1,3]$ we have $S\dfgl S_A=R[Z_1,Z_2,Z_3]$ with $\deg(Z_1)=2$, $\deg(Z_2)=3$, and $\deg(Z_3)=1$. We consider the $\Z$-graded $S$-modules $S(1)$ and $S(2)$, so that $S(1)\otimes_SS(2)$ is canonically isomorphic to and henceforth identified with $S(3)$, and we claim that $$\delta_A(S(1),S(2))\colon\S_A(S(1))\otimes_{\o_Y}\S_A(S(2))\rightarrow\S_A(S(1)\otimes_SS(2))$$ is neither a mono- nor an epimorphism.

Indeed, by construction of $\delta_A$ and as $\widehat{Z}_{\sigma_3}=Z_2$, it suffices to show that the morphism of $S_{(Z_2)}$-modules $$\delta_{\sigma_3}\colon S(1)_{(Z_2)}\otimes_{S_{(Z_2)}}S(2)_{(Z_2)}\rightarrow S(3)_{(Z_2)}$$ with $x\otimes y\mapsto xy$ is neither injective nor surjective. Keeping in mind that $S(m)_{(Z_2)}=(S_{Z_2})_m$ for $m\in\Z$ it is readily checked that $S_{(Z_2)}=R[Z_3^3Z_2^{-1},Z_3Z_1Z_2^{-1},Z_1^3Z_2^{-2}]$, $S(1)_{(Z_2)}=\langle Z_3,Z_1^2Z_2^{-1}\rangle_{S_{(Z_2)}}$, $S(2)_{(Z_2)}=\langle Z_3^2,Z_1\rangle_{S_{(Z_2)}}$ and $S(3)_{(Z_2)}=\langle Z_2\rangle_{S_{(Z_2)}}$, so that $\delta_{\sigma_3}$ equals the morphism of $S_{(Z_2)}$-modules $$\langle Z_3,Z_1^2Z_2^{-1}\rangle_{S_{(Z_2)}}\otimes_{S_{(Z_2)}}\langle Z_3^2,Z_1\rangle_{S_{(Z_2)}}\rightarrow\langle Z_2\rangle_{S_{(Z_2)}}$$ with $x\otimes y\mapsto xy$. This is neither injective, for $(Z_3Z_1Z_2^{-1}\cdot Z_3)\otimes Z_1$ and $Z_1^2Z_2^{-1}\otimes Z_3^2$ are both mapped to $Z_3^2Z_1^2Z_2^{-2}\cdot Z_2$, nor surjective, for $Z_2$ does not lie in its image.
\end{nosm}


\subsection{Strongly graded rings of fractions}\label{sub4.2}

\begin{prop}\label{3.300}
Let $\sigma\in\sig$, let $\alpha\in B$, and let $(m_{\rho})_{\rho\in\sig_1}\in(\N^{\sig_1}-\del_{\sigma})\cap c^{-1}(\alpha)$ with $m_{\rho}=0$ for every $\rho\in\sigma_1$. If $R\neq 0$ then the $S_{(\sigma)}$-module $S_{\sigma,\alpha}$ is free of rank $1$.
\end{prop}

\begin{proof}
For $(q_{\rho})_{\rho\in\sig_1}\in(\N^{\sig_1}-\del_{\sigma})\cap c^{-1}(\alpha)$ there is an $m\in M$ with $(q_{\rho})_{\rho\in\sig}=c(m)+(m_{\rho})_{\rho\in\sig_1}$, and thus $\prod_{\rho\in\sig_1}Z_{\rho}^{q_{\rho}}=(\prod_{\rho\in\sig_1}Z_{\rho}^{\rho_N(m)})(\prod_{\rho\in\sig_1}Z_{\rho}^{m_{\rho}})$ in $S_{\sigma}$. As elements of $S_{\sigma,\alpha}$ are $S_{(\sigma)}$-linear combinations of elements of the form $\prod_{\rho\in\sig_1}Z_{\rho}^{q_{\rho}}$ with $(q_{\rho})_{\rho\in\sig_1}\in(\N^{\sig_1}-\del_{\sigma})\cap c^{-1}(\alpha)$, this implies that $\prod_{\rho\in\sig_1}Z_{\rho}^{m_{\rho}}$ generates the $S_{(\sigma)}$-module $S_{\sigma,\alpha}$. This element obviously being free over $S_{(\sigma)}$ the claim is proven.
\end{proof}

\begin{cor}\label{3.310}
Let $\sigma\in\sig$ and let $\alpha\in\pic(\sig)$. If $R\neq 0$ then there exists $(m_{\rho})_{\rho\in\sig_1}\in$\linebreak$(\N^{\sig_1}-\del_{\sigma})\cap c^{-1}(\alpha)$ such that the $S_{(\sigma)}$-module $S_{\sigma,\alpha}$ is free of rank $1$ with basis $\prod_{\rho\in\sig_1}Z_{\rho}^{m_{\rho}}$.
\end{cor}

\begin{proof}
Clear from \ref{3.300} and \ref{1.120}.
\end{proof}

\begin{no}\label{3.600}
Let $\alpha\in B$. Since $\S_B(S_B(\alpha))=\S_A(S_A(\alpha))$ (\ref{3.210}), the right exact functor $$\S_B(S_B(\alpha))\otimes_{\o_Y}\bullet\colon\qcmod(\o_Y)\rightarrow\qcmod(\o_Y)$$ does not depend on $B$; we denote it by $\bullet(\alpha)$. We have $\o_Y(\alpha)=\S_B(S_B(\alpha))$, and hence for $\sigma\in\sig$ it follows $\o_Y(\alpha)(Y_{\sigma})=S_B(\alpha)_{(\sigma)}=S_{\sigma,\alpha}$.
\end{no}

\begin{cor}\label{3.390}
If $\alpha\in\pic(\sig)$ then the $\o_Y$-module $\o_Y(\alpha)$ is invertible.
\end{cor}

\begin{proof}
As $\o_Y(\alpha)\res_{Y_{\sigma}}=\S_{\sigma}(S_{\sigma,\alpha})$ for $\sigma\in\sig$ (\ref{3.600}), this follows from \ref{3.310}.
\end{proof}

\begin{thm}\label{3.330}
Suppose that $B$ is small and let $\sigma\in\sig$.

a) The $B$-graded ring $S_{B,\sigma}$ is strongly graded.\footnote{cf.~\ref{2.130}}

b) If $F$ is a $B$-graded $S_B$-module and $\alpha\in B$ then there is an isomorphism of $S_{(\sigma)}$-modules $(F_{\sigma})_{\alpha}\cong F_{(\sigma)}$.
\end{thm}

\begin{proof}
a) It suffices to show that for $\alpha,\beta\in B$ the morphism of $S_{(\sigma)}$-modules $$S_{\sigma,\alpha}\otimes_{S_{(\sigma)}}S_{\sigma,\beta}\rightarrow S_{\sigma,\alpha+\beta}$$ induced by the multiplication of $S_{B,\sigma}$ is an isomorphism. Being induced by restriction and coastriction of the canonical isomorphism $S_{B,\sigma}\otimes_{S_{B,\sigma}}S_{B,\sigma}\rightarrow S_{B,\sigma}$ it is injective. There are $p=(m_{\tau}+\tau^{\vee})_{\tau\in\sig},q=(l_{\tau}+\tau^{\vee})_{\tau\in\sig}\in\overline{P}_{\sig}$ with $e(p)=\alpha$, $e(q)=\beta$ and $m_{\tau}=l_{\tau}=0$ for $\tau\fleq\sigma$ (\ref{1.120}). As $p+q=(m_{\tau}+l_{\tau}+\tau^{\vee})_{\tau\in\sig}$ and $e(p+q)=\alpha+\beta$ the claim follows from \ref{3.300}.

b) We have $S_{\sigma,\alpha}\cong S_{(\sigma)}$ by \ref{3.310} and $F_{\sigma}\cong S_{B,\sigma}\otimes_{S_{(\sigma)}}F_{(\sigma)}$ by \ref{2.130} and a). Taking components of degree $\alpha$ yields $(F_{\sigma})_{\alpha}\cong S_{\sigma,\alpha}\otimes_{S_{(\sigma)}}F_{(\sigma)}\cong S_{(\sigma)}\otimes_{S_{(\sigma)}}F_{(\sigma)}\cong F_{(\sigma)}$.
\end{proof}

\begin{prop}\label{3.340}
Suppose that $B$ is small. If $F$ is a $B$-graded $S_B$-module of finite type (or of finite presentation) then $\S_B(F)$ is an $\o_Y$-module of finite type (or of finite presentation).
\end{prop}

\begin{proof}
We prove the claim about finite presentation; the other one is proven analogously. Let $\sigma\in\sig$. There exist finite subsets $C,C'\subseteq B$, families $(n_{\alpha})_{\alpha\in C}$ and $(n'_{\alpha})_{\alpha\in C'}$ in $\N$, and an exact sequence of $B$-graded $S_B$-modules $$\bigoplus_{\alpha\in C'}(S_B(\alpha)^{\oplus n'_{\alpha}})\overset{v}\longrightarrow\bigoplus_{\alpha\in C}(S_B(\alpha)^{\oplus n_{\alpha}})\overset{u}\longrightarrow F\longrightarrow 0.$$ Applying the exact functor $\bullet_{(\sigma)}$ yields an exact sequence of $S_{(\sigma)}$-modules $$\bigoplus_{\alpha\in C'}((S(\alpha)_{(\sigma)})^{\oplus n'_{\alpha}})\overset{v_{(\sigma)}}\longrightarrow\bigoplus_{\alpha\in C}((S(\alpha)_{(\sigma)})^{\oplus n_{\alpha}})\overset{u_{(\sigma)}}\longrightarrow F_{(\sigma)}\longrightarrow 0.$$ As $B$ is small there exists for $\alpha\in B$ an isomorphism of $S_{(\sigma)}$-modules $S_B(\alpha)_{(\sigma)}=S_{\sigma,\alpha}\cong S_{(\sigma)}$ (\ref{3.330}~b)). Hence, there is an exact sequence of $S_{(\sigma)}$-modules $$\bigoplus_{\alpha\in C'}((S_{(\sigma)})^{\oplus n'_{\alpha}})\longrightarrow\bigoplus_{\alpha\in C}((S_{(\sigma)})^{\oplus n_{\alpha}})\longrightarrow F_{(\sigma)}\longrightarrow 0,$$ implying that $F_{(\sigma)}$ is of finite presentation and thus the claim.
\end{proof}

\begin{prop}\label{3.350}
Let $F$ be a $B$-graded $S_B$-module. We consider the following statements:
\begin{aufz}
\item[(1)] $F_{\sigma}$ is flat over $R$ for every $\sigma\in\sig$;
\item[(2)] $F_{(\sigma)}$ is flat over $R$ for every $\sigma\in\sig$;
\item[(3)] $\S_B(F)$ is flat over $R$.
\end{aufz}
We have (1)$\Rightarrow$(2)$\Leftrightarrow$(3), and if $B$ is small then (1)--(3) are equivalent.\footnote{In \cite{torsion} it is shown that, if $B$ is small, flatness of $\S_B(F)$ implies flatness of $F_{\alpha}$ for certain values of $\alpha$ that can be described cohomologically.}
\end{prop}

\begin{proof}
(1) implies (2) by \cite[I.2.3 Proposition 2]{ac}. (2) holds if and only if $\S_{\sigma}(F_{(\sigma)})=\S_B(F)\res_{Y_{\sigma}}$ is flat over $R$ for every $\sigma\in\sig$ (\cite[IV.2.1.2]{ega}), and this is equivalent to (3). Finally, if $B$ is small then (2) implies that $(F_{\sigma})_{\alpha}$ is flat for every $\sigma\in\sig$ and every $\alpha\in B$ (\ref{3.330}~b)), and thus $F_{\sigma}=\bigoplus_{\alpha\in B}(F_{\sigma})_{\alpha}$ is flat (\cite[I.2.3 Proposition 2]{ac}).
\end{proof}


\subsection{Total functors of sections}\label{sub4.3}

\begin{no}\label{3.610}
Let $\alpha,\beta\in B$. For $\sigma\in\sig$ the multiplication of $S_{B,\sigma}$ induces a morphism of $S_{(\sigma)}$-modules $S_{\sigma,\alpha}\otimes_{S_{(\sigma)}}S_{\sigma,\beta}\rightarrow S_{\sigma,\alpha+\beta}$ such that for $\tau\fleq\sigma\in\sig$ the diagram $$\xymatrix@R15pt@C40pt{S_{\sigma,\alpha}\otimes_{S_{(\sigma)}}S_{\sigma,\beta}\ar[r]\ar[d]_{(\eta^{\sigma}_{\tau})_{\alpha}\otimes(\eta^{\sigma}_{\tau})_{\beta}}&S_{\sigma,\alpha+\beta}\ar[d]^{(\eta^{\sigma}_{\tau})_{\alpha+\beta}}\\S_{\tau,\alpha}\otimes_{S_{(\tau)}}S_{\tau,\beta}\ar[r]&S_{\tau,\alpha+\beta}}$$ commutes, i.e., a morphism of $\o_Y(Y_{\sigma})$-modules $$\o_Y(\alpha)(Y_{\sigma})\otimes_{\o_Y(Y_{\sigma})}\o_Y(\beta)(Y_{\sigma})\rightarrow\o_Y(\alpha+\beta)(Y_{\sigma})$$ such that for $\tau\fleq\sigma\in\sig$ the diagram $$\xymatrix@R15pt@C40pt{\o_Y(\alpha)(Y_{\sigma})\otimes_{\o_Y(Y_{\sigma})}\o_Y(\beta)(Y_{\sigma})\ar[r]\ar[d]_{\res_{Y_{\tau}}\otimes\,\res_{Y_{\tau}}}&\o_Y(\alpha+\beta)(Y_{\sigma})\ar[d]^{\res_{Y_{\tau}}}\\\o_Y(\alpha)(Y_{\tau})\otimes_{\o_Y(Y_{\tau})}\o_Y(\beta)(Y_{\tau})\ar[r]&\o_Y(\alpha+\beta)(Y_{\tau})}$$ commutes (\ref{3.600}). Gluing yields a morphism of $\o_Y$-modules $$\mu_{\alpha,\beta}\colon\o_Y(\alpha)\otimes_{\o_Y}\o_Y(\beta)\rightarrow\o_Y(\alpha+\beta).$$ Denoting for $\alpha,\beta\in B$ by $$\sigma_{\alpha,\beta}\colon\o_Y(\alpha)\otimes_{\o_Y}\o_Y(\beta)\longrightarrow\o_Y(\beta)\otimes_{\o_Y}\o_Y(\alpha)$$ the canonical isomorphism it is readily checked that $\mu_{\alpha,\beta}=\mu_{\beta,\alpha}\circ\sigma_{\alpha,\beta}$ for $\alpha,\beta\in B$ and $$\mu_{\alpha+\beta,\gamma}\circ(\mu_{\alpha,\beta}\otimes_{\o_Y}\Id_{\o_Y(\gamma)})=\mu_{\alpha+\beta,\gamma}\circ(\Id_{\o_Y(\alpha)}\otimes_{\o_Y}\mu_{\beta,\gamma})$$ for $\alpha,\beta,\gamma\in B$.
\end{no}

\begin{no}\label{3.620}
We define a presheaf of $B$-graded $\o_Y$-algebras $$\Gamma_{*,B}(\o_Y)\colon(V\subseteq U)\mapsto(\Gamma_{*,B,U}(\o_Y)\rightarrow\Gamma_{*,B,V}(\o_Y));$$ its underlying presheaf of $B$-graded $\o_Y$-modules is given by $$(V\subseteq U)\mapsto(\bigoplus_{\alpha\in B}\Gamma(U,\o_Y(\alpha))\rightarrow\bigoplus_{\alpha\in B}\Gamma(V,\o_Y(\alpha))),$$ where the restriction morphisms are induced by the restriction morphisms of $\o_Y(\alpha)$, and its multiplication is given for $U\in\U_Y$ and $\alpha,\beta\in B$ by the composition of the canonical morphism $$\o_Y(\alpha)(U)\otimes_{\o_Y(U)}\o_Y(\beta)(U)\rightarrow(\o_Y(\alpha)\otimes_{\o_Y}\o_Y(\beta))(U)$$ with $\mu_{\alpha,\beta}(U)$ (\ref{3.610}).
\end{no}

\begin{no}\label{3.630}
If $\alpha\in B$ and $\tau\fleq\sigma\in\sig$ then taking components of degree $\alpha$ of the canonical morphisms $\eta_{\sigma}\colon S_B\rightarrow S_{B,\sigma}$ and $\eta_{\tau}\colon S_B\rightarrow S_{B,\tau}$, and keeping in mind \ref{3.600}, we get a commutative diagram of $S_0$-modules $$\xymatrix@R15pt@C60pt{S_{\alpha}\ar[r]^(.45){(\eta_{\sigma})_{\alpha}}\ar[rd]_(.42){(\eta_{\tau})_{\alpha}}&\o_Y(\alpha)(Y_{\sigma})\ar[d]^{\res_{Y_{\tau}}}\\&\o_Y(\alpha)(Y_{\tau}).}$$ Hence, there is a unique morphism of $S_0$-modules $\eta_{\alpha}\colon S_{\alpha}\rightarrow\o_Y(\alpha)(Y)$ such that for $\sigma\in\sig$ the diagram $$\xymatrix@R15pt@C60pt{S_{\alpha}\ar[r]^(.45){\eta_{\alpha}}\ar[rd]_(.42){(\eta_{\sigma})_{\alpha}}&\o_Y(\alpha)(Y)\ar[d]^{\res_{Y_{\sigma}}}\\&\o_Y(\alpha)(Y_{\sigma})}$$ commutes. Thus, we get a morphism of $S_0$-modules $$\eta_B\dfgl\bigoplus_{\alpha\in B}\eta_{\alpha}\colon S_B\rightarrow\Gamma_{*,B,Y}(\o_Y),$$ which is readily seen to be a morphism of $B$-graded rings. By means of this we consider $\Gamma_{*,B,Y}(\o_Y)$ as a $B$-graded $S_B$-algebra and hence $\Gamma_{*,B}(\o_Y)$ as a presheaf of $B$-graded $S_B$-algebras.
\end{no}

\begin{prop}\label{3.640}
The morphism of $B$-graded rings $\eta_B\colon S_B\rightarrow\Gamma_{*,B,Y}(\o_Y)$ is surjective; it is an isomorphism if and only if $\sig\neq\emptyset$ or $R=0$.
\end{prop}

\begin{proof}
If $\sig=\emptyset$ then $\eta_B$ is a morphism of rings $R\rightarrow 0$, hence surjective, and an isomorphism if and only if $R=0$. Let $\sig\neq\emptyset$ and $\alpha\in B$. It suffices to show that $\eta_{\alpha}\colon S_{\alpha}\rightarrow\o_Y(\alpha)(Y)$ is an isomorphism.

First, let $x\in\o_Y(\alpha)(Y)$. If $\sigma\in\sig$ then $x\res_{Y_{\sigma}}\in S_{\sigma,\alpha}$ (\ref{3.600}), hence there is an $m\in\N$ such that if $\sigma\in\sig$ then $m\alp_{\sigma}\in B$ and there is a $y^{(\sigma)}\in S_{\alpha+m\alp_{\sigma}}$ with $x\res_{Y_{\sigma}}=\frac{y^{(\sigma)}}{\zs^m}$. For $\sigma\in\sig$ it follows $$\textstyle\frac{y^{(\sigma)}\cdot\prod_{\rho\in\sigma_1}Z_{\rho}^m}{\prod_{\rho\in\sig_1}Z_{\rho}^m}=x\res_{Y_{\sigma}}\res_{Y_{\{0\}}}=x\res_{Y_{\{0\}}}=\frac{y^{(\{0\})}}{\prod_{\rho\in\sig_1}Z_{\rho}^m}\in S_{\{0\},\alpha},$$ hence $y^{(\sigma)}\cdot\prod_{\rho\in\sigma_1}Z_{\rho}^m=y^{(\{0\})}$. Therefore, $Z_{\rho}^m$ divides $y^{(\{0\})}$ for every $\rho\in\sig_1$, and thus there is a $y\in S_{\alpha}$ with $\frac{y}{1}=x\res_{Y_{\{0\}}}\in S_{\{0\},\alpha}$. For $\sigma\in\sig$ it follows $\eta_{\alpha}(y)\res_{Y_{\sigma}}=\frac{y}{1}=\frac{y^{(\sigma)}}{\zs^m}=x\res_{Y_{\sigma}}\in S_{\sigma,\alpha}$, hence $\eta_{\alpha}(y)=x$, showing that $\eta_{\alpha}$ is surjective.

Next, let $x\in\ke(\eta_{\alpha})$. For $\sigma\in\sig$ we have $\frac{x}{1}=\eta_{\alpha}(x)\res_{Y_{\sigma}}=0\in S_{\sigma,\alpha},$ hence there is an $m\in\N$ with $\zs^mx=0\in S_B$. This implies $x=0$, and thus $\eta_{\alpha}$ is injective.
\end{proof}

\begin{no}\label{3.650}
Let $\sigma\in\sig$. It is clear from \ref{3.640} that $\eta_B$ induces an isomorphism of $B$-graded rings $(\eta_B)_{\sigma}\colon$\linebreak$S_{B,\sigma}\rightarrow\Gamma_{*,B,Y_{\sigma}}(\o_Y)$ such that for $\tau\fleq\sigma$ the diagram $$\xymatrix@R10pt{S_B\ar[rr]^{\eta_{\sigma}}\ar[rd]_{\eta_{\tau}}\ar[dd]^{\cong}_{\eta_B}&&S_{B,\sigma}\ar[ld]^{\eta^{\sigma}_{\tau}}\ar[dd]_{\cong}^{(\eta_B)_{\sigma}}\\&S_{B,\tau}\ar[dd]_(.3){\cong}^(.3){(\eta_B)_{\tau}}&\\\Gamma_{*,B,Y}(\o_Y)\ar[rr]^(.3){\res_{Y_{\sigma}}}\ar[rd]_{\res_{Y_{\tau}}}&&\Gamma_{*,B,Y_{\sigma}}(\o_Y)\ar[ld]^{\res_{Y_{\tau}}}\\&\Gamma_{*,B,Y_{\tau}}(\o_Y)&}$$ commutes. Hence, for $\tau\fleq\sigma$ we have a commutative diagram of $B$-graded $S_B$-algebras $$\xymatrix@R15pt{&\Gamma_{*,B,Y}(\o_Y)\ar[d]_{\eta_{\sigma}}\ar[ldd]_{\eta_{\tau}}\ar[rdd]^{\res_{Y_{\tau}}}\ar[rrd]^{\res_{Y_{\sigma}}}&&\\&\Gamma_{*,B,Y}(\o_Y)_{\sigma}\ar[rr]_{\cong}\ar[ld]^{\eta^{\sigma}_{\tau}}&&\Gamma_{*,B,Y_{\sigma}}(\o_Y)\ar[ld]^{\res_{Y_{\tau}}}\\\Gamma_{*,B,Y}(\o_Y)_{\tau}\ar[rr]^{\cong}&&\Gamma_{*,B,Y_{\tau}}(\o_Y)&}$$ where the unmarked isomorphisms are induced by $\res_{Y_{\sigma}}$ and $\res_{Y_{\tau}}$. By means of the above isomorphisms we identify the $B$-graded $S_B$-algebras $S_{B,\sigma}$, $\Gamma_{*,B,Y}(\o_Y)_{\sigma}$ and $\Gamma_{*,B,Y_{\sigma}}(\o_Y)$.
\end{no}

\begin{no}\label{3.660}
Let $\F$ be a quasicoherent $\o_Y$-module. We define a presheaf of $B$-graded $\Gamma_{*,B}(\o_Y)$-modules $$\Gamma_{*,B}(\F)\colon(V\subseteq U)\mapsto(\Gamma_{*,B,U}(\F)\rightarrow\Gamma_{*,B,V}(\F));$$ its underlying presheaf of $\o_Y$-modules (i.e., its composition with scalar restriction from $\Gamma_{*,B}(\o_Y)$ to $\Gamma_{*,B}(\o_Y)_0=\o_Y$) is given by $$(V\subseteq U)\mapsto(\bigoplus_{\alpha\in B}\Gamma(U,\o_Y(\alpha)\otimes_{\o_Y}\F)\rightarrow\bigoplus_{\alpha\in B}\Gamma(V,\o_Y(\alpha)\otimes_{\o_Y}\F)),$$ where the restriction morphisms are induced by the restriction morphisms of $\o_Y(\alpha)\otimes_{\o_Y}\F$, and its $\Gamma_{*,B}(\o_Y)$-action is given for $U\in\U_Y$ and $\alpha,\beta\in B$ by the composition of the canonical morphism $$\o_Y(\alpha)(U)\otimes_{\o_Y(U)}(\o_Y(\beta)\otimes_{\o_Y}\F)(U)\rightarrow(\o_Y(\alpha)\otimes_{\o_Y}\o_Y(\beta)\otimes_{\o_Y}\F)(U)$$ with $(\mu_{g,h}\otimes_{\o_Y}\F)(U)$.

In case $\F=\o_Y$ we get back $\Gamma_{*,B}(\o_Y)$ considered as a $B$-graded module over itself, so that our notation does not lead to confusion.
\end{no}

\begin{no}\label{3.670}
The construction in \ref{3.660} can be canonically extended to give rise to a functor $$\Gamma_{*,B}\colon\qcmod(\o_Y)\rightarrow\pgrmod^B(\Gamma_{*,B}(\o_Y)),\;(\F\overset{u}\rightarrow\G)\mapsto(\Gamma_{*,B}(\F)\xrightarrow{\Gamma_{*,B}(u)}\Gamma_{*,B}(\G)).$$ In particular, we get for $U\in\U_Y$ a functor $$\Gamma_{*,B,U}\colon\qcmod(\o_Y)\rightarrow\grmod^B(\Gamma_{*,B,U}(\o_Y))$$ and for $V\in\U_U$ a morphism of functors $\Gamma_{*,B,U}\rightarrow\Gamma_{*,B,V}$. For $U\in\U_Y$ we have a commutative diagram of categories $$\xymatrix@C40pt@R15pt{\qcmod(\o_Y)\ar[r]^(.48){\Gamma_{*,A,U}}\ar[rd]_{\Gamma_{*,B,U}}&\grmod^A(S_A)\ar[d]^{\bullet_{(B)}}\\&\grmod^B(S_B).}$$

If $U\in\U_Y$ then composing $\Gamma_{*,B,U}$ with scalar restriction to $S_B$ yields a functor\linebreak $\qcmod(\o_Y)\rightarrow\grmod^B(S_B)$, again denoted by $\Gamma_{*,B,U}$. In case $U=Y$ the functor $$\Gamma_{*,B,Y}\colon\qcmod(\o_Y)\rightarrow\grmod^B(S_B)$$ is called \textit{the $B$-restricted first total functor of sections over $R$ associated with $\sig$.} 
\end{no}

\begin{no}\label{3.680}
Let $\F$ be a quasicoherent $\o_Y$-module. For $\sigma\in\sig$ and $x\in\Gamma_{*,B,Y}(\F)_{(\sigma)}$ there exist $m\in\N$ with $m\alp_{\sigma}\in B$ and $x'\in\Gamma_{*,B,Y}(\F)_{m\alp_{\sigma}}=(\S_B(S_B(m\alp_{\sigma}))\otimes_{\o_Y}\F)(Y)$ with $x=\frac{x'}{\zs^m}$. As $x'\res_{Y_{\sigma}}\in\Gamma_{*,B,Y_{\sigma}}(\F)_{m\alp_{\sigma}}=(\S_B(S_B(m\alp_{\sigma}))\otimes_{\o_Y}\F)(Y_{\sigma})$ and as $\Gamma_{*,B,Y_{\sigma}}(\F)$ is an $S_{B,\sigma}$-module it follows $$\textstyle\beta_{B,\sigma}(\F)(x)\dfgl\frac{1}{\zs^m}\cdot x'\res_{Y_{\sigma}}\in\Gamma_{*,B,Y_{\sigma}}(\F)_0=\F(Y_{\sigma}).$$ This yields a morphism of $S_{(\sigma)}$-modules $\beta_{B,\sigma}(\F)\colon\Gamma_{*,B,Y}(\F)_{(\sigma)}\rightarrow\F(Y_{\sigma})$. If $\tau\fleq\sigma\in\sig$ then the diagram $$\xymatrix@R15pt@C60pt{\Gamma_{*,B,Y}(\F)_{(\sigma)}\ar[r]^(.55){\beta_{B,\sigma}(\F)}\ar[d]_{\eta^{(\sigma)}_{(\tau)}}&\F(Y_{\sigma})\ar[d]^{\res_{Y_{\tau}}}\\\Gamma_{*,B,Y}(\F)_{(\tau)}\ar[r]^(.55){\beta_{B,\tau}(\F)}&\F(Y_{\tau})}$$ commutes, and hence gluing yields a unique morphism of $\o_Y$-modules $$\beta_B(\F)\colon\S_B(\Gamma_{*,B,Y}(\F))\rightarrow\F$$ with $\beta_B(\F)(Y_{\sigma})=\beta_{B,\sigma}(\F)$ for $\sigma\in\sig$. Varying $\F$ gives rise to a morphism of functors $$\xymatrix@C40pt@R0pt{&\ar@<1ex>@{=>}[dd]^{\beta_B}&\\\qcmod(\o_Y)\ar@/^3ex/[rr]^{\S_B\circ\Gamma_*^{B,Y}}\ar@/_3ex/[rr]_{\Id_{\qcmod(\o_Y)}}&&\qcmod(\o_Y).\\&&}$$
\end{no}

\begin{no}\label{3.690}
For $\alpha\in B$ we have a functor $$\S_B(\bullet(\alpha))=\S_B(S_B(\alpha)\otimes_{S_B}\bullet)\colon\grmod^B(S_B)\rightarrow\qcmod(\o_Y).$$ Taking sections over $U\in\U_Y$ yields a functor $$\Gamma_{\alpha,B,U}(\bullet)\dfgl\Gamma(U,\S_B(\bullet(\alpha)))\colon\grmod^B(S_B)\rightarrow\catmod(\o_Y(U)).$$ So, for $U\in\U_Y$ we get a functor $$\Gamma_{**,B,U}\dfgl\bigoplus_{\alpha\in B}\Gamma_{\alpha,B,U}\colon\grmod^B(S_B)\rightarrow\catmod(\o_Y(U)).$$ If $F$ is a $B$-graded $S_B$-module and $U,V\in\U_Y$ with $V\subseteq U$ then the restriction morphisms\linebreak $\S_B(F(\alpha))(U)\rightarrow\S_B(F(\alpha))(V)$ for $\alpha\in B$ induce a morphism $\Gamma_{**,B,U}(F)\rightarrow\Gamma_{**,B,V}(F)$ of $\o_Y(U)$-modules. This defines a presheaf of $\o_Y$-modules $$\Gamma_{**,B}(F)\colon(V\subseteq U)\mapsto(\Gamma_{**,B,U}(F)\rightarrow\Gamma_{**,B,V}(F)).$$ Finally, we get a functor $\Gamma_{**,B}\colon\grmod^B(S_B)\rightarrow\prmod(\o_Y)$.
\end{no}

\begin{no}\label{3.700}
Let $F$ be a $B$-graded $S_B$-module. For $\alpha,\beta\in B$ we consider the morphism of $\o_Y$-modules $$\delta(S_B(\alpha),F(\beta))\colon\S_B(S_B(\alpha))\otimes_{\o_Y}\S_B(F(\beta))\rightarrow\S_B(F(\alpha+\beta))$$ from \ref{3.280}. By taking sections over $U\in\U_Y$ we get a structure of $B$-graded $\Gamma_{*,B,U}(\o_Y)$-module on $\Gamma_{**,B,U}(F)$ such that for $V\in\U_U$ the restriction morphism $\Gamma_{**,B,U}(F)\rightarrow\Gamma_{**,B,V}(F)$ is a morphism of $B$-graded $\Gamma_{*,B,U}(\o_Y)$-modules. Thus, we get a presheaf of $B$-graded $\Gamma_{*,B}(\o_Y)$-modules $$\Gamma_{**,B}(F)\colon(V\subseteq U)\mapsto(\Gamma_{**,B,U}(F)\rightarrow\Gamma_{**,B,V}(F)).$$

It is readily checked that by the above we can consider $\Gamma_{**,B}$ as a functor taking values in\linebreak $\pgrmod^B(\Gamma_{*,B}(\o_Y))$. Hence, for $U\in\U_Y$ we get a functor $$\Gamma_{**,B,U}\colon\grmod^B(S_B)\rightarrow\grmod^B(\Gamma_{**,B,U}(\o_Y)),$$ and the diagram of categories $$\xymatrix@C60pt@R15pt{\grmod^A(S_A)\ar[r]^(.43){\Gamma_{**,A,U}}\ar[d]_{\bullet_{(B)}}&\grmod^A(\Gamma_{*,A,U}(\o_Y))\ar[d]^{\bullet_{(B)}}\\\grmod^B(S_B)\ar[r]^(.43){\Gamma_{**,B,U}}&\grmod^B(\Gamma_{*,B,U}(\o_Y))}$$ commutes. In particular, by scalar restriction to $S_B$ we get a functor $$\Gamma_{**,B,Y}\colon\grmod^B(S_B)\rightarrow\grmod^B(S_B),$$ called \textit{the $B$-restricted second total functor of sections over $R$ associated with $\sig$.}
\end{no}

\begin{no}\label{3.710}
Let $F$ be a $B$-graded $S_B$-module and let $\alpha\in B$. If $\tau\fleq\sigma\in\sig$ then taking components of degree $\alpha$ of the canonical morphisms $\eta_{\sigma}(F)\colon F\rightarrow F_{\sigma}$ and $\eta_{\tau}(F)\colon F\rightarrow F_{\tau}$ yields a commutative diagram of $S_0$-modules $$\xymatrix@R15pt@C60pt{F_{\alpha}\ar[r]^(.49){\eta_{\sigma}(F)_{\alpha}}\ar[rd]_(.45){\eta_{\tau}(F)_{\alpha}}&(F_{\sigma})_{\alpha}\ar[d]^{\eta^{\sigma}_{\tau}(F)_{\alpha}}\\&(F_{\tau})_{\alpha}.}$$ As $(F_{\sigma})_{\alpha}=F(\alpha)_{(\sigma)}=\Gamma_{\alpha,B,Y_{\sigma}}(F)$ for $\sigma\in\sig$ there is a unique morphism of $S_0$-modules $\etaq_{\alpha}(F)\colon$\linebreak$F_{\alpha}\rightarrow\Gamma_{\alpha,B,Y}(F)$ such that for $\sigma\in\sig$ the diagram $$\xymatrix@R15pt@C60pt{F_{\alpha}\ar[r]^(.43){\etaq_{\alpha}(F)}\ar[rd]_(.45){\eta_{\sigma}(F)_{\alpha}}&\Gamma_{\alpha,B,Y}(F)\ar[d]^{\res_{Y_{\sigma}}}\\&\Gamma_{\alpha,B,Y_{\sigma}}(F)}$$ commutes. Thus, we get a morphism of $B$-graded $S_B$-modules $$\etaq_B(F)\dfgl\bigoplus_{\alpha\in B}\etaq_{\alpha}(F)\colon F\rightarrow\Gamma_{**,B,Y}(F).$$ Varying $F$ gives rise to a morphism of functors $$\xymatrix@C40pt@R0pt{&\ar@<1ex>@{=>}[dd]^{\etaq_B}&\\\grmod^B(S_B)\ar@/^3ex/[rr]^{\Id_{\grmod^B(S_B)}}\ar@/_3ex/[rr]_{\Gamma_{**,B,Y}}&&\grmod^B(S_B).\\&&}$$
\end{no}

\begin{no}\label{3.720}
For $\alpha\in B$ we consider the morphism of functors $$\delta(S_B(\alpha),\bullet)\colon \S_B(\bullet)(\alpha)\rightarrow\S_B(\bullet(\alpha))$$ from \ref{3.280}. Taking sections over $U\in\U_Y$ and direct sums over $\alpha\in B$ yields a morphism of functors $$\delta_{B,U}\colon\Gamma_{*,B,U}\circ\S_B\rightarrow\Gamma_{**,B,U}$$ such that for $V\in\U_U$ the diagram $$\xymatrix@R15pt@C40pt{\Gamma_{*,B,U}\circ\S_B\ar[r]^(.55){\delta_{B,U}}\ar[d]_{\res_V\circ\S_B}&\Gamma_{**,B,U}\ar[d]^{\res_V}\\\Gamma_{*,B,V}\circ\S_B\ar[r]^(.55){\delta_{B,V}}&\Gamma_{**,B,V}}$$ commutes, and that $\delta_{B,U}(S_B)=\Id_{\Gamma_{*,B,U}(\o_Y)}.$ In particular, we have a morphism of functors $$\xymatrix@C40pt@R0pt{&\ar@<1ex>@{=>}[dd]^{\delta_{B,Y}}&\\\grmod^B(S_B)\ar@/^3ex/[rr]^{\Gamma_{*,B,Y}\circ\S_B}\ar@/_3ex/[rr]_{\Gamma_{**,B,Y}}&&\grmod^B(S_B).\\&&}$$
\end{no}

\begin{nosm}\label{3.730}
The morphism of functors $\delta_{B,U}\colon\Gamma_{*,B,U}\circ\S_B\rightarrow\Gamma_{**,B,U}$ is not necessarily an isomorphism. A counterexample is obtained immediately from \ref{3.290}.
\end{nosm}

\begin{lemma}\label{3.800}
$\grgam{B}{I_B}\circ\ke(\etaq_B)=\ke(\etaq_B)$ and $\grgam{B}{I_B}\circ\cok(\etaq_B)=\cok(\etaq_B)$.
\end{lemma}

\begin{proof}
Let $F$ be a $B$-graded $S_B$-module and let $\alpha\in B$. First, let $x\in\ke(\etaq_B(F))_{\alpha}$. If $\sigma\in\sig$ then $0=\etaq_B(F)(x)\res_{Y_{\sigma}}=\frac{x}{1}\in(F_{\sigma})_{\alpha}$, hence there exists $m_{\sigma}\in\N$ with $\zs^{m_{\sigma}}x=0$. Therefore, there exists $m\in\N$ with $I_B^mx=0$, implying the first claim.

Next, let $x\in\Gamma_{**,B,Y}(F)_{\alpha}=\S_B(F(\alpha))(Y)$. There exists $m\in\N$ such that for $\sigma\in\sig$ we have $m\alp_{\sigma}\in B$ and there exists $x_{\sigma}\in F_{\alpha+m\alp_{\sigma}}$ with $x\res_{Y_{\sigma}}=\frac{x_{\sigma}}{\zs^m}\in\S_B(F(\alpha))(Y_{\sigma})=F_{\sigma,\alpha}$. Let $\sigma\in\sig$. For $\tau\in\sig$ we have $$\textstyle\frac{\zs^mx_{\tau}}{\zt^m}=\zs^mx\res_{Y_{\tau}}\res_{Y_{\sigma\cap\tau}}=\zs^mx\res_{Y_{\sigma}}\res_{Y_{\sigma\cap\tau}}=\frac{\zs^mx_{\sigma}}{\zs^m}=\frac{x_{\sigma}}{1}\in F_{\sigma\cap\tau},$$ so that the canonical image of $\frac{x_{\sigma}}{1}-\frac{\zs^mx_{\tau}}{\zt^m}\in F_{\tau}$ in $F_{\sigma\cap\tau}$ is zero. Therefore, there exists $r\in\N$ such that for $\tau\in\sig$ we have $$\textstyle(\prod_{\rho\in\tau_1\setminus\sigma_1}Z_{\rho}^r)\frac{x_{\sigma}}{1}=(\prod_{\rho\in\tau_1\setminus\sigma_1}Z_{\rho}^r)\frac{\zs^mx_{\tau}}{\zt^m}\in F_{\tau},$$ hence $\frac{\zs^rx_{\sigma}}{1}=\frac{\zs^{r+m}x_{\tau}}{\zt^m}\in F_{\tau}$. We set $y\dfgl\zs^rx_{\sigma}\in F_{\alpha+(r+m)\alp_{\sigma}}$. If $\tau\in\sig$ then $$\textstyle\etaq_B(F)(y)\res_{Y_{\tau}}=\frac{y}{1}=\frac{\zs^rx_{\sigma}}{1}=\frac{\zs^{r+m}x_{\tau}}{\zt^m}=\zs^{r+m}x\res_{Y_{\tau}}\in F_{\tau,\alpha+(r+m)\alp_{\sigma}},$$ so that $\zs^{r+m}x=\etaq_B(F)(y)\in\im(\etaq_B(F))$. This shows that there exists $l\in\N$ with $I_B^lx\in\im(\etaq_B(F))$. Hence, $\cok(\etaq_B(F))$ is an $I_B$-torsion module.
\end{proof}

\begin{prop}\label{3.820}
The morphism of functors $\S_B\circ\etaq_B\colon\S_B\rightarrow\S_B\circ\Gamma_{**,B,Y}$ is an isomorphism, and the diagram $$\xymatrix@R15pt@C60pt{\S_B\circ\Gamma_{*,B,Y}\circ\S_B\ar[rr]^{\S_B\circ\delta_{B,Y}}\ar[rd]_{\beta_B\circ\S_B}&&\S_B\circ\Gamma_{**,B,Y}\\&\S_B\ar[ru]_{\S_B\circ\etaq_B}&}$$ in $\hmf{\grmod^B(S_B)}{\qcmod(\o_Y)}$ commutes.
\end{prop}

\begin{proof}
Let $F$ be a $B$-graded $S_B$-module. Since $\ke(\etaq_B(F))$ and $\cok(\etaq_B(F))$ are $I_B$-torsion modules (\ref{3.800}), exactness of $\S_B$ (\ref{3.210}) and \ref{3.270} imply that $\S_B(\etaq_B(F))$ is an isomorphism and thus the first claim.

Let $\sigma\in\sig$. It remains to show that the diagram $$\xymatrix@R15pt@C60pt{\Gamma_{*,B,Y}(\S_B(F))_{(\sigma)}\ar[rd]_{\beta_{B,\sigma}(\S_B(F))\quad}\ar[rr]^{\delta_{B,Y}(F)_{(\sigma)}}&&\Gamma_{**,B,Y}(F)_{(\sigma)}\\&F_{(\sigma)}\ar[ru]_{\etaq_B(F)_{(\sigma)}}^(.4){\cong}&}$$ commutes. We know that $\Gamma_{*,B,Y_{\sigma}}(\S_B(F))$ and $\Gamma_{**,B,Y_{\sigma}}(F)$ are $S_{B,\sigma}$-modules (\ref{3.640}), so the canonical morphisms $\Gamma_{*,B,Y_{\sigma}}(\S_B(F))\rightarrow\Gamma_{*,B,Y_{\sigma}}(\S_B(F))_{\sigma}$ and $\Gamma_{**,B,Y_{\sigma}}(F)\rightarrow\Gamma_{**,B,Y_{\sigma}}(F)_{\sigma}$ are isomorphisms, and thus so are their components of degree $0$; we denote them by $\zeta$ and $\zeta'$, respectively. Furthermore, $\Gamma_{*,B,Y_{\sigma}}(\S_B(F))_0=F_{(\sigma)}=\Gamma_{**,B,Y_{\sigma}}(F)_0$ (\ref{3.660}, \ref{3.690}). So, we have a diagram of $S_{(\sigma)}$-modules $$\xymatrix@R15pt@C15pt{\Gamma_{*,B,Y_{\sigma}}(\S_B(F))_{(\sigma)}\ar[rrrr]^{\delta_{B,Y_{\sigma}}(F)_{(\sigma)}}&&&&\Gamma_{**,B,Y_{\sigma}}(F)_{(\sigma)}\\&\Gamma_{*,B,Y}(\S_B(F))_{(\sigma)}\ar[lu]^{(\cdot\,\res_{Y_{\sigma}})_{(\sigma)}}\ar[rr]^{\delta_{B,Y}(F)_{(\sigma)}}\ar[rd]_{\beta_{B,\sigma}(\S_B(F))\quad}&&\Gamma_{**,B,Y}(F)_{(\sigma)}\ar[ru]_{(\cdot\,\res_{Y_{\sigma}})_{(\sigma)}}&\\\Gamma_{*,B,Y_{\sigma}}(\S_B(F))_0\ar@{=}[rr]\ar[uu]^{\zeta}_{\cong}&&F_{(\sigma)}\ar@{=}[rr]\ar[ru]^(.4){\cong}_{\etaq_B(F)_{(\sigma)}}&&\Gamma_{**,B,Y_{\sigma}}(F)_0.\ar[uu]^{\cong}_{\zeta'}}$$ Its upper quadrangle commutes by \ref{3.720}. Its left and right quadrangles commute by the definitions of $\beta_B$ and $\etaq_B$, respectively (\ref{3.680}, \ref{3.710}). By the definition of $\delta_{B,Y_{\sigma}}$ (\ref{3.720}) and what was shown above, $\delta_{B,Y_{\sigma}}(F)_0\colon\Gamma_{*,B,Y_{\sigma}}(\S_B(F))_0\rightarrow\Gamma_{**,B,Y_{\sigma}}(F)_0$ equals $\Id_{F_{(\sigma)}}$, hence the outer quadrangle commutes, too. Therefore, the diagram commutes, and thus the claim is proven.
\end{proof}


\subsection{Correspondence between graded modules and quasicoherent sheaves}\label{sub4.4}

\begin{lemma}\label{3.910}
For $\sigma,\tau\in\sig$ there exist $p\in S_{(\tau)}$, $l\in\N$ and $q\in S_{\tau,l\alp_{\sigma}}$ with $Y_{\tau\cap\sigma}=(Y_{\tau})_p$ and $qp=\zs^l\in S_{\tau,l\alp_{\sigma}}$.
\end{lemma}

\begin{proof}
As $\sigma\cap\tau\fleq\tau$ there exists $u\in\tau^{\vee}_M$ with $\sigma\cap\tau=\tau\cap u^{\perp}$. We have $p\dfgl\prod_{\rho\in\sig_1}Z_{\rho}^{\rho_N(u)}\in S_{(\tau)}$. The canonical open immersion $Y_{\sigma\cap\tau}\hookrightarrow Y_{\tau}$ being the spectrum of $\eta_p\colon S_{(\tau)}\rightarrow(S_{(\tau)})_p$ (\ref{3.110}, \ref{3.130}), we have $Y_{\sigma\cap\tau}=(Y_{\tau})_p$. Setting $l\dfgl\max\{1,\sup\{\rho_N(u)\mid\rho\in\tau_1\setminus\sigma_1\}\}\in\N$ and $$q\dfgl(\prod_{\rho\in\sig_1\setminus(\tau_1\cup\sigma_1)}Z_{\rho}^l)(\prod_{\rho\in\tau_1\setminus\sigma_1}Z_{\rho}^{l-\rho_N(u)})(\prod_{\rho\in\sig_1\setminus\tau_1}Z_{\rho}^{-\rho_N(u)})\in S_{A,\tau}$$ it follows $qp=\zs^l\in S_{\tau,l\alp_{\sigma}}$ and thus $q\in S_{\tau,l\alp_{\sigma}}$.
\end{proof}

\begin{prop}\label{3.920}
The morphism of functors $\beta_B\colon\S_B\circ\Gamma_{*,B,Y}\rightarrow\Id_{\qcmod(\o_Y)}$ is an isomorphism.
\end{prop}

\begin{proof}
Let $\F$ be a quasicoherent $\o_Y$-module and let $\sigma\in\sig$. It suffices to show that the map $$\beta_{B,\sigma}(\F)\colon\Gamma_{*,B,Y}(\F)_{(\sigma)}\rightarrow\F(Y_{\sigma})$$ is bijective.

\textit{Surjectivity:} Let $f\in\F(Y_{\sigma})$. We construct a preimage of $f$ under $\beta_{B,\sigma}(\F)$. For $\tau\in\sig$ there are $p_{\tau}\in S_{(\tau)}$, $j_{\tau}\in\N$ with $j_{\tau}\alp_{\sigma}\in B$ and $q_{\tau}\in S_{\tau,j_{\tau}\alp_{\sigma}}$ with $Y_{\tau\cap\sigma}=(Y_{\tau})_{p_{\tau}}$ and $q_{\tau}p_{\tau}=\zs^{j_{\tau}}\in S_{\tau,j_{\tau}\alp_{\sigma}}$ (\ref{3.910}). So, we have an affine scheme $Y_{\tau}=\spec(S_{(\tau)})$, a quasicoherent $\o_{Y_{\tau}}$-module $\F\res_{Y_{\tau}}$, and sections $p_{\tau}\in S_{(\tau)}=\o_{Y_{\tau}}(Y_{\tau})$ and $f\res_{Y_{\tau\cap\sigma}}\in\F(Y_{\tau\cap\sigma})=\F((Y_{\tau})_{p_{\tau}})$. Thus, by \cite[I.1.4.1]{ega} there exist $k_{\tau}\in\N$ with $k_{\tau}\alp_{\sigma}\in B$ and $g''_{\tau}\in\F(Y_{\tau})$ with $g''_{\tau}\res_{Y_{\tau\cap\sigma}}=p_{\tau}^{k_{\tau}}\res_{Y_{\tau\cap\sigma}}f\res_{Y_{\tau\cap\sigma}}\in\F(Y_{\tau\cap\sigma})$. Let $k\dfgl\max\{k_{\tau}\mid\tau\in\sig\}\in\N$ and $j\dfgl\max\{j_{\tau}\mid\tau\in\sig\}\in\N$. If $\tau\in\sig$ then $$g'_{\tau}\dfgl\zs^{k(j-j_{\tau})}q_{\tau}^kp_{\tau}^{k-k_{\tau}}g''_{\tau}\in\F(kj\alp_{\sigma})(Y_{\tau})$$ and $$g'_{\tau}\res_{Y_{\tau\cap\sigma}}=(\zs^{k(j-j_{\tau})}q_{\tau}^kp_{\tau}^{k-k_{\tau}})\res_{Y_{\tau\cap\sigma}}p_{\tau}^{k_{\tau}}\res_{Y_{\tau\cap\sigma}}f\res_{Y_{\tau\cap\sigma}}=\zs^{kj}f\res_{Y_{\tau\cap\sigma}}\in\F(kj\alp_{\sigma})(Y_{\tau\cap\sigma}).$$ If $\tau,\tau'\in\sig$ then setting $\omega\dfgl\tau\cap\tau'$ and $h_{\tau\tau'}\dfgl g_{\tau}'\res_{Y_{\omega}}-g'_{\tau'}\res_{Y_{\omega}}\in\F(kj\alp_{\sigma})(Y_{\omega})$ there are $r_{\tau\tau'}\in S_{(\omega)}$, $l_{\tau\tau'}\in\N$ with $l_{\tau\tau'}\alp_{\sigma}\in B$ and $s_{\tau\tau'}\in S_{\omega,l_{\tau\tau'}\alp_{\sigma}}$ with $Y_{\omega\cap\sigma}=(Y_{\omega})_{r_{\tau\tau'}}$ and $s_{\tau\tau'}r_{\tau\tau'}=\zs^{l_{\tau\tau'}}$ (\ref{3.910}). So, we have an affine scheme $Y_{\omega}=\spec(S_{(\omega)})$, a quasicoherent $\o_{Y_{\omega}}$-module $\F(kj\alp_{\sigma})\res_{Y_{\omega}}$, and sections $r_{\tau\tau'}\in S_{(\omega)}=\o_{Y_{\omega}}(Y_{\omega})$ and $h_{\tau\tau'}\in\F(kj\alp_{\sigma})(Y_{\omega})$ with $$h_{\tau\tau'}\res_{(Y_{\omega})_{r_{\tau\tau'}}}=g'_{\tau}\res_{Y_{\tau\cap\sigma}}\res_{Y_{\omega\cap\sigma}}-g'_{\tau'}\res_{Y_{\tau'\cap\sigma}}\res_{Y_{\omega\cap\sigma}}=\zs^{kj}f\res_{Y_{\tau\cap\sigma\cap\tau'}}-\zs^{kj}f\res_{Y_{\tau'\cap\sigma\cap\tau}}=0.$$ Thus, by \cite[I.1.4.1]{ega} there exists $m_{\tau\tau'}\in\N$ with $m_{\tau\tau'}\alp_{\sigma}\in B$ and  $r_{\tau\tau'}^{m_{\tau\tau'}}h_{\tau\tau'}=0\in\F(kj\alp_{\sigma})(Y_{\omega})$. Let $l\dfgl\max\{l_{\tau\tau'}m_{\tau\tau'}\mid(\tau,\tau')\in\sig^2\}\in\N$. If $\tau\in\sig$ then $g_{\tau}\dfgl\zs^lg'_{\tau}\in\F((kj+l)\alp_{\sigma})(Y_{\tau})$. If $\tau,\tau'\in\sig$ then setting $\omega\dfgl\tau\cap\tau'$ it follows $$g_{\tau}\res_{Y_{\omega}}-g_{\tau'}\res_{Y_{\omega}}=\zs^{l-l_{\tau\tau'}m_{\tau\tau'}}\zs^{l_{\tau\tau'}m_{\tau\tau'}}h_{\tau\tau'}=$$$$\zs^{l-l_{\tau\tau'}m_{\tau\tau'}}s_{\tau\tau'}^{m_{\tau\tau'}}r_{\tau\tau'}^{m_{\tau\tau'}}h_{\tau\tau'}=0\in\F((kj+l)\alp_{\sigma})(Y_{\omega}).$$ Thus, there exists $g\in\F((kj+l)\alp_{\sigma})(Y)$ with $g\res_{Y_{\tau}}=g_{\tau}$ for every $\tau\in\sig$, and $\frac{g}{\zs^{kj+l}}\in\Gamma_{*,B,Y}(\F)_{(\sigma)}$. Finally, for $\tau\in\sig$ we have $$\textstyle\frac{g\res_{Y_{\sigma}}}{\zs^{kj+l}}\res_{Y_{\tau\cap\sigma}}=\frac{g\res_{Y_{\tau\cap\sigma}}}{\zs^{kj+l}}=\frac{g_{\tau}\res_{Y_{\tau\cap\sigma}}}{\zs^{kj+l}}=\frac{\zs^lg'_{\tau}\res_{Y_{\tau\cap\sigma}}}{\zs^{kj+l}}=\frac{\zs^{kj+l}f\res_{Y_{\tau\cap\sigma}}}{\zs^{kj+l}}=f\res_{Y_{\tau\cap\sigma}},$$ hence $\beta_{B,\sigma}(\F)(\frac{g}{\zs^{kj+l}})=\frac{g\res_{Y_{\sigma}}}{\zs^{kj+l}}=f$ as desired.

\textit{Injectivity:} Let $f\in\Gamma_{*,B,Y}(\F)_{(\sigma)}$ with $\beta_{B,\sigma}(\F)(f)=0$. There are $k\in\N$ and $g\in\F(k\alp_{\sigma})(Y)$ with $f=\frac{g}{\zs^k}$, and $\frac{g\res_{Y_{\sigma}}}{\zs^k}=0\in\F(Y_{\sigma})$, hence $g\res_{Y_{\sigma}}=\zs^k\frac{g\res_{Y_{\sigma}}}{\zs^k}=0\in\F(k\alp_{\sigma})(Y_{\sigma})$. To show that $f=0$ we have to show the existence of an $l\in\N$ with $\zs^lg=0\in\F((l+k)\alp_{\sigma})(Y)$. For $\tau\in\sig$ there are $p_{\tau}\in S_{(\tau)}$, $m_{\tau}\in\N$ and $q_{\tau}\in(S_{\tau})_{m_{\tau}\alp_{\sigma}}$ with $Y_{\tau\cap\sigma}=(Y_{\tau})_{p_{\tau}}$ and $q_{\tau}p_{\tau}=\zs^{m_{\tau}}\in S_{\tau,m_{\tau}\alp_{\sigma}}$ (\ref{3.910}). So, we have an affine scheme $Y_{\tau}=\spec(S_{(\tau)})$, a quasicoherent $\o_{Y_{\tau}}$-module $\F(k\alp_{\sigma})\res_{Y_{\tau}}$, and sections $p_{\tau}\in S_{(\tau)}=\o_{Y_{\tau}}(Y_{\tau})$ and $g\res_{Y_{\tau}}\in\F(k\alp_{\sigma})(Y_{\tau})$ with $$g\res_{Y_{\tau}}\res_{(Y_{\tau})_{p_{\tau}}}=g\res_{Y_{\sigma}}\res_{Y_{\tau\cap\sigma}}=0\in\F(k\alp_{\sigma})((Y_{\tau})_{p_{\tau}}).$$  Thus, by \cite[I.1.4.1]{ega} there exists $l_{\tau}\in\N$ with $p_{\tau}^{l_{\tau}}(g\res_{Y_{\tau}})=0\in\F(k\alp_{\sigma})(Y_{\tau})$, hence $$(\zs^{l_{\tau}m_{\tau}}g)\res_{Y_{\tau}}=q_{\tau}^{l_{\tau}}p_{\tau}^{l_{\tau}}(g\res_{Y_{\tau}})=0\in\F((l_{\tau}m_{\tau}+k)\alp_{\sigma})(Y_{\tau}).$$ Setting $l\dfgl\max\{l_{\tau}m_{\tau}\mid\tau\in\sig\}\in\N$ it follows $(\zs^lg)\res_{Y_{\tau}}=0\in\F((l+k)\alp_{\sigma})(Y_{\tau})$ for every $\tau\in\sig$ and thus $\zs^lg=0\in\F((l+k)\alp_{\sigma})(Y)$ as desired.
\end{proof}

\begin{thm}\label{3.930}
The functor $\S_B\colon\grmod^B(S_B)\rightarrow\qcmod(\o_Y)$ is essentially surjective.
\end{thm}

\begin{proof}
Immediately from \ref{3.920}.
\end{proof}

\begin{cor}\label{3.940}
If $F$ is a $B$-graded $S_B$-module then the map $\Xi_F\colon\J_F\rightarrow\widetilde{\J}_F$ is surjective.
\end{cor}

\begin{proof}
Let $\G\in\widetilde{\J}_F$ and let $j\colon\G\hookrightarrow\S_B(F)$ denote its canonical injection. We define a graded sub-$S_B$-module $G\dfgl\etaq_B(F)^{-1}(\im(\delta_{B,Y}(F)\circ\Gamma_{*,B,Y}(j)))\subseteq F$ with canonical injection $i\colon G\hookrightarrow F$, getting a commutative diagram of $B$-graded $S_B$-modules $$\xymatrix@R15pt@C40pt{F\ar[r]^(.4){\etaq_B(F)}&\Gamma_{**,B,Y}(F)&\Gamma_{*,B,Y}(\S_B(F))\ar[l]_(.52){\delta_{B,Y}(F)}\\G\ar@{ (->}[u]^i\ar[r]^(.27){\xi}&\im\bigl(\delta_{B,Y}(F)\circ\Gamma_{*,B,Y}(j)\bigr)\ar@{ (->}[u]&\Gamma_{*,B,Y}(\G)\ar[l]\ar[u]_{\Gamma_{*,B,Y}(j)}}$$ whose left quadrangle is cartesian. Applying the exact functor $\S_B$ (\ref{3.210}) yields by \ref{3.820} a commutative diagram of $\o_Y$-modules $$\xymatrix@R15pt@C40pt{\S_B(F)\ar[r]_(.43){\S_B(\etaq_B(F))}&\S_B(\Gamma_{**,B,Y}(F))&\S_B(\Gamma_{*,B,Y}(\S_B(F)))\ar[l]^{\S_B(\delta_{B,Y}(F))}\ar@/_5ex/[ll]^{\beta_B(\S_B(F))}\\\Xi_F(G)\ar@{ (->}[u]\ar[r]^(.27){\S_B(\xi)}&\im\bigl(\S_B(\delta_{B,Y}(F))\circ\S_B(\Gamma_{*,B,Y}(j))\bigr)\ar@{ (->}[u]&\S_B(\Gamma_{*,B,Y}(\G))\ar[l]\ar[u]_{\S_B(\Gamma_{*,B,Y}(j))}}$$ whose left quadrangle is cartesian. As $\beta_B(\S_B(F))$ and $\S_B(\etaq_B(F))$ are isomorphisms (\ref{3.920}, \ref{3.820}), $\S_B(\delta_{B,Y}(F))$ is an isomorphism, too. Therefore, the above diagram shows that $\Xi_F(G)\hookrightarrow\S_B(F)$ is the image under $\beta_B(\S_B(F))$ of the image of $$\S_B(\Gamma_{*,B,Y}(j))\colon\S_B(\Gamma_{*,B,Y}(\G))\rightarrow\S_B(\Gamma_{*,B,Y}(\S_B(F))).$$ As $\beta_B$ is an isomorphism (\ref{3.920}) this equals $i\colon\G\hookrightarrow\S_B(F)$, and the claim is proven.
\end{proof}

\begin{cor}\label{3.950}
If $F$ is a $B$-graded $S_B$-module and $\G\in\widetilde{\J}_F$ is of finite type then there exists $G\in\J_F$ of finite type with $\Xi_F(G)=\G$.
\end{cor}

\begin{proof}
By \ref{3.940} we can without loss of generality suppose $\G=\S_B(F)$. Let $L\subseteq F^{\hom}$ be a generating set of $F$ and let $\mathbbm{L}$ denote the right filtering ordered set of finite subsets of $L$. For $H\in\mathbbm{L}$ we denote by $F_H$ the graded sub-$S_B$-module of $F$ generated by $H$ and by $i_H\colon\S_B(F_H)\rightarrow\S_B(F)$ the induced monomorphism in $\qcmod(\o_Y)$. We have $F=\ilim_{H\in\mathbbm{L}}F_H$ (\ref{2.67}), and thus $\S_B(F)=\ilim_{H\in\mathbbm{L}}\S_B(F_H)$ (\ref{3.210}). Now, $\mathbbm{L}$ is right filtering, $Y$ is quasicompact, and $\S_B(F)$ is of finite type, so that by \cite[0.5.2.3]{ega} there exists $H\in\mathbbm{L}$ such that $i_H$ is an epimorphism and hence an isomorphism. It follows $\S_B(F)=\S_B(F_H)$ and thus the claim.
\end{proof}

\begin{cor}\label{3.990}
Suppose that $B$ is small and let $F$ be a $B$-graded $S_B$-module. If $F$ is (pseudo-)\linebreak coherent then so is $\S_B(F)$.\footnote{If $X$ is a scheme then an $\o_X$-module $\F$ is called \textit{pseudocoherent} if every sub-$\o_X$-module of $\F$ of finite type is of finite presentation; hence, an $\o_X$-module is coherent (in the sense of \cite[0.5.3.1]{ega}) if and only if it is pseudocoherent and of finite type.}
\end{cor}

\begin{proof}
Let $F$ be pseudocoherent and let $\G\in\widetilde{\J}_F$ be of finite type. There exists $G\in\J_F$ of finite type with $\G=\S_B(G)$ (\ref{3.950}). Then, $G$ is of finite presentation by hypothesis, hence $\G$ is of finite presentation (\ref{3.340}), and thus $\S_B(F)$ is pseudocoherent. The claim about coherence follows now from \ref{3.340}.
\end{proof}

\begin{cor}\label{3.1000}
If $\sig$ is $N$-regular and $R$ is stably coherent\/\footnote{i.e., polynomial algebras in finitely many indeterminates over $R$ are coherent (cf.~\cite[7.3]{glaz})} then $\o_Y$ is coherent.
\end{cor}

\begin{proof}
Stable coherence of $R$ implies coherence of $S_A$ (\ref{2.67}, \ref{3.20}) and $N$-regularity of $\sig$ implies $\pic(\sig)=A$ (\ref{1.280}~a)). As $\o_Y=\S_A(S_A)$ (\ref{3.210}) we get the claim from \ref{3.990}.
\end{proof}

\begin{prop}\label{3.960}
Let $F$ be a $B$-graded $S_B$-module and let $G,H\in\J_F$.\! If\/ $\Sat_F(G,I_B)=\Sat_F(H,I_B)$ then $\Xi_F(G)=\Xi_F(H)$; if $B$ is small then the converse holds.
\end{prop}

\begin{proof}
First, suppose that $\Sat_F(G,I_B)=\Sat_F(H,I_B)$, and let $\sigma\in\sig$ and $u\in G_{(\sigma)}$. There are $k\in\N$ with $k\alp_{\sigma}\in B$ and $f\in G_{k\alp_{\sigma}}$ with $u=\frac{f}{\zs^k}$, hence $f\in G_{k\alp_{\sigma}}\subseteq\Sat_F(G,I_B)_{k\alp_{\sigma}}=\Sat_F(H,I_B)_{k\alp_{\sigma}}$. Therefore, there is an $l\in\N$ with $l\alp_{\sigma}\in B$ and $\zs^lf\in H$, implying $u=\frac{f}{\zs^k}=\frac{\zs^lf}{\zs^{k+l}}\in H_{(\sigma)}$. This shows $G_{(\sigma)}\subseteq H_{(\sigma)}$. By reasons of symmetry we get $G_{(\sigma)}=H_{(\sigma)}$. By definition of $\S_B$ it follows $\Xi_F(G)=\Xi_F(H)$.

Next, suppose that $B$ is small and $\Xi_F(G)=\Xi_F(H)$, hence $G_{(\sigma)}=H_{(\sigma)}$ for $\sigma\in\sig$. Let $\alpha\in B$ and $f\in G_{\alpha}$. There is a $k\in\N$ such that for $\sigma\in\sig$ we have $k\alp_{\sigma}\in B$, hence $k\alp_{\sigma}+\alpha\in B$. For $\sigma\in\sig$ there is a family $(r_{\rho})_{\rho\in\sig_1\setminus\sigma_1}$ in $\Z$ with $k\alp_{\sigma}+\alpha=\sum_{\rho\in\sig_1\setminus\sigma_1}r_{\rho}\alpha_{\rho}$, and we have $g_{\sigma}\dfgl\frac{\zs^kf}{\prod_{\rho\in\sig_1\setminus\sigma_1}Z_{\rho}^{r_{\rho}}}\in G_{(\sigma)}=H_{(\sigma)}$, hence there are $l\in\N$ with $l\alp_{\sigma}\in B$ and $h_{\sigma}\in H_{l\alp_{\sigma}}$ with $g_{\sigma}=\frac{h_{\sigma}}{\zs^l}$. For $\sigma\in\sig$ it follows $\zs^{l+k}f=(\prod_{\rho\in\sig_1\setminus\sigma_1}Z_{\rho}^{r_{\rho}})h_{\sigma}\in H$. Thus, there is an $m\in\N$ with $I_B^mf\subseteq H$, implying $f\in\Sat_F(H,I_B)$. This shows $\Sat_F(G,I_B)\subseteq\Sat_F(H,I_B)$ (\ref{2.80}, \ref{3.30}~a)), and then the claim follows by reasons of symmetry.
\end{proof}

\begin{thm}\label{3.970}
If $F$ is a $B$-graded $S_B$-module then the map $\Xi_F^{\sat}\colon\J^{\sat}_F\rightarrow\widetilde{\J}_F$ is surjective; if $B$ is small then it is bijective.
\end{thm}

\begin{proof}
For $\G\in\widetilde{\J}_F$ there exists $G\in\J_F$ with $\Xi_F(G)=\G$ (\ref{3.940}). As $\Sat_F(G,I_B)\in\J^{\sat}_F$ and $\Xi_F(\Sat_F(G,I_B))=\Xi_F(G)$ (\ref{2.80}, \ref{3.30}~a), \ref{3.960}) this proves the first claim. If $B$ is small then for $G,H\in\J^{\sat}_F$ with $\Xi_F(G)=\Xi_F(H)$ it follows $G=\Sat_F(G,I_B)=\Sat_F(H,I_B)=H$ (\ref{3.960}), and hence $\Xi^{\sat}_F$ is injective.
\end{proof}

\begin{cor}\label{3.980}
Suppose that $B$ is small and let $F$ be a $B$-graded $S_B$-module. Then, $F$ is an $I_B$-torsion module if and only if $\S_B(F)=0$.
\end{cor}

\begin{proof}
As $\Sat_F(F,I_B)=F$ and $\Sat_F(0,I_B)=\grgam{B}{I_B}(F)$ this follows by \ref{3.270} and \ref{3.970}.
\end{proof}

\begin{nosm}\label{3.1100}
The following fact about projective schemes is maybe known, but surprisingly seems to be neither well-known nor accessible in the literature: if $S$ is a positively $\Z$-graded ring that is generated as an $S_0$-algebra by finitely many elements of degree $1$, and $X=\proj(S)$, then the functor $\grmod^{\Z}(S)\rightarrow\qcmod(\o_X)$ that maps a $\Z$-graded $S$-module to its associated quasicoherent $\o_X$-module (\cite[II.2.5.3]{ega}) preserves the properties of being of finite type, of finite presentation, pseudocoherent, and coherent. This can be proven analogously to \ref{3.340} and \ref{3.990}, on use of \cite[II.5.7; II.2.7.8; II.2.7.11]{ega} instead of \ref{3.330}~b) and \ref{3.950}.
\end{nosm}


\subsection{The toric Serre-Grothendieck correspondence}\label{sub4.5}

\begin{no}\label{4.510}
Let $U\in\U_Y$. If $\alpha\in B$ then the functor $$\Gamma_{\alpha,B,U}\colon\grmod^B(S_B)\rightarrow\catmod(S_0)$$ is left exact (\ref{3.690}). Keeping in mind \ref{2.10} and \ref{2.150} we denote its right derived cohomological functor by $(H_{\alpha,B,U}^i)_{i\in\Z}$, and we canonically identify $\Gamma_{\alpha,B,U}$ and $H_{\alpha,B,U}^0$. The functor $$\Gamma_{**,B,U}\colon\grmod^B(S_B)\rightarrow\grmod^B(S_B)$$ is left exact, too (\ref{3.690}). We denote its right derived cohomological functor by $(H_{**,B,U}^i)_{i\in\Z}$, and we canonically identify $\Gamma_{**,B,U}$ and $H_{**,B,U}^0$. Since $\grmod^B(S_B)$ fulfils Grothendieck's axiom AB4 (\ref{2.10}) we have $H_{**,B,U}^i=\bigoplus_{\alpha\in B}H_{\alpha,B,U}^i$ for every $i\in\Z$ (\cite[2.2.1]{tohoku}).

If $F$ is a $B$-graded $S_B$-module and $i\in\Z$ then exactness of $\S_B$ (\ref{3.210}) and \cite[2.2.1]{tohoku} imply that $H_{\alpha,B,U}^i(F)=H^i(U,\S_B(F(\alpha)))$ is the $i$-th sheaf cohomology of $\S_B(F(\alpha))$ over $U$ in the sense of \cite[0.12.1]{ega}. In particular, $H_{0,B,U}^i(F)=H^i(U,\S_B(F))$ is the $i$-th sheaf cohomology of $\S_B(F)$ over $U$.

If $i\in\Z$ then by the above and \ref{3.210} we have $$H^i_{\alpha,A,U}(\bullet)=H^i(U,\S_A(\bullet(\alpha)))=H^i(U,\S_B(\bullet(\alpha)_{(B)}))=H^i(U,\S_B(\bullet_{(B)}(\alpha)))=H^i_{\alpha,B,U}(\bullet_{(B)})$$ for $\alpha\in B$, and therefore $H_{**,A,U}^i(\bullet)_{(B)}=H_{**,B,U}^i(\bullet_{(B)})$.
\end{no}

\begin{lemma}\label{4.530}
$\grgam{B}{I_B}\circ\Gamma_{**,B,Y}=0$.
\end{lemma}

\begin{proof}
Let $F$ be a $B$-graded $S_B$-module, let $\alpha\in B$, let $x\in\Gamma_{**,B,Y}(F)_{\alpha}=\S_B(F(\alpha))(Y)$, and let $m\in\N$ with $I_B^mx=0$. There exists $l\in\N$ such that for $\sigma\in\sig$ we have $l\alp_{\sigma}\in B$ and there exists $x_{\sigma}\in F_{\alpha+l\alp_{\sigma}}$ with $x\res_{Y_{\sigma}}=\frac{x_{\sigma}}{\zs^l}\in(F_{\sigma})_{\alpha}$. It follows $x\res_{Y_{\sigma}}=\frac{\zs^mx_{\sigma}}{\zs^{m+l}}=\frac{1}{\zs^m}(\zs^mx)\res_{Y_{\sigma}}=0$ for $\sigma\in\sig$, thus $x=0$ as desired.
\end{proof}

\begin{prop}\label{4.540}
Suppose that $S_B$ has ITI\/\footnote{cf.~\ref{2.220}} with respect to $I_B$. There exists a unique morphism of functors $\etaq'_B\colon\Gamma_{**,B,Y}\rightarrow\gride{B}{I_B}$ such that the diagram $$\xymatrix@C40pt@R15pt{\Id_{\grmod^B(S_B)}\ar[r]^(.53){\etaq_B}\ar[rd]_{\eta_{I_B}}&\Gamma_{**,B,Y}\ar[d]^{\etaq'_B}\\&\gride{B}{I_B}}$$ commutes, and $\etaq'_B$ is an isomorphism.
\end{prop}

\begin{proof}
Existence, uniqueness and monomorphy of $\etaq'_B$ follow from \ref{d50}, \ref{3.800} and \ref{4.530}. It remains to show that it is an epimorphism. Let $F$ be a $B$-graded $S_B$-module and let $\alpha\in B$. We write $\etaq'$, $\etaq$ and $\eta$ instead of $\etaq'_B(F)$, $\etaq_B(F)$ and $\eta_{I_B}(F)$. It suffices to show that $\etaq'_{\alpha}\colon\S_B(F(\alpha))(Y)\rightarrow\gride{B}{I_B}(F)_{\alpha}$ is surjective. As $\bullet_{\alpha}$ commutes with inductive limits (\ref{2.10}) we have $$\gride{B}{I_B}(F)_{\alpha}=(\ilim_{m\in\N}\grhm{B}{S_B}{I_B^m}{F})_{\alpha}=\ilim_{m\in\N}(\grhm{B}{S_B}{I_B^m}{F(\alpha)}_0).$$ Thus, replacing $F$ by $F(-\alpha)$ we can suppose $\alpha=0$ and show that $$\etaq'_0\colon\S_B(F)(Y)\rightarrow\ilim_{m\in\N}\hm{\grmod^B(S_B)}{I_B^m}{F}$$ is surjective. For $l\in\N$ we denote by $$\iota_l\colon\hm{\grmod^B(S_B)}{I_B^l}{F}\rightarrow\ilim_{m\in\N}\hm{\grmod^B(S_B)}{I_B^m}{F}$$ the canonical morphism of $S_0$-modules, so that if $x\in F_0$ then $\iota_l$ maps the morphism $I_B^l\rightarrow F,\;a\mapsto ax$ onto $\eta(x)$.

Let $x\in\ilim_{m\in\N}\hm{\grmod^B(S_B)}{I_B^m}{F}$. There are $l\in\N$ and $h\in\hm{\grmod^B(S_B)}{I_B^l}{F}$ with $x=\iota_l(h)$. Let $y_{\sigma}\dfgl\frac{h(\zs^l)}{\zs^l}\in F_{(\sigma)}=\S_B(F)(Y_{\sigma})$ for $\sigma\in\sig$, so that $y_{\sigma}\res_{Y_{\tau}}=\frac{h(\zs^l)\zt^l}{\zs^l\zt^l}=\frac{h(\zt^l)\zs^l}{\zs^l\zt^l}=y_{\tau}$ for $\tau\fleq\sigma\in\sig$. Hence, there is a $y\in\S_B(F)(Y)$ with $y\res_{Y_{\sigma}}=y_{\sigma}$ for $\sigma\in\sig$. Let $r\in I_B^m$. For $\sigma\in\sig$ we have $$\textstyle(ry)\res_{Y_{\sigma}}=r(y\res_{Y_{\sigma}})=\frac{rh(\zs^l)}{\zs^l}=\frac{\zs^lh(r)}{\zs^l}=\frac{h(r)}{1}=\etaq(h(r))\res_{Y_{\sigma}},$$ implying $ry=\etaq(h(r))$, therefore $$r\etaq'(y)=\etaq'(ry)=\etaq'(\etaq(h(r)))=\eta(h(r))=\iota_l(rh)=r\iota_l(h)=rx,$$ and thus $r(\etaq'(y)-x)=0$. This shows $I_B^l(\etaq'(y)-x)=0$, thus $\etaq'(y)-x\in\grgam{B}{I_B}(\gride{B}{I_B}(F))=0$ (\ref{2.250}~b)), and therefore $\etaq'(y)=x$.
\end{proof}

\begin{thm}\label{4.550}
Suppose that $S_B$ has ITI with respect to $I_B$. There exist a morphism of functors $\zeta_B\colon\Gamma_{**,B,Y}\rightarrow\grloc{B}{1}{I_B}$ such that the sequence $$0\longrightarrow\grgam{B}{I_B}\xrightarrow{\xi_{I_B}}\Id_{\grmod^B(S_B)}\xrightarrow{\etaq_B}\Gamma_{**,B,Y}\xrightarrow{\zeta_B}\grloc{B}{1}{I_B}\longrightarrow 0$$ is exact, and a unique morphism of $\delta$-functors $$(\zeta^i_B)_{i\in\Z}\colon(H^i_{**,B,Y})_{i\in\Z}\longrightarrow(\grloc{B}{i+1}{I_B})_{i\in\Z}$$ with $\zeta_B^0=\zeta_B$, and $\zeta_B^i$ is an isomorphism for every $i>0$.
\end{thm}

\begin{proof}
Immediately from \ref{2.250}~a) and \ref{4.540}.
\end{proof}

\begin{prop}\label{4.560}
Suppose that $R$ is noetherian and $\sig$ is complete, and let $F$ be a $B$-graded $S_B$-module of finite type. Then, the $R$-modules $H^i_{**,B,Y}(F)_{\alpha}$ and $\grloc{B}{i}{I_B}(F)_{\alpha}$ are of finite type for $i\in\Z$ and $\alpha\in B$.
\end{prop}

\begin{proof}
Noetherianness of $R$ implies that $S_B$ is noetherian (\ref{3.45}) and in particular has ITI with respect to $I_B$ (\ref{2.220}). As $F$ is of finite type, $\S_B(F(\alpha))$ is coherent for $\alpha\in B$ (\ref{3.295}). Completeness of $\sig$ implies that $Y$ is proper over $R=S_0$ (\ref{3.40}, \ref{3.135}~b)). Thus, $H^i_{**,B,Y}(F)_{\alpha}=H^i_{\alpha,B,Y}(F)=H^i(Y,\S_B(F(\alpha)))$ is of finite type (\ref{4.510}, \cite[III.3.2.3]{ega}). The claim about local cohomology is obvious for $i<0$ and holds for $i=0$ since noetherianness of $F$ implies that $\grgam{B}{I_B}(F)_{\alpha}$ is noetherian for $\alpha\in B$ (\ref{2.120}~b)). As there are an epimorphism $\Gamma_{**,B,Y}(F)_{\alpha}\rightarrow\grloc{B}{1}{I_B}(F)_{\alpha}$ and for $i>1$ an isomorphism $H^{i-1}_{**,B,Y}(F)_{\alpha}\rightarrow\grloc{B}{i}{I_B}(F)_{\alpha}$ of $R$-modules (\ref{4.550}), the claim is proven.
\end{proof}


{\bf Acknowledgement.} This article grew out of parts of my PhD thesis, written under the guidance of Markus Brodmann; I still am very grateful to him.


\end{document}